\documentclass[namedate,webpdf,imanum]{ima-authoring-template}

\usepackage{amsfonts}
\usepackage{amssymb}
\usepackage{graphicx}
\usepackage{epstopdf}
\usepackage{verbatim}
\usepackage{mathtools}
\usepackage{algpseudocode}
\usepackage[nameinlink]{cleveref}
\usepackage{multirow}
\usepackage{tikz}
\usetikzlibrary{calc, shapes, angles, quotes, patterns, decorations.pathreplacing, cd}
\makeatletter
\renewcommand*\l@subsubsection[2]{{}{}{}} %
\makeatother
\usepackage{subcaption}
\usepackage{bm}
\usepackage{amsopn}
\usepackage{array}
\usepackage{doi}

\newif\ifarxiv
\arxivfalse
\arxivtrue

\theoremstyle{thmstyletwo}%
\newtheorem{theorem}{Theorem}[section]
\newtheorem{corollary}[theorem]{Corollary}%
\newtheorem{lemma}[theorem]{Lemma}%

\newtheorem{example}{Example}[section]%
\newtheorem{remark}{Remark}[section]%

\numberwithin{equation}{section}

\newcommand{\vecbd}[1]{ \boldsymbol{#1} }
\newcommand{\unitvec}[1]{\vecbd{#1}}
\newcommand{\vertex}[1]{\vecbd{{#1}}}
\newcommand{\tensorbd}[1]{\mathbf{{#1}}}
\newcommand{\symgrad}[1]{\boldsymbol{\epsilon}{(#1)}}

\newcommand{\newinf}{\mathop{\inf\vphantom{\sup}}}

\DeclareMathOperator{\grad}{\mathbf{grad}}
\DeclareMathOperator{\curl}{\mathbf{curl}}
\let\div\undefined
\DeclareMathOperator{\div}{div}
\DeclareMathOperator{\rot}{rot}

\renewcommand{\d}[1]{\,\mathrm{d}{#1}}
\newcommand{\dee}{\mathrm{d}}
\newcommand{\dual}[1]{{#1}'}
\newcommand{\discrete}[1]{\mathbb{#1}}
\newcommand{\operator}[1]{\mathcal{#1}}

\newcommand{\hcurl}{\vecbd{H}(\curl; \Omega)}
\newcommand{\hdiv}{\vecbd{H}(\div; \Omega)}
\newcommand{\harmonic}[1]{\mathfrak{#1}}

\renewcommand{\cite}{\citep}

\crefname{theorem}{Theorem}{Theorems}
\crefname{lemma}{Lemma}{Lemmas}
\crefname{corollary}{Corollary}{Corollaries}

\crefname{section}{Section}{Sections}
\crefname{subsection}{Subsection}{Subsections}
\crefname{appendix}{Appendix}{Appendices}
\crefname{algorithm}{Algorithm}{Algorithms}

\crefname{table}{Table}{Tables}
\crefname{figure}{Figure}{Figures}
\Crefname{figure}{Figure}{Figures}

\crefformat{equation}{\textup{#2(#1)#3}}
\crefrangeformat{equation}{\textup{#3(#1)#4--#5(#2)#6}}
\crefmultiformat{equation}{\textup{#2(#1)#3}}{ and \textup{#2(#1)#3}}
{, \textup{#2(#1)#3}}{, and \textup{#2(#1)#3}}
\crefrangemultiformat{equation}{\textup{#3(#1)#4--#5(#2)#6}}%
{ and \textup{#3(#1)#4--#5(#2)#6}}{, \textup{#3(#1)#4--#5(#2)#6}}{, and \textup{#3(#1)#4--#5(#2)#6}}

\Crefformat{equation}{#2Equation~\textup{(#1)}#3}
\Crefrangeformat{equation}{Equations~\textup{#3(#1)#4--#5(#2)#6}}
\Crefmultiformat{equation}{Equations~\textup{#2(#1)#3}}{ and \textup{#2(#1)#3}}
{, \textup{#2(#1)#3}}{, and \textup{#2(#1)#3}}
\Crefrangemultiformat{equation}{Equations~\textup{#3(#1)#4--#5(#2)#6}}%
{ and \textup{#3(#1)#4--#5(#2)#6}}{, \textup{#3(#1)#4--#5(#2)#6}}{, and \textup{#3(#1)#4--#5(#2)#6}}

\hypersetup{
  colorlinks = true,
  allcolors = green!50!black,
  urlcolor = red!50!black,
}

\ifarxiv
    \makeatletter
   \def\ps@opening
  {%
      \def\@oddfoot{{%
            \hbox to \textwidth{\parbox{\textwidth}{\hspace*{1pt}\vspace*{20pt}\newline%
            \fontsize{10bp}{8}\fontshape{n}\selectfont \centering \@copyrightstatement%
              }}%
            }}%
      \def\@evenfoot{{%
            \hbox to \textwidth{\parbox{\textwidth}{\hspace*{1pt}\vspace*{20pt}\newline%
            \fontsize{10bp}{8}\fontshape{n}\selectfont \centering \@copyrightstatement%
              }}%
            }}%
    }
    \let\@evenhead\relax
    \let\@oddhead\relax
    \makeatother
\fi

\begin{document}

\DOI{DOI HERE}
\copyrightyear{2026}
\vol{00}
\pubyear{2026}
\access{Advance Access Publication Date: Day Month Year}
\appnotes{Paper}
\copyrightstatement{Published by Oxford University Press on behalf of the 
Institute of Mathematics and its Applications. All rights reserved.}

\ifarxiv
    \copyrightstatement{Distribution Statement A.  Approved for public release: distribution is unlimited.}
\fi


\title[Iterated penalty for structure-preserving discretizations]{%
On the convergence of iterated penalty methods for 
structure-preserving discretizations of saddle point problems}

\author{Patrick E.~Farrell\ORCID{0000-0002-1241-7060}
\address{\orgdiv{Mathematical Institute}, 
         \orgname{University of Oxford}, 
         \orgaddress{\street{Radcliffe Observatory, Andrew Wiles Building, 
            Woodstock Rd}, 
         \postcode{OX2 6GG}, 
         \state{Oxford}, 
         \country{UK}}
\\
\orgdiv{Mathematical Institute},
         \orgname{Charles University},
         \orgaddress{\street{Sokolovská 49/83},
         \postcode{18675},
         \state{Prague},
         \country{Czechia}}}}

\author{Michael Neilan\ORCID{0000-0002-1564-1041}
\address{\orgdiv{Department of Mathematics}, 
         \orgname{University of Pittsburgh}, 
         \orgaddress{\street{301 Thackeray Hall}, 
         \postcode{15260}, 
         \state{Pennsylvania}, 
         \country{USA}}}}

\author{Charles Parker*\ORCID{0000-0003-0767-5732}
\address{\orgdiv{Acoustics Division}, 
         \orgname{U.S. Naval Research Laboratory}, 
         \orgaddress{\street{4555 Overlook Ave SW}, 
         \postcode{20375}, 
         \state{Washington, D.C.}, 
         \country{USA}}}}

\author{L. Ridgway Scott\ORCID{0000-0002-7885-7106}
\address{\orgdiv{Department of Computer Science}, 
         \orgname{University of Chicago}, 
         \orgaddress{\street{5730 S Ellis Ave}, 
         \postcode{60637}, 
         \state{Illinois}, 
         \country{USA}}}}

\authormark{P.~E.~Farrell, M.~Neilan, C.~Parker, and L.~R.~Scott}

\corresp[*]{Corresponding author: 
    \href{email:charles.w.parker185.ctr@us.navy.mil}{%
        charles.w.parker185.ctr@us.navy.mil}}

\received{Date}{0}{Year}
\revised{Date}{0}{Year}
\accepted{Date}{0}{Year}


\abstract{We present new convergence estimates for the iterated penalty method 
		applied to structure-preserving discretizations of 
		linear generalized saddle point systems. The method may 
		be viewed as an Uzawa iteration on an augmented Lagrangian formulation 
		of the system. As a by-product, we obtain sharper stability estimates 
		for penalized/perturbed saddle point problems. Three model
		finite element applications show agreement with the theory.}

\maketitle



\section{Introduction}

We consider the numerical solution of generalized linear saddle point 
problems of the form: Find $u \in V$ and $p \in Q$ such that
\begin{subequations}
    \label{eqn:saddle}
    \begin{alignat}{2}
        a(u, v) + (\operator{D} v, p)_{Q} &= F(v) \qquad & &\forall v \in 
        V, \\
        (\operator{D} u, q)_{Q} &= G(q) \qquad & &\forall q \in Q.
    \end{alignat}	
\end{subequations}
Here, $V$ and $Q$ are real Hilbert spaces with inner-products 
$(\cdot,\cdot)_{\circ}$, norms $\|\cdot\|_{\circ}$, 
and dual spaces $\dual{\circ}$ for $\circ \in \{V, Q\}$; 
$a : V \times V \to \mathbb{R}$ is a continuous bilinear form; 
$\operator{D} : V \to Q$ is a surjective, bounded, linear operator; and
$F \in \dual{V}$ and $G \in \dual{Q}$. Problems of the form 
\cref{eqn:saddle} arise in a variety of applications including 
incompressible flow, elasticity, liquid crystals, electromagnetism, and 
constrained optimization 
(see e.g.~\cite{BenziGolubLiesen05,BoffiBrezziFortin13}).  	
The well-posedness of \cref{eqn:saddle} is given by classical 
Babu\v{s}ka--Brezzi theory
\cite{Babuska71,Brezzi74,Necas62}.

We discretize \cref{eqn:saddle} with conforming finite dimensional 
spaces $\discrete{V} \subset V$ and $\discrete{Q} \subset Q$: 
Find $u_X \in \discrete{V}$ and $p_X \in \discrete{Q}$ such that
\begin{subequations}
    \label{eqn:saddle-discrete}
    \begin{alignat}{2}
        \label{eqn:saddle-discrete-1}
        a(u_X, v) + (\operator{D} v, p_X)_{Q} &= F(v) \qquad & &\forall v 
        \in 
        \discrete{V}, \\
        \label{eqn:saddle-discrete-2}
        (\operator{D} u_X, q)_{Q} &= G(q) \qquad & &\forall q \in 
        \discrete{Q}.
    \end{alignat}	
\end{subequations}
We assume that \cref{eqn:saddle-discrete} is well posed and
that the discrete spaces are \emph{structure-preserving}
in the sense that $\operator{D} \discrete{V} = \discrete{Q}$.
Structure-preserving schemes offer a number of advantages. 
As we will see in \cref{sec:problem-setup} below, 
a discrete inf-sup condition necessary for the well-posedness
of \cref{eqn:saddle-discrete} automatically holds. Additionally,
one obtains improved error estimates 
\cite[Theorem 5.2.4]{BoffiBrezziFortin13},
\cite[Lemma 50.2 and Corollary 50.5]{ErnGuermondII21},
robustness properties \cite[Remark 50.8]{ErnGuermondII21}
and more accurate representation of the underlying physics of the problem;
see e.g. \cite{JohnLinkeMerdonNeilanRebholz17} for the case of 
incompressible flow. However, characterizing and constructing a basis for 
$\discrete{Q} = \operator{D} \discrete{V}$ is typically not a simple task,
and in many cases is an open problem.

One solution method that avoids the need for an explicit basis of 
$\operator{D} \discrete{V}$ is a simple Uzawa iteration 
\cite{ArrowHurwiczUzawa58} applied to an augmented Lagrangian formulation 
of \cref{eqn:saddle-discrete}, defined precisely in 
\cref{eqn:iter-penalty-2param} below. While the use of augmented 
Lagrangian methods dates back to Hestenes \cite{Hestenes69} and 
Powell \cite{Powell69} in the nonlinear optimization literature, 
the first use of these techniques with finite element methods 
appears to be Fortin \& Glowinski \cite{FortinGlow83}
for incompressible flow problems ($G = 0$). 
In the context of structure-preserving finite element methods
(again for incompressible flow problems with $G = 0$), 
Scott \& Bagheri \cite{ScottBagheri90} appear to be the first
to note that an explicit basis for $\operator{D} \discrete{V}$
can be avoided with an Uzawa iteration applied to 
an augmented Lagrangian formulation; the iteration scheme was 
dubbed ``the iterated penalty method''. The method and its analysis were 
extended to general saddle point problems in the first edition of 
\cite{BrennerScott08} and to the case that $G(q) = (\operator{D} v, q)_{Q}$ 
in \cite{ScottIlinMetcalfeBagheri96}. In the context of 
incompressible flow problems, variants of the iterated penalty 
method for high-order discretizations have 
been developed in \cite{AinsworthParker23scip} and for 
inhomogeneous boundary conditions in 
\cite{EickmannScottTscherpel25,ScottIlinMetcalfeBagheri96}. 

Our main contributions are novel convergence estimates for the iterated 
penalty method and an implementation scheme for the general case $G \neq 0$. 
We show that under suitable assumptions on the penalty 
parameters, depending on the structure of the bilinear form 
$a(\cdot,\cdot)$, the iterated penalty method converges at a geometric 
rate (\cref{thm:iter-penalty-2param-error} below). 
All of our estimates are explicit in the typical stability parameters
associated with the well-posedness of the discrete
saddle point system \cref{eqn:saddle-discrete}.
As a by-product, we also obtain sharper stability estimates for a 
penalized/perturbed formulation of \cref{eqn:saddle-discrete} 
(\cref{thm:saddle-discrete-continuity} below) than 
those available in the literature. 

A detailed comparison with the current literature is given in 
\cref{sec:previous-work} below, so we provide a brief overview here.
The case $a(\cdot,\cdot)$ is symmetric positive definite (SPD) 
can be found in \cite{FortinGlow83}, $a(\cdot,\cdot)$ is nonnegative 
and SPD on the kernel of $\operator{D}$ in \cite{BrennerScott08},
and $a(\cdot,\cdot)$ coercive on the kernel of $\operator{D}$ in 
\cite{AwanouLai05,HuangDaiOrbanSaunders24}. Here, we only assume that 
$a(\cdot,\cdot)$ satisfies an inf-sup condition on the kernel of 
$\operator{D}$, which appears to be novel. 
We also show that if $a(\cdot,\cdot)$ is positive definite,
or nonnegative and SPD on the kernel of $\operator{D}$, then the
one-parameter variant of the iterated penalty method converges 
at a geometric rate for any choice of penalty parameter, another result that 
appears to be new.

The remainder of the paper is organized as follows. In 
\cref{sec:problem-setup}, we provide a more precise problem setup 
with three motivating examples. Stability estimates for perturbed 
saddle point systems appear in \cref{sec:penalty-method}. These 
estimates are then used to prove new convergence estimates for 
the iterated penalty method in \cref{sec:iterated-penalty}, which 
also details the implementation of the method. Numerical 
results for the three motivating examples appear in 
\cref{sec:numerics}. In the case that $a(\cdot,\cdot)$ is SPD,
we perform a more refined convergence analysis in 
\cref{sec:refined-analysis-aspd}. Finally, a more thorough comparison 
to existing works appears in \cref{sec:previous-work}. 

\section{Problem setup}
\label{sec:problem-setup}

Standard theory (e.g.~\cite[Theorem 4.2.3]{BoffiBrezziFortin13}) 
shows that problem \cref{eqn:saddle} is 
well-posed if and only if there exists $\alpha > 0$ such that
\begin{align}
    \label{eqn:a-inf-sup-kernel}
    \alpha \leq 
    \newinf_{u \in Z} \sup_{v \in Z} \frac{a(u, v)}{\|u\|_{V} \|v\|_{V}} 
    \quad \text{and} \quad \alpha \leq \newinf_{v \in Z} \sup_{u \in Z} 
    \frac{a(u, v)}{\|u\|_{V} \|v\|_{V}},
\end{align}
where $Z := \{ v \in V : (\operator{D} v, q)_Q = 0 \ \forall q \in Q \}$ 
denotes the kernel of $\operator{D}$,
and $\operator{D} : V \to Q$ is surjective: there exists $\beta > 0$ 
such that
\begin{align}
    \label{eqn:b-inf-sup}
    \beta \leq \newinf_{v \in V} \sup_{q \in Q} \frac{(\operator{D} v, 
        q)_Q}{ 
        \|v\|_V \|q\|_Q }.
\end{align}
Analogous to the continuous case, problem \cref{eqn:saddle-discrete} is well-posed if 
and only if 
\begin{align}
    \label{eqn:a-infsup-kernel-discrete}
    \alpha_X := 
    \newinf_{u \in \discrete{Z}} \sup_{v \in \discrete{Z}} \frac{a(u, 
        v)}{\|u\|_{V} \|v\|_{V}} = \newinf_{v \in \discrete{Z}} \sup_{u \in 
        \discrete{Z}} \frac{a(u, v)}{\|u\|_{V} \|v\|_{V}} > 0,
\end{align}
where $\discrete{Z}$ denotes the discrete kernel of $\operator{D}$,
\begin{align*}
    \discrete{Z} := \{ v \in \discrete{V} : (\operator{D} v, q)_Q = 0 \ 
    \forall q \in \discrete{Q} \},
\end{align*}
and 
\begin{align}
    \label{eqn:b-inf-sup-discrete}
    \beta_X := \newinf_{q \in \discrete{Q}} \sup_{v \in \discrete{V}} 
    \frac{(\operator{D} v, q)_Q}{ \|v\|_V \|q\|_Q } > 0.
\end{align}
We assume that conditions 
\cref{eqn:a-inf-sup-kernel}--\cref{eqn:b-inf-sup-discrete}
are satisfied throughout the paper.
Note that for a structure-preserving method 
the discrete inf-sup condition \cref{eqn:b-inf-sup-discrete} 
automatically holds since $\operator{D} : \discrete{V} \to \discrete{Q}$
is surjective by construction (although $\beta_X$ may depend on the 
dimensions of $\discrete{V}$ and $\discrete{Q}$). 

\begin{remark}
    \label{rem:non-sp-schemes}		
    Suppose that we have a scheme that is not structure-preserving
    in the sense that $\operator{D} \discrete{V} \neq \discrete{Q}$, 
    but the
    discrete saddle point problem \cref{eqn:saddle-discrete} is
    still well-posed. Then, we can rewrite \cref{eqn:saddle-discrete}
    so that the discretization is structure-preserving for a different
    choice of $\operator{D}$. In particular, let 
    $\operator{P}_{\discrete{Q}} : Q \to \discrete{Q}$ be the 
    $(\cdot,\cdot)_Q$ 
    orthogonal projection operator and define 
    $\operator{D}_{\discrete{Q}} := \operator{P}_{\discrete{Q}} 
    \operator{D}$. Then
    problem \cref{eqn:saddle-discrete} is equivalent to
    \begin{alignat*}{2}
    a(u_X, v) + (\operator{D}_{\discrete{Q}} v, p_X)_{Q} &= F(v) \qquad 
        & &\forall v \in \discrete{V}, \\
    (\operator{D}_{\discrete{Q}} u_X, q)_{Q} &= G(q) \qquad 
        & &\forall q \in \discrete{Q}.
    \end{alignat*}
    Moreover, we have $\operator{D}_{\discrete{Q}} \discrete{V} = 
    \discrete{Q}$ by the inf-sup condition \cref{eqn:b-inf-sup-discrete}.	
    Consequently, even though we assume the discretization is 
    structure-preserving for the remainder of this manuscript, all
    results here apply equally well to non structure-preserving schemes
    provided that every occurrence of $\operator{D}$ is replaced
    with $\operator{D}_{\discrete{Q}}$. 
\end{remark}

\subsection{Key examples}
\label{sec:key-examples}

Here, we provide some examples that fall within
the saddle point framework given in \cref{eqn:saddle}.

\subsubsection{Hodge decompositions}
\label{sec:hodge}

Consider a discrete Hilbert complex
\begin{equation}
    \begin{tikzcd}
        0 \arrow{r} & \discrete{V}^{0}  \arrow{r}{\dee^0} & 
        \discrete{V}^1 \arrow{r}{\dee^1} &
        \cdots \arrow{r}{\dee^{n-1}} &
        \discrete{V}^n \arrow{r} & 
        0,
    \end{tikzcd}	
\end{equation}
where $\{ \discrete{V}^k \}$ are finite-dimensional Hilbert spaces
and $\dee^{k} : \discrete{V}^{k} \to \discrete{V}^{k+1}$ are linear 
operators with the complex 
property $\dee^{k+1} \circ \dee^{k} = 0$. Each space $\discrete{V}^k$ is 
equipped with 
an underlying inner product $(\cdot,\cdot)_{\discrete{W}^k}$,
often the $L^2$ inner product, and we define 
$(\cdot,\cdot)_{\discrete{V}^k} := 
(\cdot,\cdot)_{\discrete{W}^k} 
+ (\dee^k \cdot, \dee^{k} \cdot)_{\discrete{W}^{k+1}}$. We denote the 
harmonic forms in $\discrete{V}^k$ by
\begin{align}
    \label{eqn:hodge-harmonic-forms-def}
    \harmonic{H}^k := \{ \harmonic{h} \in \discrete{V}^k : \dee^{k} 
    \harmonic{h} \equiv 0 \text{ and } 
    (\harmonic{h}, \dee^{k-1} v )_{\discrete{V}^k} 
    = (\harmonic{h}, \dee^{k-1} v )_{\discrete{W}^k} 
    = 0 \ 
    \forall v \in \discrete{V}^{k-1} \},
\end{align}
and the orthogonal complement of the kernel of $\dee^k$ by
\begin{align*}
    \mathfrak{Z}^{k, \perp} &:= \{ v \in \discrete{V}^k : 
    (v, z)_{\discrete{V}^k} 
    = (v, z)_{\discrete{W}^k}
    = 0 \ \forall 
    z \in \discrete{V}^k \text{ with } \dee^k z \equiv 0 \}.
\end{align*}
Then, $\discrete{V}^k$ has the following 
$(\cdot,\cdot)_{\discrete{V}^k}$-
and $(\cdot,\cdot)_{\discrete{W}^k}$-orthogonal
decomposition:
\begin{align*}
    \discrete{V}^k = \dee^{k-1} \mathfrak{Z}^{k-1, \perp} \oplus^{\perp} 
    \harmonic{H}^k 
    \oplus^{\perp} \mathfrak{Z}^{k, \perp},
\end{align*}
which is also called the Hodge decomposition \cite{Arnold18}. In particular, for every $u 
\in \discrete{V}^{k}$, there exist unique 
$\sigma_{\perp} \in \mathfrak{Z}^{k-1, \perp}$, 
$\harmonic{h} \in \harmonic{H}^k$ and 
$u_{\perp} \in \mathfrak{Z}^{k, \perp}$ such that
\begin{align}
    \label{eqn:hodge-decomp}
    u = \dee^{k-1} \sigma_{\perp} + \harmonic{h} + u_{\perp},
\end{align}
and $\dee^{k-1} \sigma$, $\harmonic{h}$, and $u_{\perp}$ are mutually 
$(\cdot,\cdot)_{\discrete{V}^k}$-
and $(\cdot,\cdot)_{\discrete{W}^k}$-orthogonal.

Computing the Hodge decomposition \cref{eqn:hodge-decomp} arises in a 
variety of applications including flow visualization 
and singularity detection in computer graphics and motion representation
in computer vision and robotics; see \cite{Bhatia12} for a review.
The Hodge decomposition also plays a key role in designing helicity
preserving schemes for magnetohydrodynamics \cite{HeFarrellHuAndrews25}.
Typically, explicit bases for $\mathfrak{Z}^{k, \perp}$ and 
$\harmonic{H}^k$ are not readily available and computationally expensive to 
compute. Thus, we seek to compute the Hodge decomposition without 
explicit bases for these spaces.
Direct verification shows that $\sigma_{\perp}$ satisfies the following 
saddle point problem: Find 
$\sigma_{\perp} \in \discrete{V}^{k-1}$ 
and $v \in \dee^{k-1} \discrete{V}^{k-1}$ such that
\begin{subequations}
    \label{eqn:hodge-potential-mixed}
    \begin{alignat}{2}
        (\sigma_{\perp}, \tau)_{\discrete{W}^{k-1}} + (\dee^{k-1} \tau, 
        v)_{\discrete{W}^{k}} &= 
        0 \qquad 
        & 
        &\forall \tau \in \discrete{V}^{k-1}, \\
        (\dee^{k-1} \sigma_{\perp}, w)_{\discrete{W}^k} &= (u, 
        w)_{\discrete{W}^k} \qquad & 
        &\forall w 
        \in 
        \dee^{k-1} 
        \discrete{V}^{k-1}.
    \end{alignat}
\end{subequations}
Similarly, $u_{\perp}$ satisfies the following: Find 
$u_{\perp} \in \discrete{V}^{k}$ 
and $p \in \dee^{k} \discrete{V}^{k}$ such that
\begin{subequations}
    \label{eqn:hodge-dk-mixed}
    \begin{alignat}{2}
        (u_{\perp}, v)_{\discrete{W}^{k}} + (\dee^{k} v, 
        p)_{\discrete{W}^{k+1}} &= 0 
        \qquad & 
        &\forall v 
        \in \discrete{W}^{k}, \\
        (\dee^k u_{\perp}, q)_{\discrete{W}^{k+1}} &= (\dee^k u, 
        q)_{\discrete{W}^{k+1}} \qquad & 
        & \forall 
        q \in \dee^k \discrete{V}^k.
    \end{alignat}
\end{subequations}
Then, the final component is given by $\harmonic{h} = u - \dee^{k-1} 
\sigma_{\perp} - u_{\perp}$. Thus, if we can solve 
\cref{eqn:hodge-potential-mixed,eqn:hodge-dk-mixed} without a basis for 
$\dee^{k-1} \discrete{V}^{k-1}$ and $\dee^k \discrete{V}^k$, we can compute 
the Hodge decomposition of $u$.

\subsubsection{$C^1$-spline discretization of fourth-order problems}
\label{sec:fourth-order}

As shown in \cite{AinsworthParker24c1}, many variational problems 
well-posed in $H^2$ and their $H^2$-conforming discretizations can 
be expressed in the form \cref{eqn:saddle,eqn:saddle-discrete}. For 
simplicity, let $\Omega \subset \mathbb{R}^2$ be a connected Lipschitz 
domain whose boundary consists of the union of finitely many polygons; 
i.e. $\Omega$ is a polygon with a finite number of polygonal holes. We 
partition the boundary of $\Omega$ into $\Gamma_c$, $\Gamma_s$, and 
$\Gamma_f$ and define
\begin{align*}
    W := \{ w \in H^2(\Omega) : w|_{\Gamma_c \cup 
    \Gamma_s} = 0 \text{ and } \partial_n w|_{\Gamma_c} = 0 \}
\end{align*}
with norm $\|\cdot\|_W = \|\cdot\|_{H^2(\Omega)}$.
We assume that the partition is sufficiently regular (e.g.
$\Gamma_c$, $\Gamma_s$, and $\Gamma_f$ each consist of a finite
union of open line segments.)
One conforming subspace of $W$ is the space of 
$C^1$-continuous degree $k \in \mathbb{N}$ splines on a shape-regular 
conforming 
triangular mesh $\mathcal{T}$:
\begin{align*}
    \discrete{W} := \{ w \in C^1(\Omega) : w|_{K} \in 
    \mathcal{P}_k(K) \ \forall K \in \mathcal{T} \} \cap W.
\end{align*}
Consider a well-posed discrete problem
\begin{align}
    \label{eqn:h2-primal-discrete}
    w_X \in \discrete{W}: \qquad	A(w_X, v) = L(v) \qquad \forall v \in 
    \discrete{W},
\end{align}
where $A : W \times W \to \mathbb{R}$ is a continuous bilinear 
form, $L \in \dual{\discrete{W}}$, and the discrete inf-sup condition holds:
\begin{align}
    \label{eqn:h2-a-infsup}
    \gamma_X := 
    \newinf_{u \in \discrete{W}} \sup_{v \in \discrete{W}} 
        \frac{A(u, v)}{\|u\|_{W} \|v\|_{W}} 
    = \newinf_{v \in \discrete{W}} \sup_{u \in \discrete{W}} 
        \frac{A(u, v)}{\|u\|_{W} \|v\|_{W}} > 0.
\end{align}

If we define 
\begin{subequations}
    \label{eqn:DXDY}
    \begin{align}
    \discrete{X} &:= \{ v \in C(\Omega) : v|_{K} \in \mathcal{P}_k(K) \ 
    \forall K \in \mathcal{T} 
    \text{ and } v|_{\Gamma_c \cup \Gamma_s} = 0 \}, \\
    \discrete{Y} &:= \{ \vecbd{v} \in C(\Omega)^2 : \vecbd{v}|_{K} \in 
    \mathcal{P}_{k-1}(K) \ \forall K \in \mathcal{T}, \ 
    \vecbd{v}|_{\Gamma_c} = \vecbd{0}, \text{ and } \unitvec{t}\cdot 
    \vecbd{v}|_{\Gamma_s} = 0 \},
    \end{align}
\end{subequations}
where $\unitvec{t}$ is the unit tangent vector to $\partial \Omega$, then 
$\discrete{W} = \{ v \in \discrete{X} : \grad v \in \discrete{Y} \}$.
Suppose that $A(\cdot,\cdot)$ and $L(\cdot)$ may be decomposed similarly
\begin{align*}
    A(\cdot, \cdot) = a_0(\cdot, \cdot) + a_1(\grad \cdot, \grad \cdot) 
    \quad \text{and} \quad L(\cdot) = L_0(\cdot) + L_1(\grad \cdot),
\end{align*}
where $a_0(\cdot,\cdot)$ is bounded on $H^1(\Omega) \times H^1(\Omega)$, 
$a_1(\cdot,\cdot)$ is bounded on $H^1(\Omega)^2 \times H^1(\Omega)^2$,
$L_0 \in \dual{\discrete{X}}$, and $L_1 \in \dual{\discrete{Y}}$. 

Theorem 4.1 of \cite{AinsworthParker24c1} shows that problem 
\cref{eqn:h2-primal-discrete} can be written in the 
form \cref{eqn:saddle-discrete} as follows. 
Let $\discrete{V} := \discrete{X} \times \discrete{Y}$ and for 
$u = (u_0, \vecbd{u}_1), v = (v_0, \vecbd{v}_1) \in \discrete{V}$, define
\begin{alignat*}{2}
    (u, v)_V &:= (u_0, v_0)_{H^1(\Omega)} 
        + (\vecbd{u}_1, \vecbd{v}_1)_{H^1(\Omega)^2}, 
    \qquad & 
    \operator{D} v &:= \grad v_0 -	\vecbd{v}_1, \\
    a(u, v) &:=  a_0(u_0, v_0) + a_1(\vecbd{u}_1, \vecbd{v}_1), \qquad &
    F(v) &:= L_0(v_0) + L_1(\vecbd{v}_1).
\end{alignat*}
Moreover, set $\discrete{Q} := \operator{D} \discrete{V}$ and define
\begin{align*}
    (\vecbd{p}, \vecbd{q})_{Q} := (\vecbd{p}, \vecbd{q})_{L^2(\Omega)^2} + 
    (\rot \vecbd{p}, \rot \vecbd{q})_{L^2(\Omega)} \qquad \forall 
    \vecbd{p}, \vecbd{q} \in \operator{D} \discrete{V},
\end{align*}
where $\rot \vecbd{p} := \partial_y p_1 - \partial_x p_2$. Then, 
\cite[Lemma 4.1]{AinsworthParker24c1} shows that 
${\discrete{Z} = \{ (w, \grad w) : w \in \discrete{W} \}}$,
and so $\alpha_X \geq C \gamma_X$ and the continuity constant $M_a$ 
(defined by \cref{eqn:a-d-bounded} below) is bounded by the norms of 
$a_0(\cdot,\cdot)$ and $a_1(\cdot,\cdot)$. For this choice of spaces 
and operators, \cref{eqn:saddle-discrete} is well-posed and	
\cite[Theorem 4.1]{AinsworthParker24c1} shows that $u_{X} = (w_{X}, 
\grad w_X)$, where $w_X$ is the solution to 
\cref{eqn:h2-primal-discrete}. 

A basis for the space $\discrete{W}$ on general meshes is only known 
for $k \geq 4$ (see \cite{MorganScott75} for the case $k \geq 5$,
\cite{AlfeldPiperSchumaker87} for $k=4$ without boundary conditions,
and \cite{Scott19} for $k=4$ with boundary conditions)
and an implementation of $\discrete{W}$ is not available in 
any freely available finite element software%
\footnote{While Argyris elements are sometimes available 
\cite{kirby2018a,kirby2019a}, these implementations do not support 
strongly-imposed boundary conditions on general domains.}. Nevertheless,
one may still compute the $H^2$-conforming finite element approximation 
\cref{eqn:h2-primal-discrete} by solving the saddle point system 
\cref{eqn:saddle-discrete}, which involves spaces with at most 
$C^0$-continuity. For three-dimensional 
problems, where a basis for $\discrete{W}$ is not known on general meshes 
unless $k \geq 9$ \cite[Section 18.2]{LaiSchumaker07} 
and additional supersmoothness is imposed on subsimplices of the 
triangulation, one may modify the boundary conditions in 
$\discrete{Y}$ and $(\cdot,\cdot)_Q$ appropriately to again arrive at 
\cref{eqn:saddle-discrete}; see \cite{AinsworthParker24c1} for more details.
We also note that in one dimension, a similar approach works for $2n$-th 
order elliptic equations with a similar form \cite{Li17}.

\subsubsection{Scott--Vogelius discretization of incompressible flow}
\label{sec:key-examples-sv}

Let $\Omega \subset \mathbb{R}^d\ (d \in \{2, 3\})$ be a Lipschitz polyhedral 
domain.
Consider a linear incompressible flow problem of the form \cref{eqn:saddle} 
with $V = H_0^1(\Omega)^d$, 
$Q = \{ q \in L^2(\Omega) : \int_{\Omega} q \d{x} = 0 \}$,
$\operator{D} \vecbd{v} := -\div \vecbd{v}$, and
\begin{align}
    \label{eqn:flow-a-bilinear}
    a(\vecbd{u}, \vecbd{v}) := \nu(\symgrad{\vecbd{u}}, 
    \symgrad{\vecbd{v}})_{L^2(\Omega)} + 
    (\vecbd{w} \cdot \grad \vecbd{u}, \vecbd{v})_{L^2(\Omega)} 
    \qquad \forall \vecbd{u}, \vecbd{v} \in V,
\end{align}
where the given $\vecbd{w} \in \vecbd{L}^{\infty}(\Omega)$ is 
divergence-free. Note
that, thanks to Korn's inequality and the anti-symmetry of 
$(\vecbd{w} \cdot \grad \cdot, \cdot)$, we have
\begin{align*}
    \| \vecbd{v} \|_{H^1(\Omega)^d} \leq C_{\mathrm{K}} \| 
    \symgrad{\vecbd{v}}\|_{L^2(\Omega)^d} 
    \quad \text{and} \quad
    (\vecbd{w} \cdot \grad \vecbd{u}, \vecbd{v})_{L^2(\Omega)} 
    = -(\vecbd{u}, \vecbd{w} \cdot \grad \vecbd{v})_{L^2(\Omega)}
\end{align*}
for all $\vecbd{u}, \vecbd{v} \in V$,
and therefore $a(\cdot,\cdot)$ is coercive on $V$.

We discretize this problem with the Scott--Vogelius elements 
\cite{ScottVogelius85conf,ScottVogelius85norm}:
\begin{align}
    \label{eqn:flow-discrete-spaces-sv}
    \discrete{V} := \{ \vecbd{v} \in C(\Omega)^d : \vecbd{v}|_{K} \in 
    \mathcal{P}_k(K)^d \ \forall K \in \mathcal{T} \} \quad \text{and} 
    \quad \discrete{Q} = \div \discrete{V},
\end{align}
where $\mathcal{T}$ is a shape-regular, conforming simplicial mesh.
For $d=2$ and $k \geq 4$, $\beta_X \geq \beta_0 > 0$ for 
some $\beta_0$ independent of the mesh size $h = \max_{K \in \mathcal{T}} 
\mathrm{diam}(K)$ \cite{ScottVogelius85norm} and the polynomial degree $k$ 
\cite{AinsworthParker22unlocking}. However, $\beta_0$ does depend on a 
particular geometric quantity associated with the mesh, and one can 
construct meshes for which $\beta_X$ is arbitrarily small; we shall see one 
such example in \cref{sec:numerics-incompressible}. For some meshes with 
special structure, additional stability results are known; see e.g.\
\cite{Qin92} for $d=2$ and \cite{Zhang05,Zhang11} for $d=3$. The behavior 
of $\beta_X$ on general tetrahedral meshes remains an open problem.

\section{Penalty method}
\label{sec:penalty-method}

We first consider the stability of the penalty method, 
a simple computational method often used in conjunction with unstable 
discretizations of \cref{eqn:saddle-discrete} ($\beta_X \to 0$ as the 
discretization is refined). The results of this section are of independent 
interest, but they play a critical role in proving convergence
estimates for the iterated penalty method in the next section.	
For $\epsilon \geq 0$, a standard penalty method seeks
$u_{X, \epsilon} \in \discrete{V}$ and $p_{X, \epsilon} \in 
\discrete{Q}$ such that
\begin{subequations}
    \label{eqn:saddle-discrete-penalty}
    \begin{alignat}{2}
        \label{eqn:saddle-discrete-penalty-1}
        a(u_{X, \epsilon}, v) + (\operator{D} v, p_{X, \epsilon})_Q &= 
        F(v) \qquad & &\forall v \in \discrete{V}, \\
        \label{eqn:saddle-discrete-penalty-2}
        (\operator{D} u_{X, \epsilon}, q)_Q - \epsilon (p_{X, \epsilon}, 
        q)_Q &= G(q) \qquad & &\forall q\in \discrete{Q}.
    \end{alignat}
\end{subequations}
For $\epsilon = 0$, note that $u_{X, 0} = u_X$. 

In the case of 
structure-preserving discretizations, 
using the second equation \cref{eqn:saddle-discrete-penalty-2}, we can 
eliminate the pressure $p_{X,\epsilon}$ when $\epsilon > 0$ to obtain the 
following decoupled system:
\begin{subequations}
    \label{eqn:saddle-discrete-penalty-decoupled}
    \begin{alignat}{2}
        \label{eqn:saddle-discrete-penalty-decoupled-1}
        a_{\epsilon}(u_{X, \epsilon}, v) &= F(v) + 
        \epsilon^{-1}G(\operator{D} v) 
        \qquad & &\forall v \in \discrete{V}, \\
        \label{eqn:saddle-discrete-penalty-decoupled-2}
        p_{X, \epsilon} &= \epsilon^{-1}\left( \operator{D} u_{X, 
        \epsilon} - \operator{R}_{\discrete{Q}}^{-1} G \right) \qquad & 
        &\forall q \in \discrete{Q},
    \end{alignat}
\end{subequations}
where
\begin{align}
    a_{\epsilon}(u, v) := a(u, v) + \epsilon^{-1} (\operator{D} u, 
    \operator{D} v)_Q 
    \qquad 
    \forall u, v \in \discrete{V},
\end{align}
and $\operator{R}_{\discrete{Q}} : \discrete{Q} \to 
\dual{\discrete{Q}}$ is the Riesz map.
In general, 
\cref{eqn:saddle-discrete-penalty,eqn:saddle-discrete-penalty-decoupled} 
will not be well-posed for all $\epsilon > 0$, as the following example 
demonstrates.
\begin{example}
    Consider the $3 \times 3$ system
    \begin{align*}
        \begin{pmatrix}
            \alpha_0 & 0 & 0  \\
            0 & \alpha_1 & \beta \\
            0 & \beta & -\epsilon
        \end{pmatrix} \begin{pmatrix}
            u_1 \\ u_2 \\ p
        \end{pmatrix} = \begin{pmatrix}
            f_1 \\ f_2 \\ g
        \end{pmatrix}, 
    \end{align*}
    where $\alpha_0, \alpha_1 \in \mathbb{R}$ and  $\beta$ and $\epsilon$ 
    are positive constants. Then, 
    \cref{eqn:saddle-discrete-penalty-decoupled} becomes
    \begin{align*}
        p = \frac{\beta u_2  - g}{\epsilon},  \quad \text{and} \quad 
        \begin{pmatrix}
            \alpha_0 & 0 \\
            0 & \alpha_1 + \frac{\beta^2}{\epsilon}
        \end{pmatrix} \begin{pmatrix}
            u_1 \\ u_2 
        \end{pmatrix} = \begin{pmatrix}
            f_1 \\ f_2 + \frac{\beta}{\epsilon} g
        \end{pmatrix}.
    \end{align*}
    If $\alpha_1 < 0$, then the system is not invertible for the choice 
    $\epsilon = -\beta^2/\alpha_1 > 0$.
\end{example}

\subsection{Well-posedness}

To describe the stability of \cref{eqn:saddle-discrete-penalty}, we require
some additional constants. Let $M_a$ and $M_{\operator{D}}$ denote
the continuity constants of $a(\cdot,\cdot)$ and $\operator{D}$:
\begin{align}
    \label{eqn:a-d-bounded}
    M_{a} := \sup_{u, v \in V} \frac{ a(u, v) }{\|u\|_V \|v\|_V}
    \quad \text{and} \quad
    M_{\operator{D}} := \sup_{v \in V} \frac{\|\operator{D} v\|_Q}{ 
        \|v\|_V }.
\end{align}
Moreover, given a subspace $U \subseteq V$, let
\begin{align*}
    \|F\|_{\dual{U}} := \sup_{u \in U} \frac{|F(u)|}{\|u\|_V} \qquad 
    \forall F \in \dual{V}.
\end{align*}
The bilinear form $a(\cdot,\cdot)$ oftentimes possesses more structure
than \cref{eqn:a-infsup-kernel-discrete}. The three additional
cases we shall consider are the following: 

\makeatletter
\newcommand{\mylabel}[2]{#2\def\@currentlabel{#2}\label{#1}}
\makeatother

\begin{enumerate}
    \item[\mylabel{asmp:a-sym-nonneg-coercive-z}{(A1)}]
    $a(\cdot,\cdot)$ is 
    symmetric and there exists $\tilde{\alpha}_X > 0$ such that
    \begin{align}
        \label{eqn:a-nonnegative-coercive-kernel}
        a(u, u) \geq 0 \qquad \forall u \in \discrete{V} \quad \text{and} 
        \quad 
        a(z, z) \geq \tilde{\alpha}_X \|z\|_{V}^2 \qquad \forall z \in 
        \discrete{Z}.
    \end{align}
    
    \item[\mylabel{asmp:a-coercive-all}{(A2)}] There exists 
    $\tilde{\alpha}_X > 0$ such that
    \begin{align}
        \label{eqn:a-coercive-all}
        a(u, u) \geq \tilde{\alpha}_X \|u\|_{V}^2 \qquad \forall u \in 
        \discrete{V}.
    \end{align}
    
    \item[\mylabel{asmp:a-sym-coercive-all}{(A3)}] $a(\cdot,\cdot)$ is 
    symmetric and satisfies \ref{asmp:a-coercive-all}.
\end{enumerate}

\vspace{1em}

\noindent Note that \ref{asmp:a-sym-coercive-all} implies 
\ref{asmp:a-sym-nonneg-coercive-z}, which in turn implies 
\cref{eqn:a-infsup-kernel-discrete} with $\alpha_X = \tilde{\alpha}_X$. 
Conversely, if $a(\cdot,\cdot)$ is symmetric and nonnegative, then the 
inf-sup condition \cref{eqn:a-infsup-kernel-discrete} implies 
\cref{eqn:a-nonnegative-coercive-kernel} with $\tilde{\alpha}_X = 
\alpha_X^2 / M_a$ \cite[Lemma 4.2.2]{BoffiBrezziFortin13}. We also define
\begin{align}
    \label{eqn:phix-upsilonx-def}
    \Phi_X := \begin{cases}
        \sqrt{\frac{M_a}{\tilde{\alpha}_X}} & \text{if 
            \ref{asmp:a-sym-nonneg-coercive-z} holds}, \\
        \frac{M_a}{\tilde{\alpha}_X} & \text{if \ref{asmp:a-coercive-all} 
        holds},	\\
        \frac{M_a}{\alpha_X} & \text{otherwise}, 
    \end{cases}
    \quad \text{and} \quad 	\Upsilon_X 
    := \begin{cases}
        \Phi_X & \text{if \ref{asmp:a-coercive-all} holds}, \\
        1 + \Phi_X & \text{otherwise}.
    \end{cases}
\end{align}

Finally, we complement the discrete kernel $\discrete{Z}$ with two 
``$a(\cdot,\cdot)$-orthogonal" subspaces:
\begin{align*}
    \tilde{\discrete{Z}} &:= \{ v \in \discrete{V} : a(v, z) 
    = 0 \ \forall z \in \discrete{Z} \}
    \quad \text{and} \quad 
    \hat{\discrete{Z}} := \{ v \in \discrete{V} : a(z, v) 
    = 0 \ \forall z \in \discrete{Z} \}.
\end{align*}	
Note that if $a(\cdot,\cdot)$ is symmetric, then $\tilde{\discrete{Z}} = 
\hat{\discrete{Z}}$. The well-posedness of the penalty method is
summarized in the following theorem.

\begin{theorem}
    \label{thm:saddle-discrete-continuity}
    Let $F \in \dual{\discrete{V}}$ and $G \in \dual{\discrete{Q}}$
    and define
    \begin{align}
        \label{eqn:eps0-def}
        \epsilon_0 := \begin{cases}
            \infty & \text{if \ref{asmp:a-sym-nonneg-coercive-z} or 
            \ref{asmp:a-coercive-all} holds}, \\
            \frac{1}{M_a} \left( \frac{\beta_X}{\Upsilon_X} 
            \right)^2 & \text{otherwise}.
        \end{cases} 
    \end{align}
    For all $\epsilon \in [0, \epsilon_0)$, 
    there exists a unique 
    solution $u_{X, \epsilon} \in \discrete{V}$ and $p_{X, \epsilon} \in 
    \discrete{Q}$ to \cref{eqn:saddle-discrete-penalty}. Moreover, 
    there exist unique $z_{X, \epsilon} \in \discrete{Z}$ and 
    $\tilde{z}_{X, \epsilon} \in \tilde{\discrete{Z}}$ such that
    $u_{X, \epsilon} = z_{X, \epsilon} + \tilde{z}_{X, \epsilon}$,
    and there holds
    \begin{subequations}
        \label{eqn:saddle-perturb-cont}
        \begin{align}
            \label{eqn:saddle-perturb-z-cont}
            \|z_{X, \epsilon}\|_V &\leq C_{X}^{[1]} 
            \|F\|_{\dual{\discrete{Z}}}, \\
            \label{eqn:saddle-perturb-tildez-cont}
            \|\tilde{z}_{X, \epsilon}\|_V &\leq \epsilon C_{X, 
                \epsilon}^{[2]}
            \|F\|_{\dual{\hat{\discrete{Z}}}} 
            + C_{X, \epsilon}^{[3]} \|G\|_{\dual{\discrete{Q}}}, \\
            \label{eqn:saddle-perturb-p-cont}
            \|p_{X, \epsilon}\|_Q &\leq C_{X, \epsilon}^{[3]} 
            \|F\|_{\dual{\hat{\discrete{Z}}}} 
            + C_{X, \epsilon}^{[4]} \|G\|_{\dual{\discrete{Q}}}, 
        \end{align}
    \end{subequations}	
    where the constants $C_X^{[1]}$ and $C_{X, \epsilon}^{[i]}$, 
    $i=2,\ldots,4$, are given in 
    \cref{tab:saddle-perturb-cont-const}. 
\end{theorem}
The proof of \cref{thm:saddle-discrete-continuity} for each column in 
\cref{tab:saddle-perturb-cont-const} is given in 
\cref{sec:saddle-discrete-continuity-proof-a-gen,%
    sec:saddle-discrete-continuity-proof-a-sym-nonneg,%
    sec:saddle-discrete-continuity-proof-a-coercive-all,%
    sec:saddle-discrete-continuity-proof-a-sym-coercive-all}.

\begin{table}[htb]
    \caption{Stability constants appearing in 
        \cref{thm:saddle-discrete-continuity}, 
        where $\alpha_X$ is defined in \cref{eqn:a-infsup-kernel-discrete},
        $\tilde{\alpha}_X$ in \cref{eqn:a-nonnegative-coercive-kernel} or 
        \cref{eqn:a-coercive-all},
        $M_a$ and $M_{\operator{D}}$ in \cref{eqn:a-d-bounded}, 
        $\beta_X$ in \cref{eqn:b-inf-sup-discrete}, 
        $\Upsilon_X$ in \cref{eqn:phix-upsilonx-def}, 
        and $\epsilon_0$ in \cref{eqn:eps0-def}. In the final two rows of 
        the final column, the values of the constants are the same as the  
        referenced columns, with $\Upsilon_X = \sqrt{M_a / 
        \tilde{\alpha}_X}$ as in \cref{eqn:phix-upsilonx-def}.
        \label{tab:saddle-perturb-cont-const}}
    \begin{tabular*}{\textwidth}{@{\extracolsep\fill} c c c c c @{\extracolsep\fill}}
        & \multicolumn{4}{c}{Additional assumptions on 
                    $a(\cdot,\cdot)$} \\[0.2em]
            & None & \ref{asmp:a-sym-nonneg-coercive-z} 
            & \ref{asmp:a-coercive-all}
            & \ref{asmp:a-sym-coercive-all} \\
            \hline
            $\displaystyle C_{X}^{[1]}$ \rule{0pt}{1.5\normalbaselineskip} 
            & $\displaystyle \frac{1}{\alpha_X}$
            & $\displaystyle \frac{1}{\tilde{\alpha}_X}$ 
            & $\displaystyle \frac{1}{\tilde{\alpha}_X}$
            & $\displaystyle \frac{1}{\tilde{\alpha}_X}$ \\
            $\displaystyle C_{X, \epsilon}^{[2]}$ \rule{0pt}{2\normalbaselineskip} 
            & $\displaystyle \frac{\Upsilon_X^2}{\beta_X^2} 
            \left( \frac{\epsilon_0}{\epsilon_0 - \epsilon} \right)$
            & $\displaystyle \frac{\Upsilon_X^2}{\beta_X^2}$
            & $\displaystyle \frac{\Upsilon_X^3}{\epsilon \tilde{\alpha}_X \Upsilon_X^2 + 
            \beta_X^2}$ 
            & $\displaystyle \frac{\Upsilon_X^2}{\epsilon M_a + \beta_X^2}$ \\
            $\displaystyle C_{X, \epsilon}^{[3]}$ \rule{0pt}{2\normalbaselineskip} 
            & $\displaystyle \frac{\Upsilon_X}{\beta_X} \left(
            \frac{\epsilon_0}{\epsilon_0 - \epsilon} \right)$
            & $\displaystyle \frac{\Upsilon_X}{\beta_X}$
            &  $\displaystyle \frac{\Upsilon_X M_{\operator{D}}}{ 
            \sqrt{(\epsilon \tilde{\alpha}_X \Upsilon_X^2 + 
            \beta_X^2)(\epsilon \tilde{\alpha}_X + M_{\operator{D}}^2)
            } }$
            & same as \ref{asmp:a-coercive-all} \\
            $\displaystyle C_{X, \epsilon}^{[4]}$ \rule{0pt}{2\normalbaselineskip} 
            & $\displaystyle \frac{M_a \Upsilon_X}{\beta_X^2} 
            \left( \frac{\epsilon_0}{\epsilon_0 - \epsilon}  \right)$
            & $\displaystyle \frac{M_a}{\epsilon M_a + \beta_X^2}$
            & $\displaystyle \frac{M_a \Upsilon_X}{\epsilon M_a \Upsilon_X + \beta_X^2}$
            & same as \ref{asmp:a-sym-nonneg-coercive-z}
    \end{tabular*}
\end{table}

\Cref{thm:saddle-discrete-continuity} readily extends to the case
that the $(2,2)$ block of \cref{eqn:saddle-discrete-penalty} is
replaced with a symmetric coercive bilinear form 
$c(\cdot,\cdot) : \discrete{Q} \times \discrete{Q} \to \mathbb{R}$ with
\begin{align}
    \label{eqn:c-bounded-coercive}
    c(p, p) \geq \kappa_X \|p\|_{Q}^2 \quad \text{and} \quad |c(p, q)| \leq 
    M_c \|p\|_Q \|q\|_Q \qquad \forall p, q \in \discrete{Q}.
\end{align}
That is,
consider the problem: Find $u_{X, c, \epsilon} \in \discrete{V}$ and 
$p_{X, c, \epsilon} \in \discrete{Q}$ such that
\begin{subequations}
    \label{eqn:saddle-discrete-penalty-c22}
    \begin{alignat}{2}
        \label{eqn:saddle-discrete-penalty-c22-1}
        a(u_{X, c, \epsilon}, v) + (\operator{D} v, p_{X, c, \epsilon})_Q 
        &= F(v) 
        \qquad & &\forall v \in \discrete{V} \\
        \label{eqn:saddle-discrete-penalty-c22-2}
        (\operator{D} u_{X, c, \epsilon}, q)_Q 
            - \epsilon c(p_{X, c, \epsilon}, q) &= G(q) 
        \qquad & &\forall q\in \discrete{Q}.
    \end{alignat}
\end{subequations}
Then, we have the following result.
\begin{corollary}
    Let $M_{\operator{D}, c} := M_{\operator{D}} / \sqrt{\kappa_X}$
    and $\beta_{X, c} := \beta_X / \sqrt{M_c}$ and define $\epsilon_{0, c}$
    and $C_{X, c, \epsilon}^{[i]}$, $i=2,\ldots,4$, as in 
    \cref{eqn:eps0-def}
    and \cref{tab:saddle-perturb-cont-const} with $M_{\operator{D}}$ and 
    $\beta_X$ replaced by $M_{\operator{D}, c}$ and $\beta_{X, c}$, respectively.
    Then, for all $\epsilon \in [0, \epsilon_{0, c})$, there exists a 
    unique 
    solution $u_{X, \epsilon} \in \discrete{V}$ and $p_{X, \epsilon} \in 
    \discrete{Q}$ to \cref{eqn:saddle-discrete-penalty}. Moreover, 
    there exist unique $z_{X, c, \epsilon} \in \discrete{Z}$ and 
    $\tilde{z}_{X, c, \epsilon} \in \tilde{\discrete{Z}}$ such that
    $u_{X, c, \epsilon} = z_{X, c, \epsilon} + \tilde{z}_{X, c, \epsilon}$,
    and there holds
    \begin{subequations}
        \label{eqn:saddle-perturb-cont-c22}
        \begin{align}
            \label{eqn:saddle-perturb-z-cont-c22}
            \|z_{X, c, \epsilon}\|_V &\leq 
                C_{X}^{[1]} \|F\|_{\dual{\discrete{Z}}}, \\
            \label{eqn:saddle-perturb-tildez-cont-c22}
            \|\tilde{z}_{X, c, \epsilon}\|_V &\leq 
                \epsilon C_{X, c, \epsilon}^{[2]}		
                \|F\|_{\dual{\hat{\discrete{Z}}}} 
            + \frac{C_{X, c, \epsilon}^{[3]}}{\sqrt{\kappa_X}} 
            \|G\|_{\dual{\discrete{Q}}}, \\
            \label{eqn:saddle-perturb-p-cont-c22}
            \|p_{X, \epsilon}\|_Q &\leq 
                \frac{C_{X, c, \epsilon}^{[3]}}{\sqrt{\kappa_X}} 			
                \|F\|_{\dual{\hat{\discrete{Z}}}} 
            + \frac{C_{X, c, \epsilon}^{[4]}}{\kappa_X} 
            \|G\|_{\dual{\discrete{Q}}}, 
        \end{align}
    \end{subequations}	
    where $C_{X}^{[1]}$ is defined in \cref{tab:saddle-perturb-cont-const} 
    and $\kappa_X$ and $M_c$ are defined in \cref{eqn:c-bounded-coercive}.
\end{corollary}
\begin{proof}
    Let $\operator{D}_c : \discrete{V} \to \discrete{Q}$ be defined by
    \begin{align*}
        c(\operator{D}_c v, q) = (\operator{D} v, q)_Q \qquad \forall q \in 
        \discrete{Q}, \ \forall v \in \discrete{V},
    \end{align*}
    and set $\|\cdot\|_c^2 := c(\cdot,\cdot)$. Then, there holds
    \begin{align*}
        \| \operator{D}_c v \|_c 
        = \frac{(\operator{D} v, \operator{D}_c v)_Q}{
            \| \operator{D}_c v \|_c}
        \leq \|\operator{D} v\|_{Q} 
            \frac{\|\operator{D}_c v \|_Q}{\| \operator{D}_c v \|_c} 
        \leq \frac{M_{\operator{D}}}{\sqrt{\kappa_X}} \|v\|_V = 
        M_{\operator{D}, c} \|v\|_V  \qquad \forall v \in \discrete{V}, 
    \end{align*}
    and
    \begin{align*}
        \newinf_{q \in \discrete{Q}} \sup_{v \in \discrete{V}} 
        \frac{c(\operator{D}_c v, q)}{ \|v\|_V \|q\|_c } 
        \geq \frac{1}{\sqrt{M_c}} 
        \newinf_{q \in \discrete{Q}} \sup_{v \in \discrete{V}} 
        \frac{(\operator{D} v, q)}{ \|v\|_V \|q\|_Q }
        = \frac{\beta_X}{\sqrt{M_c}} = \beta_{X, c}.
    \end{align*}
    Since $u_{X, c, \epsilon}$ and $p_{X, \epsilon}$ satisfy
    \begin{alignat*}{2}
        a(u_{X, c, \epsilon}, v) + c(\operator{D}_c v, p_{X, c, \epsilon})
        &= F(v) 
        \qquad & &\forall v \in \discrete{V}, \\
        c(\operator{D} u_{X, c, \epsilon}, q) 
        - \epsilon c(p_{X, c, \epsilon}, q) &= G(q) 
        \qquad & &\forall q\in \discrete{Q},
    \end{alignat*}
    we apply \cref{thm:saddle-discrete-continuity} 
    with ``$(\cdot,\cdot)_Q = c(\cdot,\cdot)$" and ``$M_{\operator{D}} = 
    M_{\operator{D}, c}$" to obtain
    \cref{eqn:saddle-perturb-z-cont-c22} and
    \begin{align*}
        \|\tilde{z}_{X, c, \epsilon}\|_V &\leq 
        \epsilon C_{X, c, \epsilon}^{[2]}		
        \|F\|_{\dual{\hat{\discrete{Z}}}} 
        + C_{X, c, \epsilon}^{[3]} \sup_{q \in \discrete{Q}} 
        \frac{|G(q)|}{\|q\|_c}, 
        \quad \text{and} \quad 
        \|p_{X, \epsilon}\|_c \leq 
        C_{X, c, \epsilon}^{[3]} 			
        \|F\|_{\dual{\hat{\discrete{Z}}}} 
        + C_{X, c, \epsilon}^{[4]} \sup_{q \in \discrete{Q}} 
        \frac{|G(q)|}{\|q\|_c}.
    \end{align*}
    The result now follows from \cref{eqn:c-bounded-coercive}.
\end{proof}

\subsection{Error}

We now turn to the error of the penalty method.	
\begin{lemma}
    \label{lem:saddle-perturb-error}
    Let $u_X \in \discrete{V}$ and $p_X \in \discrete{Q}$ denote the 
    solution 
    to \cref{eqn:saddle-discrete} and $u_{X, \epsilon} \in \discrete{V}$ 
    and 
    $p_{X, \epsilon} \in \discrete{Q}$ the solution to 
    \cref{eqn:saddle-discrete-penalty} with $\epsilon > 0$ as in 
    \cref{thm:saddle-discrete-continuity}. Let $u_{X} = z_X + 
    \tilde{z}_{X}$ 
    and $u_{X, \epsilon} = z_{X, \epsilon} + \tilde{z}_{X, \epsilon}$ with 
    $z_{X}, z_{X, \epsilon} \in \discrete{Z}$ and $\tilde{z}_{X}, 
    \tilde{z}_{X, \epsilon} \in \tilde{\discrete{Z}}$. Then, there holds
    \begin{subequations}
        \label{eqn:saddle-perturb-error}
        \begin{align}
            \label{eqn:saddle-perturb-z-error}
            z_X &= z_{X, \epsilon}, \\
            \label{eqn:saddle-perturb-tildez-error}
            \| \tilde{z}_X - \tilde{z}_{X, \epsilon}\|_V &\leq \epsilon 
            \min\left\{  C_{X, \epsilon}^{[2]} \left( 
            \|F\|_{\dual{\hat{\discrete{Z}}}} + M_a \|\tilde{z}_X\|_V 
            \right), C_{X, \epsilon}^{[2,3]} \|p_X\|_Q \right\},\\
            \label{eqn:saddle-perturb-p-error}
            \|p_X - p_{X, \epsilon}\|_Q &\leq \min\left\{ 
            \epsilon C_{X,\epsilon}^{[4]} \|p_X\|_Q,
            \frac{M_a}{\beta_X} \|\tilde{z}_X - \tilde{z}_{X, \epsilon}\|_V
            \right\},
        \end{align}
    \end{subequations}
    where $C_{X, \epsilon}^{[j]}$, $j =2,\ldots, 4$, 
    are defined in \cref{tab:saddle-perturb-cont-const}, 
    $C_{X, \epsilon}^{[2,3]} := \min\{
        M_{\operator{D}} C_{X, \epsilon}^{[2]}, C_{X, \epsilon}^{[3]}\}$, 
    $\beta_X$ is defined in \cref{eqn:b-inf-sup-discrete}, and
    $M_a$ is defined in \cref{eqn:a-d-bounded}.
\end{lemma}
\begin{proof}
    Let $e_{X, \epsilon} := u_{X} - u_{X, \epsilon}$ and $r_{X, \epsilon} 
    := p_{X} 
    - p_{X, \epsilon}$. Then, there holds
    \begin{subequations}
        \label{proof:eqn:sadlle-perturb-error-mixed}
        \begin{alignat}{2}
            \label{proof:eqn:sadlle-perturb-error-mixed-1}
            a(e_{X, \epsilon}, v) + (\operator{D} v, r_{X, \epsilon})_Q &= 
            0 
            \qquad & &\forall v \in \discrete{V}, \\
            \label{proof:eqn:sadlle-perturb-error-mixed-2}
            (\operator{D} e_{X, \epsilon}, q)_Q - \epsilon(r_{X, \epsilon}, 
            q)_Q 
            &= -\epsilon(p_{X}, q) 
            \qquad & &\forall q \in \discrete{Q}.
        \end{alignat}
    \end{subequations}
    Applying \cref{eqn:saddle-perturb-cont} then gives that $e_{X, 
        \epsilon} \in 
    \tilde{\discrete{Z}}$ and
    \begin{align*}
        \| e_{X, \epsilon} \|_{V} & \leq \epsilon C_{X, \epsilon}^{[3]} 
        \|p_X\|_Q 
        \quad \text{and} \quad
        \| r_{X, \epsilon} \|_{Q} \leq \epsilon C_{X, \epsilon}^{[4]} 
        \|p_X\|_Q.
    \end{align*}
    Moreover, the inf-sup condition \cref{eqn:b-inf-sup-discrete} gives
    \begin{align*}
        \beta_X \|r_{X, \epsilon}\|_Q \leq \sup_{v \in \discrete{V}} 
        \frac{(\operator{D} v, r_{X, \epsilon})_Q }{\|v\|_V}
        = \sup_{v \in \discrete{V}} \frac{a(e_{X, \epsilon}, v)}{\|v\|_V} 
        \leq M_a \|e_{X, \epsilon}\|_V.
    \end{align*}    
    
    We can obtain an alternative estimate for $e_{X, \epsilon}$ by 
    eliminating $r_{X, \epsilon}$. In particular, there holds
    \begin{align*}
        a_{\epsilon}(e_{X, \epsilon}, v) = -(p_X, \operator{D} v)_Q = 
        a(u_X, v) - 
        F(v) =: L(v) \qquad \forall v 
        \in \discrete{V}.
    \end{align*}	
    Note that $L$ vanishes on $\discrete{Z}$ by 
    \cref{eqn:saddle-discrete-1} 
    and 
    \begin{align*}
        L(\hat{z}) = a(\tilde{z}_{X}, \hat{z}) - F(\hat{z})  \qquad \forall 
        \hat{z} 
        \in \hat{\discrete{Z}} \implies \|L\|_{\dual{\hat{\discrete{Z}}}} 
        \leq 
        \|F\|_{\dual{\hat{\discrete{Z}}}} + M_a \| \tilde{z}_X\|_V.
    \end{align*}
    Alternatively, the identity $L(\hat{z}) = -(p_X, \operator{D} 
    \hat{z})_Q$ gives $\|L\|_{\dual{\hat{\discrete{Z}}}} \leq 
    M_{\operator{D}} \|p_X\|_Q$.
    Now applying \cref{eqn:saddle-perturb-cont}, we obtain
    \begin{align*}
        \| \tilde{z}_X - \tilde{z}_{X, \epsilon}\|_V &\leq \epsilon C_{X, 
            \epsilon}^{[2]} \min\left\{ 
        \|F\|_{\dual{\hat{\discrete{Z}}}} + M_a \|\tilde{z}_X\|_{V}, 
        M_{\operator{D}} \|p_X\|_{Q} \right\} .
    \end{align*}
\end{proof}

\begin{remark}
    If $a(\cdot,\cdot)$ satisfies \ref{asmp:a-coercive-all}, one
    can obtain a better dependence on $\alpha_X$ and $M_a$
    in \cref{eqn:saddle-perturb-error} by first
    applying \cref{eqn:saddle-perturb-error} with 
    $\|\cdot\|_V = |\cdot|_a := \sqrt{a(\cdot,\cdot)}$ for which $\tilde{\alpha}_X = M_a = 1$,
    and then applying \cref{eqn:a-coercive-all,eqn:a-d-bounded}
    to obtain estimates in the original norm on $V$. 
\end{remark}	

\section{Iterated penalty method}
\label{sec:iterated-penalty}

The iterated penalty method reads as follows. 
Let $\lambda, \rho \geq 0$ be given and 
$\discrete{Q} \ni p_{X, \lambda, \rho}^{0} := 0$.
For $n \in \mathbb{N}$, define 
$u_{X, \lambda, \rho}^{n} \in \discrete{V}$ and 
$p_{X, \lambda, \rho}^{n} \in \discrete{Q}$ by
\begin{subequations}
    \label{eqn:iter-penalty-2param}
    \begin{alignat}{2}
        \label{eqn:iter-penalty-2param-1}
        a(u_{X, \lambda, \rho}^{n}, v) + \lambda (\operator{D} u_{X, 
            \lambda, \rho}^{n}, 
        \operator{D} v)_Q &= F(v) + \lambda G(\operator{D} 
        v) - (\operator{D} v, p_{X, \lambda, \rho}^{n-1})_Q  \quad 
        & &\forall v \in \discrete{V}, \\
        \label{eqn:iter-penalty-2param-2}
        p_{X, \lambda, \rho}^{n} &= p_{X, \lambda, \rho}^{n-1} + \rho 
        (\operator{D} u_{X, \lambda, \rho}^{n} - 
        \operator{R}_{\discrete{Q}}^{-1} G). \quad & &
    \end{alignat}
\end{subequations}
One may view \cref{eqn:iter-penalty-2param} as an 
Uzawa iteration \cite{ArrowHurwiczUzawa58}
on the equivalent augmented system
\begin{alignat*}{2}
    a_{1/\lambda}(u_X, v) + (\operator{D} v, p_X)_Q 
        &= F(v) + \lambda G(\operator{D} v)
        \qquad & &\forall v \in \discrete{V}, \\
    (\operator{D} u_X, q)_Q &= G(q) \qquad & &\forall q \in \discrete{Q}.		
\end{alignat*}
In standard references (see e.g. \cite[Chapter 1, 
(5.7)--(5.9)]{FortinGlow83} or \cite[eq. (13.1.4)]{BrennerScott08}), 
the Lagrange multiplier index in \cref{eqn:iter-penalty-2param-2} is 
``$n+1$". Here, we use ``$n$" for aesthetic purposes. The goal of this 
section is to prescribe conditions on $\rho$ and $\lambda$ which guarantee
that the iterated penalty method is well-defined and converges at a 
geometric rate to the solution to \cref{eqn:saddle-discrete}.

To describe the error propagation operator for the iterated penalty method,
we first express \cref{eqn:iter-penalty-2param} in operator notation.	
Let $\operator{A}_{\epsilon} : \discrete{V} \to \dual{\discrete{V}}$ denote 
the linear operator associated with $a_{\epsilon}(\cdot,\cdot)$:
\begin{align*}
    (\operator{A}_{\epsilon} v)(w) = a_{\epsilon}(v, w) \qquad \forall v, w 
    \in \discrete{V},
\end{align*}
with $\operator{A} := \operator{A}_0$, and let $\operator{D}^{\star} : 
\discrete{Q} \to \dual{\discrete{V}}$ 
denote 
the following adjoint of $\operator{D}$:
\begin{align*}
    (\operator{D}^{\star} q)(v) = (\operator{D} v, q)_{Q} \qquad \forall 
    v \in 
    \discrete{V}, \ \forall q \in \discrete{Q}.
\end{align*}
The first result in this section characterizes the error propagation operator.
\begin{lemma}
    \label{lem:iter-penalty-2param-error-prop}
    Let $F \in \dual{\discrete{V}}$, $G \in \dual{\discrete{Q}}$,
    and $\rho > 0$.
    For $\lambda > 1/\epsilon_0$, 
    the iterates 
    $u_{X, \lambda, \rho}^{n} \in \discrete{V}$ and 
    $p_{X, \lambda, \rho}^{n} \in \discrete{Q}$, $n \in \mathbb{N}$,
    given by \cref{eqn:iter-penalty-2param} are well-defined.
    Moreover, let 
    $u_X \in \discrete{V}$ and $p_X \in \discrete{Q}$ denote the 
    solution to \cref{eqn:saddle-discrete}.
    For all $n \in \mathbb{N}$, there holds 
    $u_X - u_{X, \lambda, \rho}^{n} \in \tilde{\discrete{Z}}$ and
    \begin{subequations}
        \label{eqn:iter-penalty-2param-error-prop}
        \begin{align}
            \label{eqn:iter-penalty-2param-error-prop-u}
            u_X - u_{X, \lambda, \rho}^{n} &= \left(  
            \operator{I}_{\discrete{V}} - \rho 
            \operator{A}_{1/\lambda}^{-1} \operator{D}^\star \operator{D} 
            \right)^{n-1}  (u_X - u_{X, 1/\lambda}), \\
            \label{eqn:iter-penalty-2param-error-prop-p}
            p_X - p_{X, \lambda, \rho}^n &= \left( 
            \operator{I}_{\discrete{Q}} 
            - \rho \operator{D} \operator{A}_{1/\lambda}^{-1} 
            \operator{D}^{\star} \right)^n p_X,
        \end{align}
    \end{subequations}
    where $u_{X, 1/\lambda} \in \discrete{V}$ satisfies 
    \cref{eqn:saddle-discrete-penalty-decoupled-1} with $\epsilon=1/\lambda$
    and $\mathcal{I}_{\circ} : \circ \to \circ$ is the identity operator.
\end{lemma}
\begin{proof}
    \Cref{thm:saddle-discrete-continuity} and 
    \cref{eqn:saddle-discrete-penalty-decoupled-1} show that $u_{X, 
        \lambda, \rho}^n$, and hence $p_{X, \lambda, \rho}^n$ are
        well-defined 
    provided $1/\lambda < \epsilon_0$. 
    For $n \in \mathbb{N}$, let 
    $e_{X}^n := u_{X} - u_{X, \lambda, \rho}^n$ 
    and $r_{X}^n:= p_{X} - p_{X,\lambda,\rho}^n$ 
    with $r_X^0 := p_X - p_{X,\lambda,\rho}^0 = p_X$. Direct calculation 
    shows that for all $n \geq 1$, there holds
    \begin{align}
        \label{eqn:proof:iter-penalty-2param-key-ids}
        a_{1/\lambda}(e_X^n, v) = -(\operator{D} v, r_X^{n-1})_Q \qquad 
        \forall v \in \discrete{V}, \quad 
        \text{and} \quad r_X^n = r_X^{n-1} + \rho \operator{D} e_X^n.
    \end{align}
    
    \noindent \textbf{Step 1: \Cref{eqn:iter-penalty-2param-error-prop-p}. 
    } 	In operator form, the relations in 
    \cref{eqn:proof:iter-penalty-2param-key-ids} give
    \begin{align}
        \label{eqn:proof:iter-penalty-2param-error-op-ids}
        e_X^{n} = -\operator{A}_{1/\lambda}^{-1} \operator{D}^{\star} 
        r^{n-1}_X \implies r_X^{n} = \left( \operator{I}_{\discrete{Q}} - 
        \rho \operator{D} 
        \operator{A}_{1/\lambda}^{-1} 
        \operator{D}^{\star} \right) r_X^{n-1} \qquad \forall n \in 
        \mathbb{N}.
    \end{align}
    \Cref{eqn:iter-penalty-2param-error-prop-p} now follows on noting that 
    $r_X^0 = p_X$. 
    
    \noindent \textbf{Step 2: \Cref{eqn:iter-penalty-2param-error-prop-u}. }
    Substituting the second relation in 
    \cref{eqn:proof:iter-penalty-2param-key-ids} into the first, we obtain
    \begin{align}
        \label{eqn:proof:iter-penalty-2param-en-en1-id}
        a_{1/\lambda}(e_X^n, v) &= -(\operator{D} v, r_X^{n-2})_Q - \rho 
        (\operator{D} e_X^{n-1}, \operator{D} v )_Q 
        = a_{1/\lambda}(e_X^{n-1}, v) - \rho(\operator{D} e_X^{n-1}, 
        \operator{D} v)_Q
    \end{align}
    for $n \geq 2$ and all $v \in \discrete{V}$. In operator form, we obtain
    \begin{align*}
        e_X^n = \left( \operator{I}_{\discrete{V}} - \rho 
        \operator{A}_{1/\lambda}^{-1} \operator{D}^\star \operator{D} \right) 
        e_X^{n-1}, \qquad n \geq 2.
    \end{align*}
        \Cref{eqn:iter-penalty-2param-error-prop-u} now follows on noting that 
        $u_{X, \lambda, \rho}^1 = u_{X, 1/\lambda}$.
\end{proof}

We now look at the special case that $\rho = \lambda$.
\begin{lemma}
    \label{lem:ip-error-prop-rho-lambda-same}
    Let $\lambda > 1/\epsilon_0$. Then, 
    $\operator{A}_{1/\lambda}^{-1} \operator{D}^\star \operator{D} : 
    \discrete{V} \to \tilde{\discrete{Z}}$ and $\operator{D} 
    \operator{A}_{1/\lambda}^{-1} \operator{D}^{\star} : \discrete{Q} \to 
    \discrete{Q}$, and there holds
    \begin{subequations}
        \label{eqn:error-prop-continuity}
        \begin{alignat}{2}
            \label{eqn:error-prop-continuity-u}
            \| (\operator{I}_{\discrete{V}} - \lambda 
            \operator{A}_{1/\lambda}^{-1} \operator{D}^\star \operator{D}) 
            \tilde{z} \|_{V} &\leq \frac{M_a}{\lambda} \min \left\{ 
            C_{X, 1/\lambda}^{[2]}, 
            \frac{1}{\beta_X} C_{X,1/\lambda}^{[3]} \right\} 
            \|\tilde{z}\|_V \qquad & &\forall 
            \tilde{z} \in \tilde{\discrete{Z}}, \\
            \label{eqn:error-prop-continuity-p}
            \| ( \operator{I}_{\discrete{Q}} - \lambda \operator{D} 
            \operator{A}_{1/\lambda}^{-1} \operator{D}^{\star} ) q 
            \|_Q &\leq \frac{1}{\lambda} \min\left\{ \frac{M_a}{\beta_X} 
            C_{X,1/\lambda}^{[2,3]}, C_{X,1/\lambda}^{[4]} \right\} \|q\|_Q 
            \qquad & &\forall q \in 
            \discrete{Q},
        \end{alignat}
    \end{subequations}
    where
    $C_{X,1/\lambda}^{[2,3]} = \min\{ M_{\operator{D}} 
    C_{X,1/\lambda}^{[2]}, C_{X,1/\lambda}^{[3]} \}$.
\end{lemma}
\begin{proof}
    For $\epsilon \geq 0$ as in \cref{thm:saddle-discrete-continuity}, let
    $\operator{S}_{\epsilon, \discrete{V}} : \dual{\discrete{V}} \times 
    \dual{\discrete{Q}} \to \discrete{V}$ and $\operator{S}_{\epsilon, 
        \discrete{Q}} : \dual{\discrete{V}} \times 
    \dual{\discrete{Q}} \to \discrete{Q}$ be defined by
    \begin{align*}
        \begin{pmatrix}
            \operator{S}_{\epsilon, \discrete{V}}(F, G)
            \\
            \operator{S}_{\epsilon, \discrete{Q}}(F, G)	
        \end{pmatrix}
        := 
        \begin{pmatrix}
            \operator{A} & \operator{D}^{\star} \\
            \operator{R}_{\discrete{Q}} \operator{D} & -\epsilon 
            \operator{R}_{\discrete{Q}}
        \end{pmatrix}^{-1} \begin{pmatrix}
            F \\ G
        \end{pmatrix} \qquad \forall F \in \dual{\discrete{V}}, \ \forall G 
        \in \dual{\discrete{Q}}.
    \end{align*}
    First note that
    \begin{align*}
        \operator{S}_{0, \discrete{V}}(0, \operator{R}_{\discrete{Q}} 
        \operator{D} \tilde{z}) = \tilde{z} \quad \forall \tilde{z} \in 
        \tilde{\discrete{Z}} 
        \quad \text{and} \quad
        \operator{S}_{0, \discrete{Q}}(\operator{D}^{\star} q, 0) 
        = 
        q \quad \forall q \in \discrete{Q}.
    \end{align*}     
    
    \noindent \textbf{Step 1: Inequality 
    \cref{eqn:error-prop-continuity-u}. } 
    \Cref{eqn:saddle-discrete-penalty-1} gives
    \begin{align*}
    \operator{S}_{1/\lambda, \discrete{V}}(F, G) = 
    \operator{A}_{1/\lambda}^{-1} (F + \lambda 
    \operator{D}^{\star} \operator{R}_{\discrete{Q}}^{-1} G),
    \end{align*}
    and so for any $v \in \discrete{V}$, $\lambda\operator{A}_{1/\lambda} 
    \operator{D}^{\star} \operator{D} v = 
    S_{1/\lambda, \discrete{V}}(0, \operator{R}_{\discrete{Q}} \operator{D} 
    v ) \in \tilde{\discrete{Z}}$ by \cref{thm:saddle-discrete-continuity}.

    Now let $\tilde{z} \in \tilde{\discrete{Z}}$. Then, there holds	
    \begin{align*}
        (\operator{I}_{\discrete{V}} - \lambda 
        \operator{A}_{1/\lambda}^{-1} \operator{D}^{\star} \operator{D})
        \tilde{z} 
        = 
        (\operator{S}_{0, \discrete{V}} - \operator{S}_{1/\lambda, 
            \discrete{V}})(0, \operator{R}_{\discrete{Q}} 
        \operator{D} \tilde{z}).
    \end{align*}
    Thanks to \cref{lem:saddle-perturb-error}, we have
    \begin{align*}
        \|(\operator{S}_{0, \discrete{V}} - \operator{S}_{1/\lambda, 
            \discrete{V}})(0, \operator{R}_{\discrete{Q}} 
        \operator{D} \tilde{z})\|_V \leq \frac{1}{\lambda} \min\left\{ 
        M_a C_{X, 1/\lambda}^{[2]}  \| \tilde{z} \|_V, 
        C_{X,1/\lambda}^{[2,3]} \| p_{\tilde{z}} \|_Q \right\},
    \end{align*}
    where 
    $p_{\tilde{z}} := \operator{S}_{0, \discrete{Q}}(0, 
    \operator{R}_{\discrete{Q}} \operator{D} \tilde{z})$. 
    The inf-sup condition \cref{eqn:b-inf-sup-discrete} and 
    \cref{eqn:saddle-discrete-1} give
    \begin{align*}
        \beta_X \|p_{\tilde{z}}\|_Q \leq \sup_{ v \in \discrete{V} } 
        \frac{(\operator{D} v, p_{\tilde{z}})_Q}{\|v\|_V} = \sup_{ v \in 
        \discrete{V} } \frac{a(\tilde{z}, v)}{\|v\|_V} \leq 
        \frac{M_a}{\beta_X} \|\tilde{z}\|_V,
    \end{align*}
    and so
    \begin{align*}
        \|(\operator{S}_{0, \discrete{V}} - \operator{S}_{1/\lambda, 
            \discrete{V}})(0, \operator{R}_{\discrete{Q}} 
        \operator{D} \tilde{z})\|_V \leq \frac{M_a}{\lambda} \min \left\{ 
            C_{X, 1/\lambda}^{[2]}, 
            \frac{1}{\beta_X} C_{X,1/\lambda}^{[3]} \right\} \| \tilde{z} \|_V.
    \end{align*}        
    
    \noindent \textbf{Step 2: Inequality 
    \cref{eqn:error-prop-continuity-p}. } 
    \Cref{eqn:saddle-discrete-penalty-decoupled-2} gives
    \begin{align*}
    \operator{S}_{\epsilon, \discrete{Q}}(F, G) &= \epsilon^{-1} \left( 
    \operator{D} \operator{A}_{\epsilon}^{-1}(F + \epsilon^{-1} 
    \operator{D}^{\star} \operator{R}_{\discrete{Q}}^{-1} G ) - 
    \operator{R}_{\discrete{Q}}^{-1} G \right) \\
    &= \epsilon^{-1} \operator{D} \operator{A}_{\epsilon}^{-1} F + 
    \epsilon^{-1} \left( \epsilon^{-1}  
    \operator{D} \operator{A}_{\epsilon}^{-1} \operator{D}^{\star}  - 
    \operator{I}_{\discrete{Q}} \right) 
    \operator{R}_{\discrete{Q}}^{-1} G,
    \end{align*}
    and so for $q \in \discrete{Q}$, there holds
    \begin{align*}
    (\operator{I}_{\discrete{Q}} -
    \lambda  
    \operator{D} \operator{A}_{1/\lambda}^{-1} \operator{D}^{\star})q = 
    -\frac{1}{\lambda} \operator{S}_{1/\lambda, \discrete{Q}}(0, 
    \operator{R}_{\discrete{Q}} 
    q) = (\operator{S}_{0, \discrete{Q}} - 
    \operator{S}_{1/\lambda, \discrete{Q}}) (\operator{D}^{\star} 
    q, 0).
    \end{align*}
    Applying \cref{thm:saddle-discrete-continuity} gives
    \begin{align*}
        \|(\operator{I}_{\discrete{Q}} -
        \lambda  
        \operator{D} \operator{A}_{1/\lambda}^{-1} 
        \operator{D}^{\star})q\|_Q \leq \frac{1}{\lambda} C_{X, 
        1/\lambda}^{[4]} \|q\|_Q,
    \end{align*}
    while \cref{lem:saddle-perturb-error} gives
    \begin{align*}
        \|(\operator{I}_{\discrete{Q}} -
        \lambda  
        \operator{D} \operator{A}_{1/\lambda}^{-1} 
        \operator{D}^{\star})q\|_Q \leq \min\left\{ \frac{C_{X, 
                1/\lambda}^{[4]}}{\lambda}  \|q\|_Q, \frac{M_a}{\beta_X} \| 
        (\operator{S}_{0, \discrete{V}} - 
        \operator{S}_{1/\lambda, \discrete{V}}) (\operator{D}^{\star} 
        q, 0) \|_V \right\}.
    \end{align*}
    Note that $S_{0, \discrete{V}}(\operator{D}^{\star} q, 0) = 0$ since 
    $\operator{D}^{\star} q$ vanishes on $\discrete{Z}$, and so 
    \cref{lem:saddle-perturb-error} gives
    \begin{align*}
        \| (\operator{S}_{0, \discrete{V}} - 
        \operator{S}_{1/\lambda, \discrete{V}}) (\operator{D}^{\star} 
        q, 0) \|_V \leq \frac{1}{\lambda} \min\left\{  
        C_{X,1/\lambda}^{[2]} \|\operator{D}^{\star} 
        q\|_{\dual{\hat{\discrete{Z}}}}, C_{X,1/\lambda}^{[2,3]} \|q\|_Q  
        \right\}.
    \end{align*}
    Since $\|D^{\star} q\|_{\dual{\check{\discrete{Z}}}} \leq 
    M_{\operator{D}} \|q\|_Q$, we obtain
    \begin{align*}
        \| (\operator{S}_{0, \discrete{V}} - 
        \operator{S}_{1/\lambda, \discrete{V}}) (\operator{D}^{\star} 
        q, 0) \|_V \leq \frac{1}{\lambda} C_{X,1/\lambda}^{[2,3]} \|q\|_Q.
    \end{align*}
\end{proof}			

Based on Lemma \ref{lem:ip-error-prop-rho-lambda-same},
we derive convergence of the iterated penalty 
method \cref{eqn:iter-penalty-2param} for the case $\lambda \neq \rho$.
\begin{theorem}
    \label{thm:iter-penalty-2param-error}		
    Let $\lambda, \rho > 0$, $u_X, u_{X, \lambda, \rho}^{n} \in 
    \discrete{V}$, and $p_X, p_{X, \lambda, \rho}^n \in \discrete{Q}$ be 
    as in \cref{lem:iter-penalty-2param-error-prop}. Define
    \begin{align}
        \label{eqn:iter-penalty-2param-error-eta-tau}
        \eta_{\lambda, \rho} := \left| 1 - \frac{\rho}{\lambda} \right| 
        \quad \text{and} \quad \tau_{X, \lambda, \rho} := \frac{\rho 
        M_a}{\lambda^2} 
        \min \left\{ C_{X, 1/\lambda}^{[2]}, 
        \frac{1}{\beta_X} C_{X,1/\lambda}^{[3]} \right\} 
        + \eta_{\lambda, \rho}.
    \end{align}
    Then, there holds
    \begin{subequations}
        \label{eqn:iter-penalty-2param-error-e1-agen}
        \begin{align}
            \label{eqn:iter-penalty-2param-u-error-e1-agen}
            \|u_X - u_{X, \lambda, \rho}^n\|_V 
            &\leq \tau_{X, \lambda, \rho}^{n-1} 
                \|u_X - u_{X, 1/\lambda}\|_V, \\
            \label{eqn:iter-penalty-2param-p-error-e1-agen}
            \|p_X - p_{X, \lambda, \rho}^n\|_Q 
            &\leq \frac{\rho M_a}{\lambda \beta_X} 
            \sum_{i=1}^{n}  \eta_{\lambda, \rho}^{n-i} 
                \tau_{X, \lambda, \rho}^{i-1} \|u_X - u_{X, 1/\lambda}\|_V
            + \eta_{\lambda, \rho}^n \|p_X\|_Q.
        \end{align}	
    \end{subequations}
    Alternatively, if we define
    \begin{align}
        \label{eqn:iter-penalty-2param-error-xi}
        \xi_{X, \lambda, \rho} := \frac{\rho}{\lambda^2} \min\left\{ 
        \frac{M_a}{\beta_X} 
        C_{X,1/\lambda}^{[2,3]}, C_{X,1/\lambda}^{[4]} \right\} + 
        \eta_{\lambda, \rho},
    \end{align}
    where $C_{X,1/\lambda}^{[2,3]} = \min\{ M_{\operator{D}} 
    C_{X,1/\lambda}^{[2]}, C_{X,1/\lambda}^{[3]} \}$,
    then there holds
    \begin{subequations}
    \label{eqn:iter-penalty-2param-error-px-agen}
    \begin{align}
        \label{eqn:iter-penalty-2param-u-error-px-agen}
        \|u_X - u_{X, \lambda, \rho}^n\|_V &\leq \frac{1}{\lambda} 
        C_{X,1/\lambda}^{[2,3]} \xi_{X, \lambda, \rho}^{n-1} \|p_X\|_Q,  \\
        \label{eqn:iter-penalty-2param-p-error-px-agen}
        \|p_X - p_{X, \lambda, \rho}^n\|_Q &\leq \xi_{X, \lambda, \rho}^n 
        \|p_X\|_{Q}.
    \end{align}	
    \end{subequations}
\end{theorem}
\begin{proof}
    \noindent \textbf{Step 1: Inequality 
    \cref{eqn:iter-penalty-2param-u-error-e1-agen}. } Applying the 
    identity		
    \begin{align*}
        \operator{I}_{\discrete{V}} - \rho 
        \operator{A}_{1/\lambda}^{-1} \operator{D}^* \operator{D} = 
        \frac{\rho}{\lambda} \left( \operator{I} - \lambda 
        \operator{A}_{1/\lambda}^{-1} \operator{D}^* \operator{D}\right) + 
        \left( 
        1 - \frac{\rho}{\lambda} \right) \operator{I}_{\discrete{V}}
    \end{align*}
    gives
    \begin{align*}
        \| (\operator{I}_{\discrete{V}} - \rho 
        \operator{A}_{1/\lambda}^{-1} \operator{D}^* 
        \operator{D})\tilde{z}\|_V &\leq  \tau_{X, \lambda, \rho} 
        \|\tilde{z}\|_V \qquad \forall \tilde{z} \in \tilde{\discrete{Z}}.
    \end{align*}
    Inequality \cref{eqn:iter-penalty-2param-u-error-e1-agen} now follows 
    from \cref{eqn:error-prop-continuity-u}.
    
    \noindent \textbf{Step 2: Inequality 
    \cref{eqn:iter-penalty-2param-p-error-e1-agen}. }  
    \Cref{eqn:proof:iter-penalty-2param-key-ids} gives
    \begin{align*}
        (r^n_X, \operator{D} v)_Q 
        &= (r_X^{n-1}, \operator{D} v)_Q 
            + \rho (\operator{D} e_X^n, \operator{D} v)_Q 
        = -\frac{\rho}{\lambda} a(e_X^n, v) 
            + \left( 1 - \frac{\rho}{\lambda} \right) 
                (r_X^{n-1}, \operator{D} v)_Q,
    \end{align*}
    where we recall that $e_X^n := u_X - u_{X, \lambda, \rho}^n$ and 
    $r_X^n := p_X - p_{X, \lambda, \rho}^n$. 
    Thanks to \cref{eqn:b-inf-sup-discrete}, there exists $w \in 
    \discrete{V}$ such that $\operator{D} w = r_X^n$ and $\|w\|_{V} \leq 
    \beta_X^{-1} \|r_X^n\|_Q$, and so choosing $v = w$ above gives
    \begin{align}
        \label{eqn:proof:p-error-by-e-error}
        \|r_X^n\|_Q \leq \frac{\rho M_a}{\lambda \beta_X} \|e_X^n\|_V 
            + \left| 1 - \frac{\rho}{\lambda} \right| \|r_X^{n-1}\|_Q.
    \end{align} 
    By induction, we obtain
    \begin{align*}
        \|r_X^n\|_Q \leq \frac{\rho M_a}{\lambda \beta_X} 
        \sum_{i=1}^{n}  
        \left| 
        1 - \frac{\rho}{\lambda} \right|^{n-i} \|e_X^i\|_V 
        + \left| 
        1 - \frac{\rho}{\lambda} \right|^n \|r_X^0\|_Q,
    \end{align*}
    and \cref{eqn:iter-penalty-2param-p-error-e1-agen} now follows on 
    recalling that $r_X^0 = p_X$.
    
    \noindent \textbf{Step 3: 
    \cref{eqn:iter-penalty-2param-p-error-px-agen}. } Inequality 
    \cref{eqn:iter-penalty-2param-p-error-px-agen} follows from analogous 
    arguments as the proof of 
    \cref{eqn:iter-penalty-2param-u-error-e1-agen} on using 
    \cref{eqn:error-prop-continuity-p} in place of 
    \cref{eqn:error-prop-continuity-u}.
    
    \noindent \textbf{Step 4: 
    \cref{eqn:iter-penalty-2param-u-error-px-agen}. } Recall from
    \cref{eqn:proof:iter-penalty-2param-error-op-ids} that $e_X^n = 
    -\operator{A}_{1/\lambda}^{-1} \operator{D}^{\star} r_X^{n-1}$. 
    Taking $F = \operator{D}^{\star} r_X^{n-1}$ and $G = 0$ or 
    $F = 0$ and $G = \lambda^{-1} \operator{R}_{\discrete{Q}} r_X^{n-1}$  
    in \cref{thm:saddle-discrete-continuity} gives
    \begin{align*}
        \|e_X^n\|_V \leq \frac{1}{\lambda} M_{\operator{D}} C_{X, 
        1/\lambda}^{[2]} \|r_X^{n-1}\|_Q
        \quad \text{and} \quad \|e_X^n\|_V \leq \frac{1}{\lambda} C_{X, 
        1/\lambda}^{[3]} \|r_X^{n-1}\|_Q. 
    \end{align*}
    Inequality \cref{eqn:iter-penalty-2param-u-error-px-agen} now follows.
\end{proof}

Note that
\begin{align*}
    \sum_{i=1}^{n} \eta_{\lambda, \rho}^{n-i} 
    \tau_{X, \lambda, \rho}^{i-1} 
    = \frac{\tau_{X, \lambda, \rho}(\tau_{X, \lambda, \rho}^n 
        - \eta_{\lambda, \rho}^n)}{\tau_{X, \lambda, \rho} 
        - \eta_{\lambda, \rho}},
\end{align*}
and so \cref{thm:iter-penalty-2param-error} shows that if $\eta_{\lambda, 
\rho} < 1$ and either $\xi_{X, \lambda, \rho} < 1$ 
or $\tau_{X, \lambda, \rho} < 1$, then the iterated penalty 
method converges at a geometric rate. A common choice in practice is 
$\lambda = \rho$, and the next result expands these conditions in this 
case.		
    
\begin{corollary}
    \label{cor:iter-penalty-conv}
    Let $\rho \in \mathbb{R}$, $u_X, u_{X, \rho, \rho}^n \in 
    \discrete{V}$, 
    and $p_X, p_{X, \rho, \rho}^n \in \discrete{Q}$, $n \in \mathbb{N}$ 
    be as 
    in \cref{thm:iter-penalty-2param-error}. Define
    \begin{align}
        \label{eqn:rho0-def}
        \rho_0 := \frac{1}{\epsilon_0} \left(1 + 
        \frac{1}{\Upsilon_X}\right) 
        = \frac{(M_a + \alpha_X)(M_a + 2\alpha_X)M_a }{ \alpha_X^2 
        \beta_X^2 }.
    \end{align}
    If $\rho > \rho_0$, then the iterated penalty method 
    \cref{eqn:iter-penalty-2param} with $\lambda = \rho$ converges 
    geometrically with a 
    rate given by
    \begin{align}
        \label{eqn:iter-penalty-geo-rate-conv}
        \frac{1}{\Upsilon_X(\rho \epsilon_0 - 
        1)}
        = \frac{\rho_0}{\Upsilon_X(\rho - \rho_0) + \rho} 
        = \frac{M_a \alpha_X(\alpha_X + M_a)}{\rho \alpha_X^2 \beta_X^2 - 
        M_a(\alpha_X + M_a)^2}.
    \end{align}
    If \ref{asmp:a-sym-nonneg-coercive-z} or \ref{asmp:a-coercive-all} 
    holds, then for all $\rho > 0$ the 
    iterated penalty method \cref{eqn:iter-penalty-2param} 
    with $\lambda = \rho$ converges geometrically with a rate given by
    \begin{align}
        \label{eqn:iter-penalty-geo-rate-conv-asspd}
        \frac{M_a}{M_a + \rho \beta_X^2} \quad \text{if 
        \ref{asmp:a-sym-nonneg-coercive-z} holds} 
        \quad \text{or} \quad
        \frac{M_a \Upsilon_X}{M_a \Upsilon_X + \rho \beta_X^2} \quad 
        \text{if \ref{asmp:a-coercive-all} holds}.
    \end{align}
\end{corollary}
\begin{proof} 
    For general $a(\cdot,\cdot)$, we have	
    \begin{align*}
        C_{X, 1/\rho}^{[4]} = \frac{M_a}{\beta_X} C_{X, 1/\rho}^{[3]} = 
        \frac{M_a}{\Upsilon_X} C_{X, 1/\rho}^{[2]} \leq M_a C_{X, 
            1/\rho}^{[2]} \leq \frac{M_a M_{\operator{D}}}{\beta_X}
        C_{X, 1/\rho}^{[2]},
    \end{align*}
    and so 
    $\xi_{X, \rho, \rho} = C_{X,1/\rho}^{[4]} / \rho 
    \leq \tau_{X, \rho, \rho}$. Consequently, 
    \cref{thm:iter-penalty-2param-error} ensures that the iterated penalty
    method converges geometrically provided that 
    $C_{X,1/\rho}^{[4]} < \rho$ at a rate $C_{X, 1/\rho}^{[4]} / \rho$,
    and a better rate cannot be obtained with 
    \cref{thm:iter-penalty-2param-error}. This condition reads
    \begin{align*}
        \frac{M_a \Upsilon_X}{\beta_X^2} \left( 
        \frac{\epsilon_0}{\epsilon_0 - \rho^{-1}} \right) < \rho \iff 
        \rho > \rho_0 \quad \text{and} \quad \frac{C_{X, 
                1/\rho}^{[4]}}{\rho} = \frac{\rho_0}{\Upsilon_X(\rho - 
                \rho_0) 
            + \rho}.
    \end{align*}
    If $a(\cdot,\cdot)$ satisfies \ref{asmp:a-sym-nonneg-coercive-z}
    or \ref{asmp:a-coercive-all}, then $\xi_{X, \rho, \rho} \leq 
    C_{X,1/\rho}^{[4]}/\rho$, and so
    \begin{align*}
        \xi_{X, \rho, \rho} \leq \frac{M_a}{M_a + \rho \beta_X^2} 
        \quad \text{if \ref{asmp:a-sym-nonneg-coercive-z} holds}
        \quad \text{or} \quad
        \xi_{X, \rho, \rho} \leq \frac{M_a \Upsilon_X}{M_a \Upsilon_X + 
        \rho \beta_X^2} \quad 
        \text{if \ref{asmp:a-coercive-all} holds}.
    \end{align*}
    In both cases, $\xi_{X, \rho, \rho} < 1$, which completes the proof.
\end{proof}

\begin{remark}
    Note that the constraint 
    $\rho > \rho_0 =\frac{1}{\epsilon_0}(1+\Upsilon_X^{-1})$ 
    required for convergence of the 
    iterated penalty method in \cref{cor:iter-penalty-conv} is slightly 
    more strict than the condition $\rho > 1/\epsilon_0$ for the 
    well-posedness of the penalized problem in 
    \cref{thm:saddle-discrete-continuity}.
\end{remark}
        
\begin{remark}
    The extra conditions 
    \ref{asmp:a-sym-nonneg-coercive-z}, \ref{asmp:a-coercive-all}, and
    \ref{asmp:a-sym-coercive-all} are non-exhaustive. For example, if 
    $\operator{A}$ is maximally rank deficient in the sense that
    $\operator{A} \tilde{\discrete{Z}} = \{ 0 \}$, then 
    \cref{eqn:iter-penalty-2param} is well-defined and converges in two 
    iterations for all $\lambda, \rho > 0$ as we now show.
    
    We first show that the iterates are well-defined. Suppose that
    $u \in \discrete{V}$ satisfies $a_{1/\lambda}(u, v) = 0$ for all 
    $v \in \discrete{V}$. According to \cref{lem:v-z-tildez-decomp} below,
    $u = z + \tilde{z}$ for some $z \in \discrete{Z}$ and 
    $\tilde{z} \in \tilde{\discrete{Z}}$. 		
    Taking $v \in \discrete{Z}$ and applying 
    \cref{eqn:a-infsup-kernel-discrete} shows that $z \equiv 0$. Thus,
    $\lambda \|\mathcal{D}\tilde{z}\|_{Q}^2 = 0$, and so $\tilde{z} \equiv 0$
    by \cref{lem:tilde-z-div-continuity} below. Thus, 
    $\operator{A}_{1/\lambda}$ is invertible for any $\lambda > 0$
    and the iterates \cref{eqn:iter-penalty-2param} are well-defined.
    
    We now turn to convergence. Clearly, the error of the first iterate 
    $e_{X}^1$ and $r_{X}^1$ satisfy 
    \cref{proof:eqn:sadlle-perturb-error-mixed} with 
    $\epsilon = \lambda^{-1}$. Since $e_{X}^1 \in \tilde{\discrete{Z}}$,
    \cref{proof:eqn:sadlle-perturb-error-mixed-1} shows that
    \begin{align*}
        (r_X^1,  \operator{D} v)_Q = -a(e_X^1, v) = 0 
            \qquad \forall v \in \discrete{V},
    \end{align*}
    and so $r_X^1 \equiv 0$ by \cref{eqn:b-inf-sup-discrete}. Thus,
    $e_X^{n+1} \equiv 0$ and $r_X^n \equiv 0$ for all $n \in \mathbb{N}$ by 
    \cref{eqn:proof:iter-penalty-2param-error-op-ids}. In particular,
    \cref{eqn:iter-penalty-2param} converges in exactly two iterations.
\end{remark}

\begin{remark}
    \label{rem:results-valid-v-q}
    All of the theoretical results in 
    \cref{sec:iterated-penalty,sec:penalty-method}
    and in \cref{sec:penalty-well-posedness-proofs} below remain 
    valid if we replace $\discrete{V}$, $\discrete{Q}$, 
    and $\discrete{Z}$ with their 
    possibly infinite-dimensional counterparts $V$, $Q$, and $Z$
    and if we replace the discrete inf-sup constants 
    $\alpha_X$ \cref{eqn:a-infsup-kernel-discrete} and 
    $\beta_X$ \cref{eqn:b-inf-sup-discrete} with the continuous counterparts
    $\alpha$ \cref{eqn:a-inf-sup-kernel} and 
    $\beta$ \cref{eqn:b-inf-sup}.
\end{remark}

\subsection{Stopping criteria}		

With the convergence of the iterated penalty method quantified by 
\cref{cor:iter-penalty-conv}, we now turn to a suitable stopping 
criterion. We will continue to focus on the case  $\lambda = \rho$, and
we simplify notation by setting 
$u_{X, \rho}^n := u_{X, \rho, \rho}^n$ 
and $p_{X, \rho}^n := p_{X, \rho, \rho}^n$.
One common option in the literature is to use the residual 
of the constraint \cref{eqn:saddle-discrete-2}:
\begin{align}
    \| \operator{D} u_{X, \rho}^{n} - \operator{R}_{\discrete{Q}}^{-1}G\|_Q 
    = \| \operator{D} u_{X, \rho}^{n} - \operator{D} u_{X}\|_Q 
    \leq M_{\operator{D}} \|u_{X, \rho}^{n} - u_X\|_{V},
\end{align}
which converges to zero at the same geometric rate in 
\cref{eqn:iter-penalty-geo-rate-conv,eqn:iter-penalty-geo-rate-conv-asspd}.
Thanks to 
\cref{thm:iter-penalty-2param-error}, ${u_X - u_{X, \rho}^{n} \in 
\tilde{\discrete{Z}}}$, and so 
\cref{eqn:tilde-z-div-continuity} below and
\cref{eqn:proof:p-error-by-e-error} give
\begin{align}
    \label{eqn:iter-penalty-1param-error-by-residual}
    \|u_{X, \rho}^{n} - u_X\|_V &\leq \frac{\Upsilon_X}{\beta_X} \| 
    \operator{D} 
    u_{X, \rho}^{n} - \operator{R}_{\discrete{Q}}^{-1} G\|_Q
    \quad \text{and} \quad 
    \| p_{X, \rho}^{n} - p_X \|_Q \leq 
    \frac{M_a \Upsilon_X}{\beta_X^2} \| \operator{D} 
    u_{X, \rho}^{n} - \operator{R}_{\discrete{Q}}^{-1} G\|_Q.
\end{align}	
Thus, provided $M_a \Upsilon_X/\beta_X^2$ is not too large, the stopping 
criterion
\begin{align}
    \label{eqn:residual-stopping-criteria}
    \| \operator{D} u_{X, \rho}^{n} - 
    \operator{R}_{\discrete{Q}}^{-1} G\|_Q < \texttt{tol}
\end{align}
for a prescribed 
tolerance $\texttt{tol}$ is reasonable. We shall adopt this stopping 
criterion for the remainder of the manuscript. Other possible choices for 
stopping criteria could be 
$\|u_{X, \rho}^{n} - u_{X, \rho}^{n-1}\|_V < \texttt{tol}$ and/or
$\|p_{X, \rho}^{n} - p_{X, \rho}^{n-1}\|_Q < \texttt{tol}$.

\subsection{Implementation}

Let $\rho > 0$ be as in \cref{cor:iter-penalty-conv} so that  
\cref{eqn:iter-penalty-2param} (with $\lambda=\rho$) is well-posed and 
converges to the 
solution to \cref{eqn:saddle-discrete}. At first 
glance, it is not immediately apparent that the iterated penalty method
\cref{eqn:iter-penalty-2param} can be implemented
more easily than the original mixed problem \cref{eqn:saddle-discrete}.
Most notably, we require an iterate $p_{X, \rho}^{n} \in \discrete{Q}$, 
but we generally do not have a basis for $\discrete{Q}$. Further 
complicating matters is that for $n \in \mathbb{N}$, computing $p_{X, 
\rho}^{n}$ and the stopping criterion 
\cref{eqn:residual-stopping-criteria} requires 
$\operator{R}_{\discrete{Q}}^{-1} G$, which is also 
generally not available without a basis for $\discrete{Q}$. Thus, we 
seek to eliminate the apparent need for a basis for	$\discrete{Q}$.

\begin{lemma}
    \label{lem:iter-penalty-u-computable}
    Consider the following algorithm. Let $\rho > 0$ and $\discrete{V} \ni 
    w_{X, \rho}^{0} := 0$. For $n \in \mathbb{N}$, define $\bar{u}_{X, 
        \rho}^{n}, w_{X, \rho}^n \in \discrete{V}$ by		
    \begin{subequations}
        \label{eqn:iter-penalty-decoupled-computable}
        \begin{alignat}{2}
            \label{eqn:iter-penalty-decoupled-computable-1}
            a(\bar{u}_{X, \rho}^{n}, v) 
            + \rho (\operator{D} \bar{u}_{X, \rho}^{n}, \operator{D} v)_Q 
            &= F(v) 
            - (\operator{D} v, \operator{D} w_{X, \rho}^{n-1})_Q 
            + n \rho G(\operator{D} v) 
            \qquad 
            & &\forall v \in \discrete{V}, \\
            \label{eqn:iter-penalty-decoupled-computable-2}
            w_{X, \rho}^{n} &= w_{X, \rho}^{n-1} 
            + \rho \bar{u}_{X, \rho}^{n}. 
            \qquad & &
        \end{alignat}
    \end{subequations}	
    Then, $u_{X, \rho}^{n} = \bar{u}_{X, \rho}^{n}$ and 
    $p_{X, \rho}^{n} =  \operator{D} w_{X, \rho}^n - n\rho 
    \operator{R}_{\discrete{Q}}^{-1} G$ for all $n \in \discrete{N}$.
\end{lemma}	
\begin{proof}
    We proceed by induction. For $n=1$, we have
    \begin{align*}
        a(\bar{u}_{X, \rho}^{1}, v) 
        + \rho (\operator{D} \bar{u}_{X, \rho}^{1}, \operator{D} v)_Q 
        &= F(v) + \rho G(\operator{D} v) \qquad \forall v \in 
        \discrete{V},  
    \end{align*}
    and so $u_{X, \rho}^{1} = \bar{u}_{X, \rho}^{1}$ since 
    \cref{eqn:saddle-discrete-penalty-decoupled-1} is well-posed. Then,
    \begin{align*}
        \operator{D} w_{X, \rho}^{1} = \rho \operator{D} \bar{u}_{X, 
        \rho}^{1} = \rho \operator{D} u_{X, \rho}^{1} = p_{X, \rho}^{1} 
        + \rho \operator{R}_{\discrete{Q}}^{-1} G, 
    \end{align*}
    and so the result is true for $n=1$. Suppose it holds for some $n 
    \in \mathbb{N}$. Then,
    \begin{align*}
        a(\bar{u}_{X, \rho}^{n+1}, v) 
        + \rho (\operator{D} \bar{u}_{X, \rho}^{n+1}, \operator{D} v)_Q 
        &= F(v) 
        - (\operator{D} v, \operator{D} w_{X, \rho}^{n})_Q 
        + (n+1) \rho G(\operator{D} v) \\
        &=  F(v) 
        - (\operator{D} v, p_{X, \rho}^n)_Q 
        + \rho G(\operator{D} v)
    \end{align*}
    for all $v \in \discrete{V}$. By well-posedness, $\bar{u}_{X, 
    \rho}^{n+1} = u_{X, \rho}^{n+1}$. Moreover,
    \begin{align*}
        \operator{D} w_{X, \rho}^{n+1} 
        = \operator{D} w_{X, \rho}^{n} + 
        \rho \operator{D} \bar{u}_{X, \rho}^{n+1} 
        = p_{X, \rho}^{n} + \rho \operator{D} u_{X, \rho}^{n+1} + n \rho 
        \operator{R}_{\discrete{Q}}^{-1} G 
        = p_{X, \rho}^{n+1} + (n+1)\rho 
        \operator{R}_{\discrete{Q}}^{-1} G,
    \end{align*}
    which completes the proof.
\end{proof}
\noindent Thanks to \cref{lem:iter-penalty-u-computable}, we can 
compute $u_{X, \rho}^n$ by iterating 
\cref{eqn:iter-penalty-decoupled-computable} which only requires a 
basis for $\discrete{V}$. Of course, we are implicitly assuming that we 
can compute $(\operator{D} u, \operator{D} v)_{Q}$ for all 
$u, v \in \discrete{V}$, which is the case for all of the examples in 
\cref{sec:key-examples}.

However, computing $p_{X, \rho}^n$ and the stopping criterion 
\cref{eqn:residual-stopping-criteria} requires 
$\operator{R}_{\discrete{Q}}^{-1} G$. If $G = 0$, then
${\operator{R}_{\discrete{Q}}^{-1} G = 0}$ and 
\cref{lem:iter-penalty-u-computable} shows that $p_{X, \rho}^{n} = 
\operator{D} w_{X, \rho}^n$. Thus, $\operator{D} w_{X, \rho}^n \approx 
p_X$ is fully computable, as is the stopping criterion $\|\operator{D} 
u_{X, \rho}^n\|_{Q} < \texttt{tol}$. 
If $G \neq 0$, then we can consider the 
following discrete system: 	Find $\zeta_X \in \discrete{V}$ and $g_X 
\in \discrete{Q}$ such that
\begin{subequations}
    \label{eqn:q-inverse-reisz-by-saddle}
    \begin{alignat}{2}
        (\zeta_X, v)_V + (\operator{D} v, g_X)_Q &= G(\operator{D} v) 
        \qquad & 
        &\forall v \in \discrete{V}, \\
        \label{eqn:q-inverse-reisz-by-saddle-2}
        (\operator{D} \zeta_X, q)_Q &= 0 \qquad & &\forall q \in 
        \discrete{Q}.
    \end{alignat}
\end{subequations}
Then, we see that $g_X = \mathcal{R}_{\discrete{Q}}^{-1} G$ by 
choosing $v \in \discrete{Z}^{\perp} := \{ v \in \discrete{V} : (v, 
    z)_V = 0 \ \forall z \in \discrete{Z} \}$ and noting $\operator{D} 
    \discrete{Z}^{\perp} = \discrete{Q}$. Consequently, we can 
    approximate $g_X$ with the iterated penalty method applied to 
    \cref{eqn:q-inverse-reisz-by-saddle}. In particular, the approximation 
    after $m$ iterations $g_{X, \rho}^{m}$ is of the form 
    $\operator{D} y_{X, \rho}^{m}$, 
    for some $y_{X, \rho}^{m} \in \discrete{V}$, 	   
    which is computable using \cref{eqn:iter-penalty-decoupled-computable} 
    since the right-hand side of 
    \cref{eqn:q-inverse-reisz-by-saddle-2} is zero. Then, we can use
    $g_{X, \rho}^m$ in place of $\mathcal{R}_{\discrete{Q}}^{-1} G$
    in \cref{lem:iter-penalty-u-computable}; i.e. we take
    $p_{X, \rho}^{n, m} := \operator{D} w_{X, \rho}^{n} - n \rho g_{X, 
    \rho}^m$  to be the approximation to $p_{X}$. Moreover, we can use 
    $\| \operator{D} u_{X, \rho}^n - g_{X, \rho}^m\|_Q < \texttt{tol}$ 
    as a sufficient approximation to the stopping criterion in 
    \cref{eqn:residual-stopping-criteria}.
    This procedure is summarized in \cref{alg:iter-penalty-computable}. 

\begin{algorithm}[htb]
    \caption{A computable iterated penalty method for 
        \cref{eqn:saddle-discrete}}
    \label{alg:iter-penalty-computable}
    \begin{algorithmic}[1]
        \Require{$\rho > 0$ and $\texttt{tol2} > \texttt{tol1} > 0$
            \vphantom{$\big[$}}

        \If{$G = 0$\vphantom{$[$}}
        
        \State{$u_{X, \rho}^n$, $w_{X,\rho}^{n}\gets$ 
            \Call{PartialIPM}{$a(\cdot,\cdot)$, $F(\cdot)$, $G(\cdot)$, 
            $\rho$, 0, 
            \texttt{tol2}} 
            \vphantom{$\bigg[$}}
        
        \State{\Return $u_{X, \rho}^n$ and $p_{X, 
                \rho}^{n} := \operator{D} w_{X, 
                \rho}^n$}
        
        \Else
        
        \State{$\zeta_{X, \rho}^{m}$, 
            $y_{X, \rho}^m \gets$
            \Call{PartialIPM}{%
                $(\cdot,\cdot)_V$, 
                $G(\operator{D}\cdot)$, 
                0, $\rho$, 0, \texttt{tol1}}\vphantom{$\big[$}}
            
        \State{$g_{X, \rho}^m \gets \operator{D} y_{X, \rho}^m$ 
        \vphantom{$\bigg[$}}
        
        \State{$u_{X, \rho}^n$, $w_{X,\rho}^{n}\gets$ 
            \Call{PartialIPM}{$a(\cdot,\cdot)$, $F(\cdot)$, $G(\cdot)$, 
            $\rho$, 
            $g_{X, \rho}^m$, \texttt{tol2}} 
            \vphantom{$\big[$}}
        
        \State{\Return $u_{X, \rho}^n$ and $\bar{p}_{X, 
                \rho}^{n, m}:= \operator{D} w_{X, 
                \rho}^n - n \rho g_{X, \rho}^m$ \vphantom{$\bigg[$}}
        
        \EndIf

        \Statex{}

        \Function{PartialIPM}{$\mathfrak{a}(\cdot,\cdot)$,
            $\mathcal{F}(\cdot)$, $\mathcal{G}(\cdot)$, $\rho$, $g_X$, 
            \texttt{tol}}
        \State{Set $\omega_{X, \rho}^0 := 0$.}
        \For{$n = 1, 2, \ldots, $}
        \State{Find $\mu_{X, \rho}^{n} \in \discrete{V}$ 
            such that}
        \State{$\mathfrak{a}(\mu_{X, \rho}^{n}, v) 
            + \rho (\operator{D} \mu_{X, \rho}^{n}, 
            \operator{D} v)_Q 
            = \mathcal{F}(v) 
            - (\operator{D} v, \operator{D} \omega_{X, 
                \rho}^{n-1})_Q 
            + n \rho \mathcal{G}(\operator{D} v) \quad \forall 
            v \in 
            \discrete{V}$ }
        \State{Set $\omega_{X, \rho}^n := \omega_{X, 
                \rho}^{n-1} + \rho \mu_{X, \rho}^{n}$} 
        \If{$\| \operator{D} \mu_{X, \rho}^{n} - g_X\|_Q < \texttt{tol}$ }
        \State{\Return $\mu_{X, \rho}^{n}$ and $\omega_{X, 
                \rho}^{n}$}
        \EndIf
        \EndFor
        \EndFunction	 		
    \end{algorithmic}
\end{algorithm}

\begin{remark}	
    Note that we have two stopping tolerances in 
    \cref{alg:iter-penalty-computable}: a tolerance for iterated 
    penalty applied to \cref{eqn:q-inverse-reisz-by-saddle} to compute 
    $\operator{R}_{\discrete{Q}}^{-1} G$, labeled \texttt{tol1},
    and a tolerance for iterated penalty applied to the original saddle 
    point system \cref{eqn:saddle-discrete}, labeled \texttt{tol2}. 
    If $G \neq 0$, then we have
    \begin{align*}
        \| \operator{D} u_{X, \rho}^n - g_{X, \rho}^m\|_Q 
        &\geq \left| \| \operator{D} u_{X, \rho}^n - 
        \operator{R}_{\discrete{Q}}^{-1} G\|_Q - \|g_{X, \rho}^m - 
        \operator{R}_{\discrete{Q}}^{-1} G\|_Q \right| 
        = \left| \| \operator{D} u_{X, \rho}^n - 
        \operator{R}_{\discrete{Q}}^{-1} G\|_Q - \texttt{tol1} \right|,
    \end{align*} 
    where the first term on the final line converges to zero. Thus, we 
    must choose $\texttt{tol2} > \texttt{tol1}$ for 
    \cref{alg:iter-penalty-computable} to terminate.
\end{remark}

\begin{remark}
    If $G(\cdot) = (g_X, \cdot)_Q$ for some readily available $g_X \in 
    \discrete{Q}$, then $g_X = \operator{R}_{\discrete{Q}}^{-1} G$ can be 
    used in place of $g_{X, \rho}^m$ and lines 5 and 6 of 
    \cref{alg:iter-penalty-computable} can be removed. 
\end{remark}

\section{Numerical examples}
\label{sec:numerics}

We now revisit the key examples in \cref{sec:key-examples} to demonstrate
the performance of the iterated penalty method. All of the examples are 
coded using Firedrake \cite{FiredrakeUserManual} and every linear system
encountered is solved with MUMPS \cite{amestoy2001} 
via PETSc \cite{balay2025}. The code is available at \cite{zenodo}.
    
\begin{remark}
    For large problems, particularly in 3D, direct solvers are not a 
    feasible option, and iterative solvers must be used for
    \cref{eqn:iter-penalty-2param-1}. 
    For $\lambda \gg 1$, the linear system \cref{eqn:iter-penalty-2param-1}
    is ill-conditioned and preconditioners are necessary. 
    Developing preconditioners for nearly singular problems like 
    \cref{eqn:iter-penalty-2param-1}
    is delicate \cite{lee2007,schoberl1999}.
    Preconditioners robust
    in the discretization parameters and $\lambda$ are available
    for the Hodge decomposition problem in \cref{sec:numerics-hodge} 
    \cite{ArnoldFalkWinther00,BrubeckFarrell24,Brubecketal25,Hiptmair97,Hiptmair98,%
        HiptmairSchiekoferWohlmuth96,HiptmairToselli00,PaznerKolevDohrmann23}
    and for some Scott--Vogelius discretizations of the incompressible flow
    problem in \cref{sec:numerics-incompressible} 
    \cite{AinsworthParker24lep,FarrellMitchellWechsung19,%
        FarrellMitchellScottWechsung22,LeeWuChen09}.
\end{remark}

\subsection{Hodge decompositions of a discrete de Rham complex}
\label{sec:numerics-hodge}

To demonstrate the effectiveness of \cref{alg:iter-penalty-computable}
with nonzero $G(\cdot)$ and $a(\cdot,\cdot)$ satisfying 
\ref{asmp:a-sym-nonneg-coercive-z}, 
we return to computing Hodge decompositions as described in 
\cref{sec:hodge}.
For $p \geq 1$, consider the following discrete de Rham complex:
\begin{equation}
    \label{eqn:discrete-de-rham}
    \begin{tikzcd}
        0 \arrow{r} & \mathrm{CG}_p  \arrow{r}{\grad} & \mathrm{Ned}_p^1 
        \arrow{r}{\curl} 
        & \mathrm{RT}_p \arrow{r}{\div} & \mathrm{DG}_{p-1} \arrow{r} & 0.
    \end{tikzcd}
\end{equation}
Here, $\mathrm{CG}_p$ is the standard space of continuous piecewise 
polynomials of degree $p$ equipped with the $H^1(\Omega)$ inner product; 
$\mathrm{Ned}_p^1$ is the space of degree $p$ N\'{e}d\'{e}lec elements of 
the first kind equipped with the $\hcurl$ inner product; $\mathrm{RT}_p$ is 
the space of degree $p$ Raviart--Thomas elements equipped with the $\hdiv$ 
inner product; and $\mathrm{DG}_{p-1}$ is the space of discontinuous 
piecewise polynomials of degree $p-1$ equipped with the $L^2(\Omega)$ inner
product.

We consider the problem of computing the Hodge decomposition of a vector 
field in $\mathrm{Ned}_p^1$ and in $\mathrm{RT}_p$, $p \in \{1,2,3\}$, as 
defined in \cref{eqn:hodge-decomp}. For vector fields in 
$\mathrm{Ned}_p^1$, we use a coarse 96 element tetrahedral mesh of a box 
with two tunnels: 
\begin{align*}
    \Omega_{\mathrm{Ned}} := (0, 2) \times (0, 1)^2 \setminus \left( [1/4, 
    3/4]^2 \times [0,1] 
    \cup [5/4, 7/4]^2 \times [0, 1]  \right).
\end{align*}
For vector fields in $\mathrm{RT}_p$, we use a coarse 145 element 
tetrahedral mesh of a box with two voids:
\begin{align*}
    \Omega_{\mathrm{RT}} := (0, 2) \times (0, 1)^2 \setminus \left( [1/4, 
    3/4]^3  
    \cup [5/4, 7/4]^2 \times [1/4, 3/4]  \right).
\end{align*}
The domains are chosen so that the first Betti number of 
$\Omega_{\mathrm{Ned}}$ and the second Betti number of 
$\Omega_{\mathrm{RT}}$ is 2. Thus, the dimension of the space of 
harmonic forms in $\mathrm{Ned}_p^1$ on $\Omega_{\mathrm{Ned}}$ and in 
$\mathrm{RT}_p$ on $\Omega_{\mathrm{RT}}$ is 2 
(see e.g.~\cite[Section 4.3]{Arnold18}).
For each domain and polynomial degree, we generate a random vector field by
choosing the degrees of freedom to be uniformly distributed in 
$[-0.5, 0.5]$. We have verified numerically that the resulting fields have 
a nontrivial harmonic form component.

We compute the Hodge decomposition as described in 
\cref{sec:hodge}. In particular, we solve the two saddle point problems 
\cref{eqn:hodge-potential-mixed,eqn:hodge-dk-mixed}. Both of these 
saddle point problems satisfy \ref{asmp:a-sym-nonneg-coercive-z} with 
$\tilde{\alpha}_X = M_a = 1$ and $M_{\operator{D}} = 1$. 
We choose $\texttt{tol1} = 2 \cdot 10^{-12}$, $\texttt{tol2} = 
10^{-11}$, and $\rho = 10^3$ in \cref{alg:iter-penalty-computable}.
We only compute $\operator{R}_{\discrete{Q}}^{-1} G$ for 
\cref{eqn:hodge-potential-mixed}, as 
$\operator{R}_{\discrete{Q}}^{-1} G = \dee^k u$ in 
\cref{eqn:hodge-dk-mixed} is already available.
The results are displayed in \cref{tab:hodge-results}. We see 
that the number of iterations of the iterated penalty method solves 
in lines 5 and 7 of \cref{alg:iter-penalty-computable} never exceeds 5 
iterations, which is consistent with the geometric rate of convergence 
guaranteed by \cref{cor:iter-penalty-conv}. Moreover, we observe in the 
final three columns of \cref{tab:hodge-results} that the computed 
components of the Hodge decomposition are $L^2(\Omega)$, and hence 
$(\cdot,\cdot)_{\discrete{V}_k}$, orthogonal up to the solver tolerances.

\begin{table}[htb]
    \centering
    \caption{Numerical results for the Hodge decomposition problem of 
    vector fields in the discrete de Rham complex 
    \cref{eqn:discrete-de-rham}. The ``$\sigma_{\perp}$ iters.'' columns
    show the iterated penalty iteration counts for solving 
    \cref{eqn:hodge-potential-mixed}, with the 
    left (right) subcolumn indicating the number of iterations for line 5 
    (line 7) of \cref{alg:iter-penalty-computable}. The 
    ``$u_{\perp}$ iters.'' column shows the iteration counts for line 7 of 
    \cref{alg:iter-penalty-computable} applied to \cref{eqn:hodge-dk-mixed}. 
    The final three columns record the $L^2(\Omega)$ inner product of 
    the three components of the computed Hodge decomposition, where we use 
    the notation $\dee := \dee^{k-1}$. Note that these three $L^2(\Omega)$ 
    inner products are analytically equivalent to  the corresponding 
    $(\cdot,\cdot)_{\discrete{V}^k}$ inner product since
    at least one argument is $\dee^k$-free.}
    \label{tab:hodge-results}
    \begin{tabular*}{\textwidth}{@{\extracolsep\fill} c  c | c c | c | c c c @{\extracolsep\fill}}
        Space & $p$ & \multicolumn{2}{c|}{$\sigma_{\perp}$ iters.} & 
        $u_{\perp}$ iters. & $(\harmonic{h}, \dee 
        \sigma_{\perp})_{L^2(\Omega)}$ & 
        $(\harmonic{h}, u_{\perp})_{L^2(\Omega)}$ & 
        $(\dee \sigma_{\perp}, u_{\perp})_{L^2(\Omega)}$ \\
        \hline
        \multirow{3}{*}{$\mathrm{Ned}_p^1$} 
        & 1 & 4 & 3 & 3 & $3.196 \times 10^{-13}$ & $6.242 \times 10^{-14}$ 
                        & $4.342 \times 10^{-16}$ \rule{0pt}{0.9\normalbaselineskip}  \\
        & 2 & 5 & 4 & 3 & $2.354 \times 10^{-15}$ & $3.514 \times 10^{-14}$ 
                        & $-2.675 \times 10^{-15}$ \\
        & 3 & 5 & 4 & 3 & $5.815 \times 10^{-15}$ & $2.700 \times 10^{-13}$ 
                        & $-1.118 \times 10^{-13}$ \\
        \hline
        \multirow{3}{*}{$\mathrm{RT}_p$} 
        & 1 & 5 & 3 & 3 & $4.022 \times 10^{-13}$ & $1.131 \times 10^{-14}$ 
                        & $-2.792 \times 10^{-15}$ \rule{0pt}{0.9\normalbaselineskip} \\
        & 2 & 5 & 3 & 3 & $1.208 \times 10^{-11}$ & $6.078 \times 10^{-14}$ 
                        & $-3.513 \times 10^{-13}$ \\
        & 3 & 5 & 3 & 3 & $1.720 \times 10^{-11}$ & $7.266 \times 10^{-14}$ 
                        & $1.408 \times 10^{-13}$
    \end{tabular*}
\end{table}

\subsection{Fourth-order wave equation}
\label{sec:numerics-kirchhoff}

We now turn to an example where $a(\cdot,\cdot)$ is indefinite.
Let $\Omega \subset \mathbb{R}^2$ be the L-shaped domain ${(0, 1)^2 
\setminus [0.5, 1]^2}$ with three
holes in \cref{fig:kirchhoff-min-under}
and let $\discrete{W}$ with $k = 5$ be as in \cref{sec:fourth-order}.
We take $\Gamma_c = \emptyset$, $\Gamma_s$ to be the outer boundary
of the L-shape, and $\Gamma_f$ to be the union of the boundaries of 
the holes. Consider the following discrete fourth-order wave equation
subject to an oscillating point load at the mesh vertex $x_0 = (0.66, 
0.33)$ with frequency $\omega$: 
Find $\phi_X(t) \in \discrete{W}$ such that
\begin{align*}
    \frac{\dee^2}{\dee t^2} (\phi_X(t), v)_{L^2(\Omega)} 
    + a_1(\grad \phi_X(t), \grad v) 
    = \sin(\omega t) v(x_0) 
    \qquad \forall v \in \discrete{W},
\end{align*}
where $a_1(\vecbd{\theta}, \vecbd{\psi}) := 
(\symgrad{\vecbd{\theta}}, \symgrad{\vecbd{\theta}})_{L^2(\Omega)^2}
+ (\div \vecbd{\theta}, \div \vecbd{\psi})_{L^2(\Omega)}$, subject to
appropriate initial conditions.
Ignoring physical constants, this system resembles a $C^1$-spline 
semi-discretization of a dynamic Kirchhoff plate
simply-supported on $\Gamma_s$ \cite{LagneseLions88}.

We seek a solution of the form $\phi_X(t) = \sin(\omega t) w_X$ with 
$w_X \in \discrete{W}$ which satisfies
\begin{align}
    \label{eqn:harmonic-kirchhoff}
    w_X \in \discrete{W} : \qquad a_1(\grad w_X, \grad v) - \omega^2 (w_X, 
    v)_{L^2(\Omega)} = v(z) \qquad \forall v \in \discrete{W}.
\end{align}
Problem \cref{eqn:harmonic-kirchhoff} is well-posed provided that 
$\omega$ is not an eigenvalue of the following problem: 
Find $z \in \discrete{W}$ and $\lambda \in \mathbb{R}_+$ such that
\begin{align}
    \label{eqn:h2-eigenvalue-problem}
    a_1(\grad z, \grad v) = \lambda^2 (z, v)_{L^2(\Omega)} \qquad 
    \forall v \in \discrete{W}.
\end{align}
In particular, $\gamma_X$ defined in \cref{eqn:h2-a-infsup} and 
subsequently $\alpha_X$ \cref{eqn:a-infsup-kernel-discrete} converge
to zero as $\omega$ converges to an eigenvalue $\lambda$. Thus, if we 
solve \cref{eqn:harmonic-kirchhoff} by applying the method described in 
\cref{sec:fourth-order}, taking 
\begin{align*}
    a((u_0, \vecbd{u}_1), (v_0, \vecbd{v}_1)) := -\omega^2 (u_0, 
    v_0)_{L^2(\Omega)} + 
    a_1(\vecbd{u}_1, \vecbd{v}_1)
\end{align*}
for all $(u_0, \vecbd{u}_1), (v_0, 
    \vecbd{v}_1) \in \discrete{V} = \discrete{X}\times \discrete{Y}$
with $\discrete{X},\discrete{Y}$ defined in \cref{eqn:DXDY},
then we expect $\rho_0$, the lower bound for $\rho$ necessary for 
well-posedness of the iterated penalty method, to blow up
as $\omega$ approaches a generalized eigenvalue. Consequently, for a 
fixed penalty parameter $\rho$, we expect the number of iterations 
to be large if $\omega$ is close to a generalized eigenvalue.

We now solve \cref{eqn:harmonic-kirchhoff} using the method in 
\cref{sec:fourth-order} with the iterated penalty method with
${\lambda = \rho = 5 \times 10^4}$ and $\texttt{tol2} = 10^{-12}$. 
First, we choose 
$\omega^2 = \tau \lambda_{\min}^2$ for 
${\tau \in \{ 0.9, 0.95, 0.975, 1.025, 1.05, 1.1 \}}$, where
$\lambda_{\min}^2 \approx 4787$ has been computed using a
suitable power method iteration.
The corresponding iteration counts were $\{6, 8, 11, 10, 8, 7\}$ and
the solutions for $\tau \in \{0.975, 1.025\}$ are displayed in 
\cref{fig:kirchoff-min-under-over}. As expected, the iteration counts
increase as $\omega$ approaches $\lambda_{\min}$. We also observe
a sign flip and reflection about $y=x$ in the solution from $\omega < 
\lambda_{\min}$ to
$\omega > \lambda_{\min}$. Second, we choose 
$\omega^2 = 227^2 \tau$ with the same values of $\tau$. The corresponding
iteration counts were $\{ 22, 13, 228, 11, 9, 8 \}$ and the solutions
for $\tau \in \{0.975, 1.025\}$ are displayed in 
\cref{fig:kirchhoff-other-under-over}. Since we again observe an increase
of iteration counts as $\omega$ approaches $227$ and a sign flip and 
reflection over $y=x$
in the solution, we expect another eigenvalue to be near 227.

\begin{figure}[htb]
    \centering
    \begin{subfigure}{0.48\linewidth}
        \centering
        \includegraphics[width=\linewidth]%
        {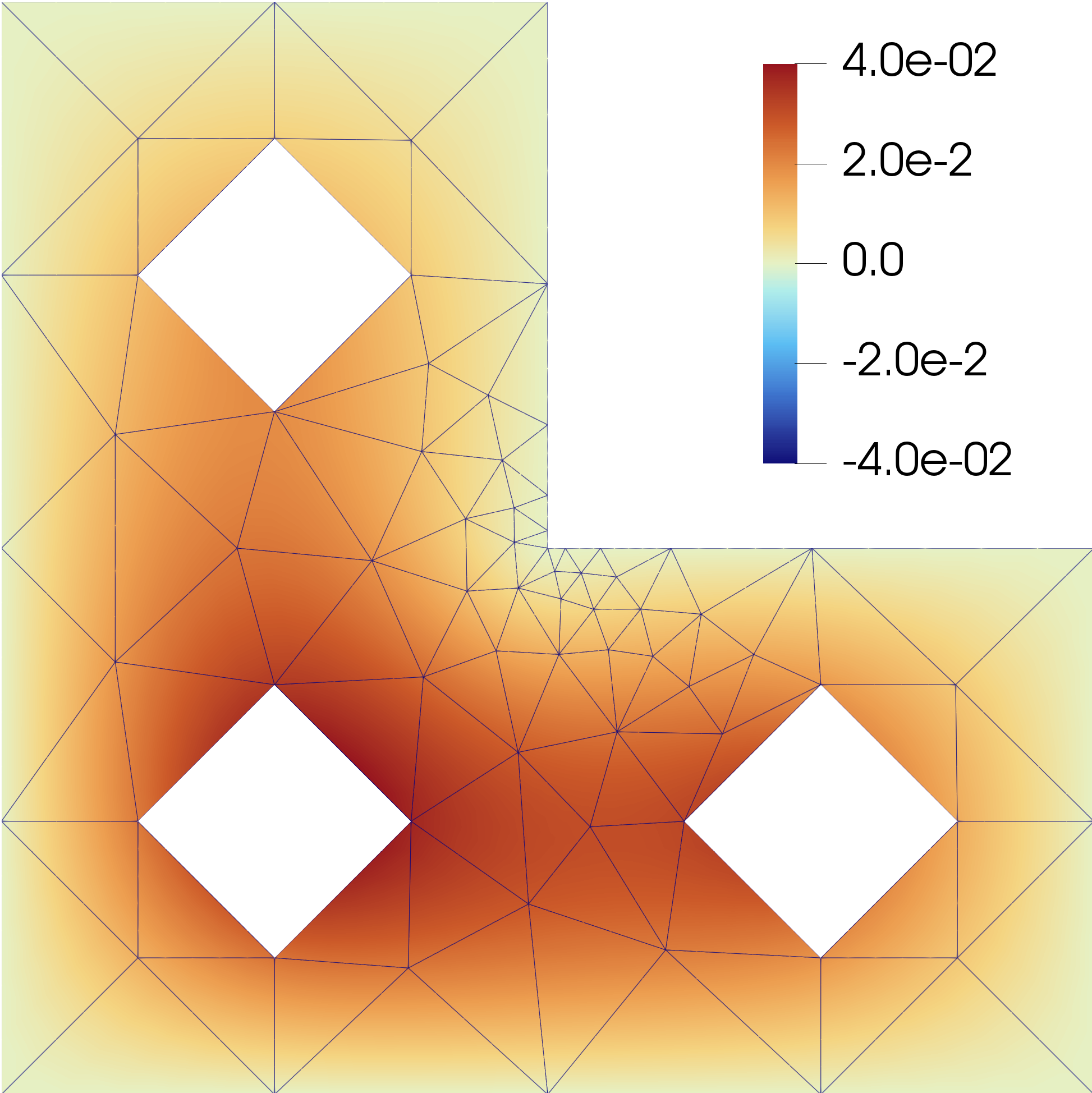}
        \caption{}
        \label{fig:kirchhoff-min-under}
    \end{subfigure}
    \hfill
    \begin{subfigure}{0.48\linewidth}
        \centering
        \includegraphics[width=\linewidth]%
        {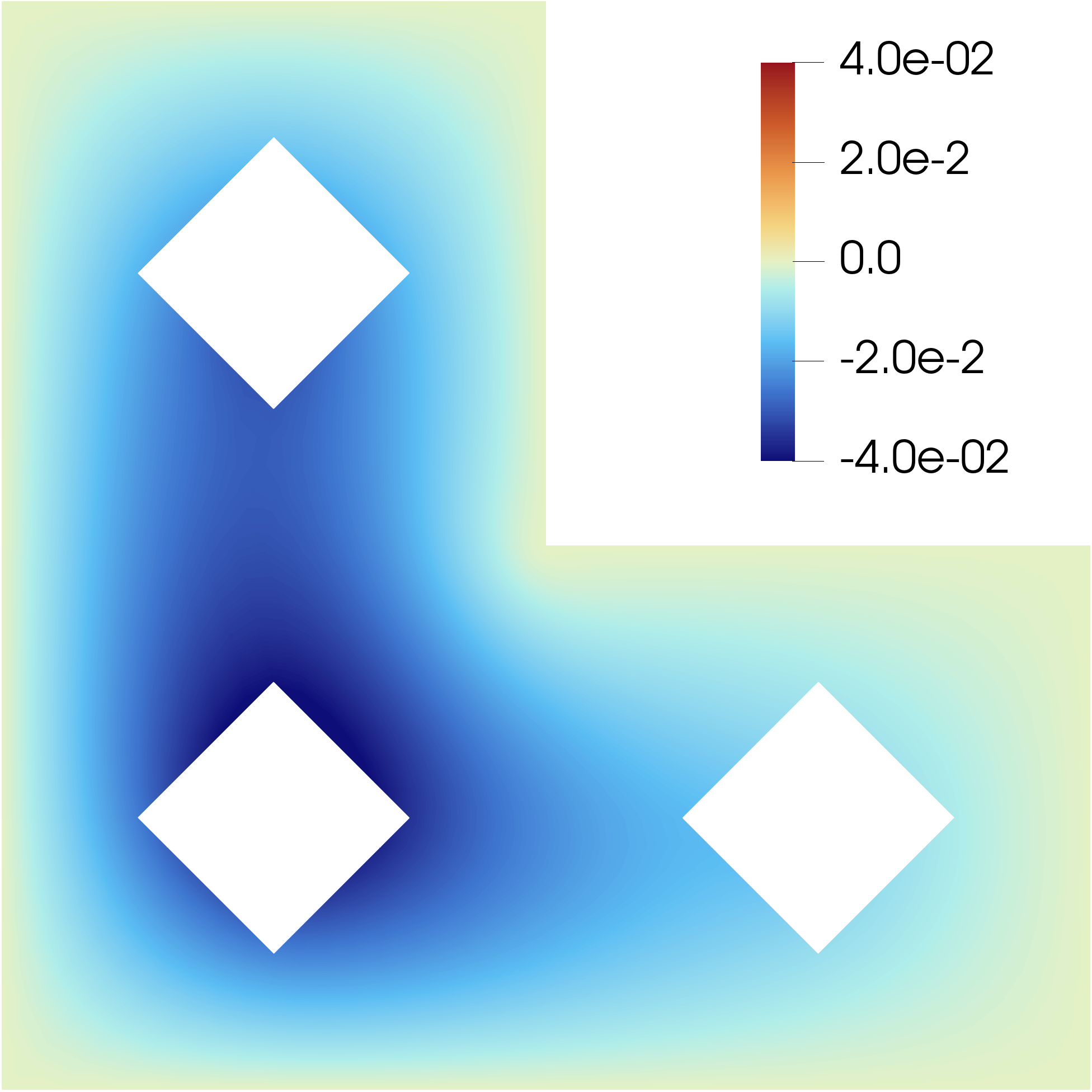}
        \caption{} 
        \label{fig:kirchhoff-min-over}
    \end{subfigure}
    \caption{Solution to \cref{eqn:harmonic-kirchhoff} with $\omega^2 = \tau 
    \lambda_{\min}^2$, where (a) $\tau = 0.975$ and (b) $\tau = 1.025$. Note 
    that the solution flips sign and reflects over $y=x$ from (a) to (b).} 
    \label{fig:kirchoff-min-under-over}
\end{figure}

\begin{figure}[htb]
    \centering
    \begin{subfigure}{0.48\linewidth}
        \centering
        \includegraphics[width=\linewidth]%
        {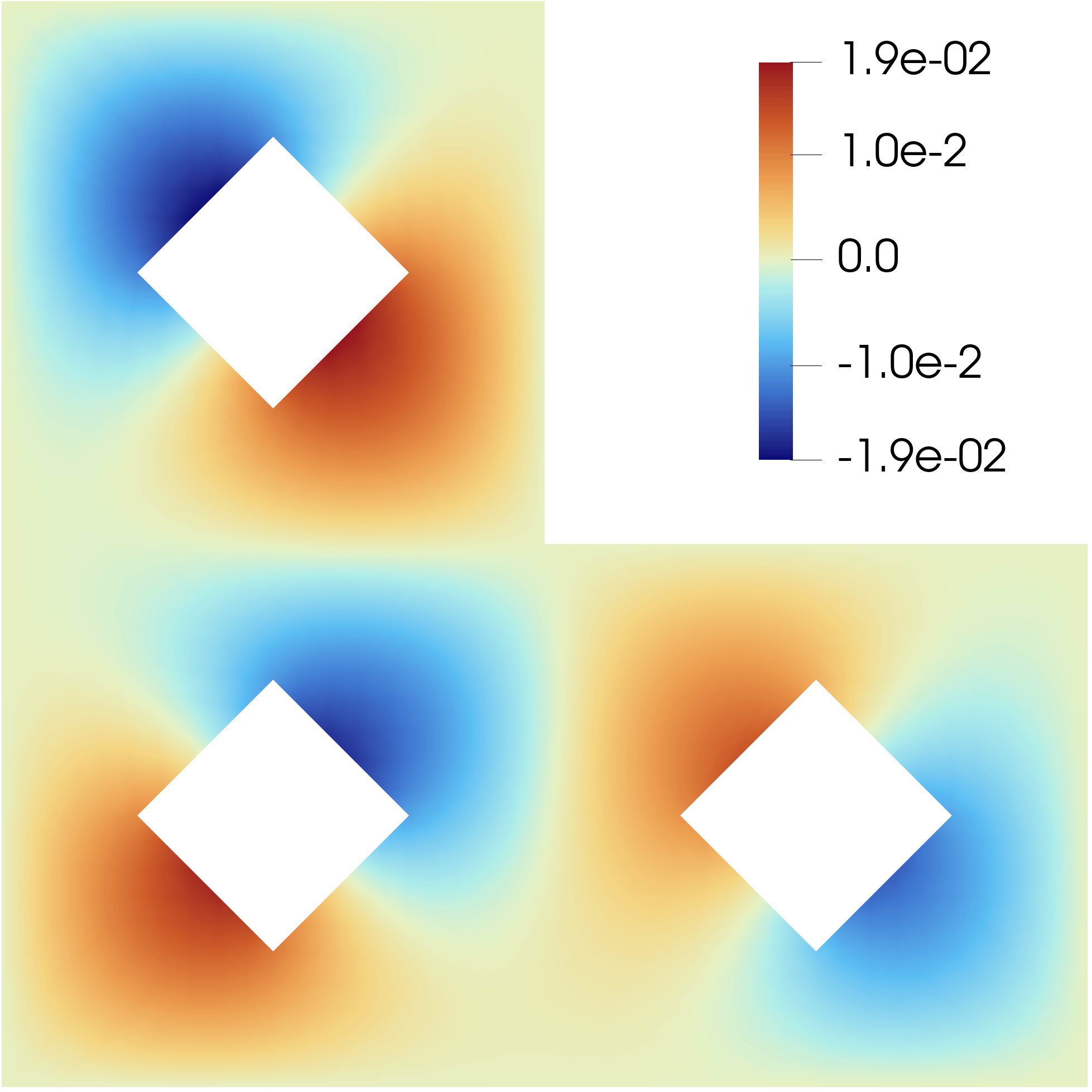}
        \caption{}
        \label{fig:kirchhoff-other-under}
    \end{subfigure}
    \hfill
    \begin{subfigure}{0.48\linewidth}
        \centering
        \includegraphics[width=\linewidth]%
        {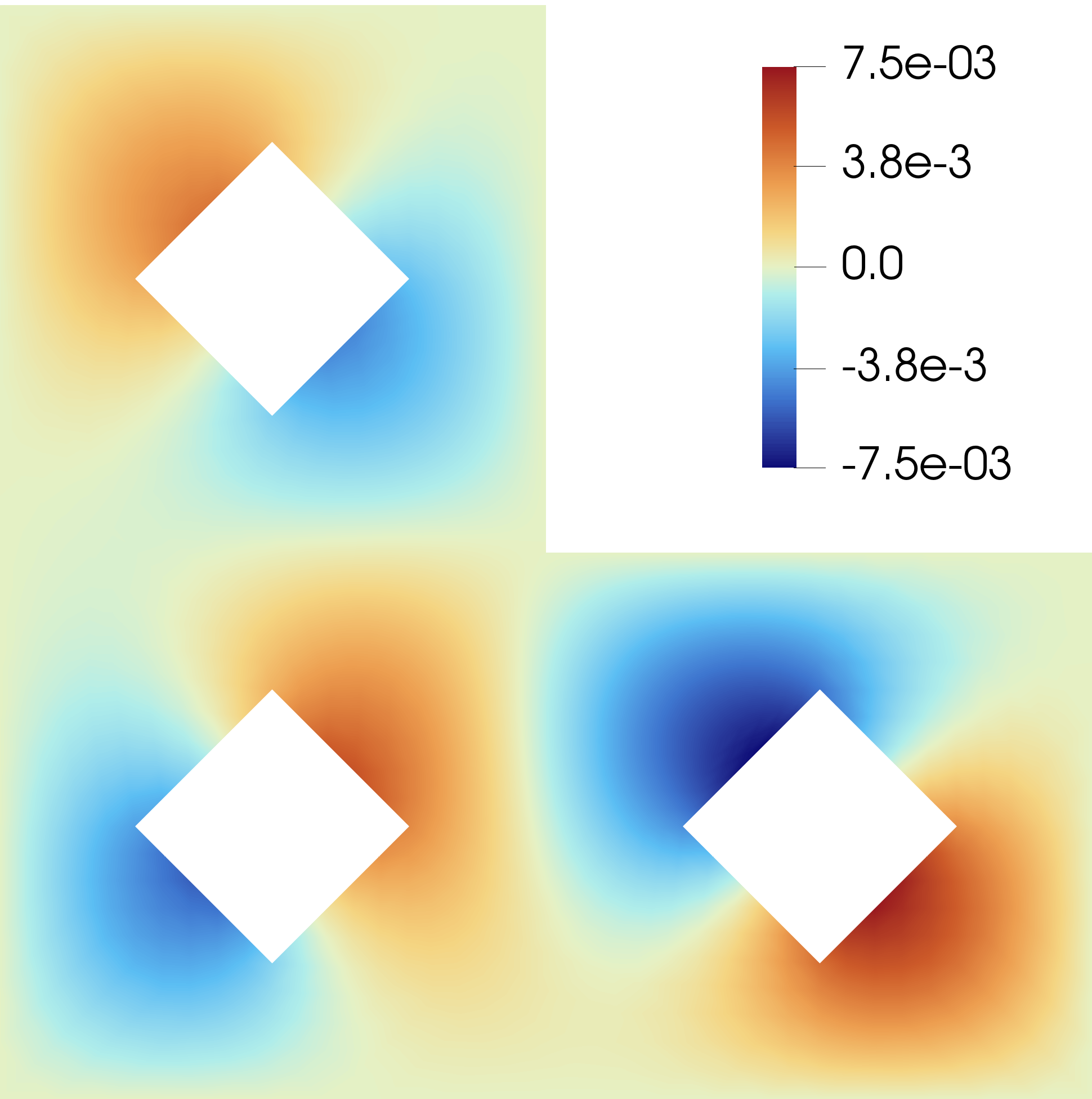}
        \caption{} 
        \label{fig:kirchhoff-othe}
        \end{subfigure}
    \caption{Solution to \cref{eqn:harmonic-kirchhoff} with $\omega^2 = 
        227^2 \tau$, where (a) $\tau = 0.975$ and (b) $\tau = 1.025$. Note 
        that the solution flips sign and reflects over $y=x$ from (a) to 
        (b).}
    \label{fig:kirchhoff-other-under-over} 
\end{figure}

\subsection{Incompressible flow}
\label{sec:numerics-incompressible}


The final example showcases the behavior of
\cref{alg:iter-penalty-computable} when $a(\cdot,\cdot)$ satisfies 
\ref{asmp:a-coercive-all} or \ref{asmp:a-sym-coercive-all}, but the 
inf-sup constant $\beta_X$ \cref{eqn:b-inf-sup-discrete} is made 
arbitrarily small.	
Consider the linear incompressible flow problem with $\nu=1$ from 
\cref{sec:key-examples-sv} discretized with the Scott--Vogelius elements 
\cref{eqn:flow-discrete-spaces-sv} with $k=4$ 
on the four element mesh of the unit square in 
\cref{fig:flow-mesh}. 
We vary the point at which the four triangles meet 
$\vertex{z} = (1/2 + \delta, 1/2)$ for chosen $\delta \in (0, 1/2)$, 
which in turn varies the inf-sup constant $\beta_X$. In particular, 
let $\theta_1, \ldots, \theta_4$ denote the angles in 
\cref{fig:flow-mesh} and set $\theta_5 := \theta_1$ and 
$\theta_{\vertex{z}} := \max_{1 \leq i \leq 4} |\sin(\theta_i + 
\theta_{i+1})|$. The inf-sup constant satisfies the following 
equivalence \cite{GrassleBohneSauter24,ScottVogelius85norm}:
\begin{align*}
    \underline{\beta}_0 \theta_{\vertex{z}} \leq \beta_X \leq 
    \overline{\beta}_0 \theta_{\vertex{z}}
\end{align*}
provided that $\theta_{\vertex{z}}$ is sufficiently small, where 
$\underline{\beta}_0$ and $\overline{\beta}_0$ are universal constants.
An elementary computation shows that $\theta_{\vertex{z}} \sim \delta$ for 
$\delta \ll 1$.

\begin{figure}[htb]
    \centering
    \begin{subfigure}[t]{0.4\linewidth}
        \centering
        \begin{tikzpicture}[scale=0.8]
            \coordinate (a) at (3.25, 3);
            \coordinate (a0) at (0, 0);
            \coordinate (a1) at (6, 0);
            \coordinate (a2) at (6, 6);
            \coordinate (a3) at (0, 6);
            \coordinate (center) at (3, 3);
            
            \filldraw[fill=black!50!white, draw=black!50!white] (center) circle (2pt);
            
            \filldraw (a) circle (2pt) 
            node[align=center,below]{$\vertex{z}$}
            -- (a0) circle (2pt) 	
            -- (a1) circle (2pt) 
            -- (a);
            \filldraw (a) circle (2pt) node[align=center,below]{}	
            -- (a1) circle (2pt) 
            -- (a2) circle (2pt) 
            -- (a);
            \filldraw (a) circle (2pt) node[align=center,below]{}	
            -- (a2) circle (2pt) 
            -- (a3) circle (2pt) 
            -- (a);
            \filldraw (a) circle (2pt) node[align=center,below]{}	
            -- (a3) circle (2pt) 
            -- (a0) circle (2pt) 
            -- (a);
            
            \draw (3, 0.5) node(K0){$K_1$};
            \draw (5.5, 3) node(K1){$K_2$};
            \draw (3, 5.5) node(K1){$K_3$};
            \draw (0.5, 3) node(Km){$K_4$};
            \pic["$\theta_1$"{anchor=north}, draw, angle radius=0.8cm, 
            angle eccentricity=1] {angle=a0--a--a1};
            \pic["$\theta_2$"{anchor=west}, draw, angle radius=1cm, 
            angle eccentricity=1] {angle=a1--a--a2};
            \pic["$\theta_3$"{anchor=south}, draw, angle radius=0.8cm, 
            angle eccentricity=1] {angle=a2--a--a3};
            \pic["$\theta_4$"{anchor=east}, draw, angle radius=1cm, 
            angle eccentricity=1] {angle=a3--a--a0};
            
        \end{tikzpicture}
        \caption{}
        \label{fig:flow-mesh}
    \end{subfigure}
    \hfill
    \begin{subfigure}[t]{0.48\linewidth}
        \centering
        \includegraphics[width=\linewidth]{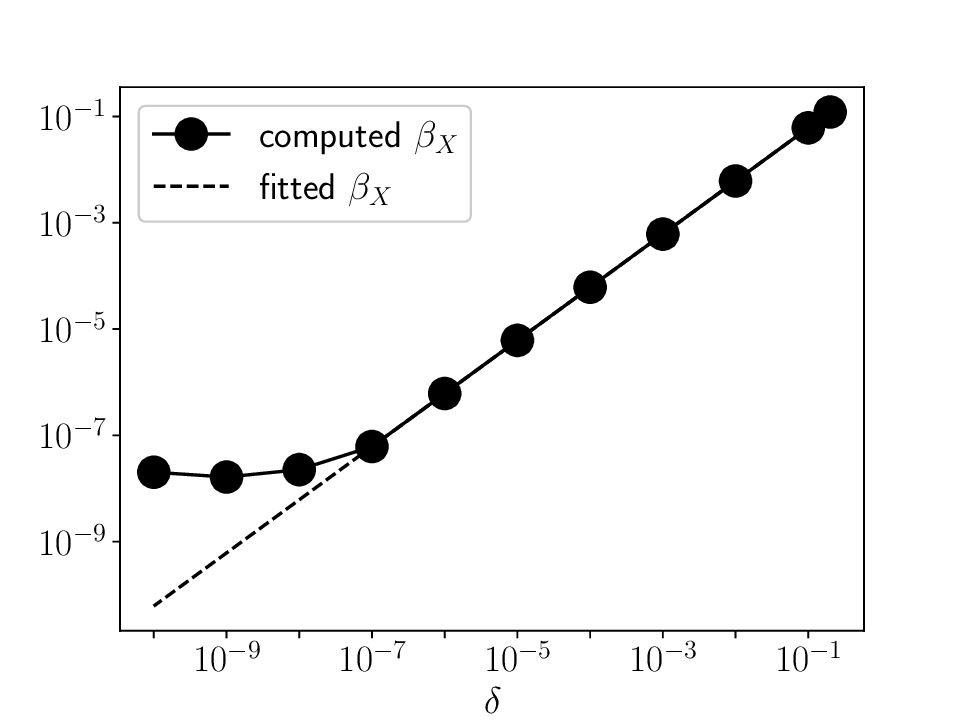}
        \caption{}
        \label{fig:flow-infsup}
    \end{subfigure}
    \caption{(a) Mesh for the example, with the gray dot at 
        $(1/2, 1/2)$, and the vertex at $(1/2 + \delta, 1/2)$.
        (b) The computed inf-sup constants $\beta_X$, with the 
        line of best fit continuation dashed, for $\delta \in \{ 10^{-10}, 
        10^{-9}, \ldots, 10^{-1}, 2\cdot 10^{-1} \}$.}
\end{figure}
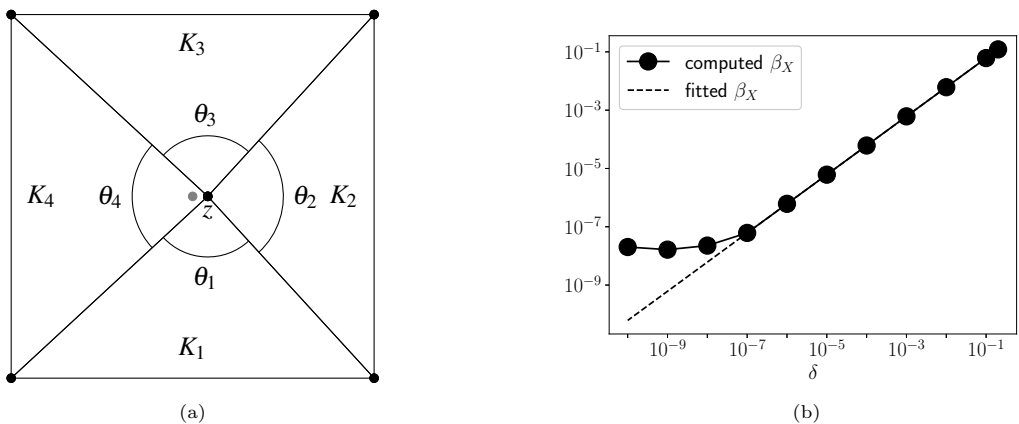

We numerically compute the inf-sup constant as follows. 
Let $\sigma_{\min}^2$ denote the smallest nonzero eigenvalue of the 
following generalized eigenvalue problem: Find $\vecbd{u} \in \discrete{V}$ 
nonzero and $\sigma \in \mathbb{R}_+$ such that
\begin{align*}
    (\vecbd{u}, \vecbd{v})_{H^1(\Omega)^2} 
        = \sigma^2 (\div \vecbd{u}, \div \vecbd{v})_{L^2(\Omega)} 
    \qquad \forall \vecbd{v} \in \discrete{V}.
\end{align*}
Then, as we shall see in \cref{sec:refined-analysis-aspd} below, 
$\beta_X = \sigma_{\min}$. 
Since the above system is small for this example, 
we assemble the matrices and use \texttt{scipy.linalg.eigh} to compute 
the eigenvalues displayed in \cref{fig:flow-infsup}. 
Since the eigenvalue solver terminates when $\sigma_{\min}^2$ is on the 
order of machine precision, the computed inf-sup constant levels off around
$10^{-8}$ in double precision. 
The dashed line indicates the continuation of the line of 
best fit for $\delta > 10^{-7}$, which is a more 
accurate calculation of $\beta_X$ for $\delta < 10^{-7}$.

To demonstrate the convergence behavior of the iterated penalty method, we 
consider two choices of the convective velocity 
$\vecbd{w}$: (i) $\vecbd{w} = \vecbd{0}$ 
corresponding to the Stokes equations and (ii) ${\vecbd{w} = \curl 
(\sin(\pi x) \cos(\pi y))}$. The linear functionals are taken to be of the 
form $F(\vecbd{v}) = (\vecbd{f}, \vecbd{v})_{L^2(\Omega)}$ 
and ${G(q) = -(g, q)_{L^2(\Omega)}}$ for some $\vecbd{f}$ and $g$. We take 
$\vecbd{f}$ to be either (i) a smooth 
function ${\vecbd{f} = [\sin(\pi x), \sin(\pi y)]^{T}}$ or (ii) a continuous 
vector-valued piecewise polynomial of degree 4 on the mesh whose 
degrees of freedom are chosen uniformly at random in $[-0.5,0.5]$. 
Similarly, $g$ is taken to be either (i) zero or (ii) $\div \vecbd{\Phi}$ 
for a randomly chosen $\vecbd{\Phi} \in \discrete{V}$, so that in both 
cases, $\operator{R}_{\discrete{Q}}^{-1} G = {g}$. 
The six test cases are 
the smooth choice of $\vecbd{f}$ with $g = 0$, 
the random choice of $\vecbd{f}$ with $g = 0$, and 
the random choice of $\vecbd{f}$ and $g$, each with both choices of 
convective velocity.

For each of the six test cases, we choose 
$\delta \in \{ 10^{-10}, 10^{-9}, \ldots, 10^{-1}, 2\cdot 10^{-1} \}$ and	
run \cref{alg:iter-penalty-computable} with $\lambda = \rho = 10^3$ and 
$\texttt{tol2} = 10^{-12}$, skipping lines 5--6 and taking $g_{X, \rho}^m =   
\operator{R}_{\discrete{Q}}^{-1} G$ as above. 
We also terminate the method if 
the number of iterations exceeds $10^3$. The iteration counts and 
final values of
$\|\div \vecbd{u}_{X, \rho}^n - g\|_{L^2(\Omega)}$ 
are displayed in \cref{fig:flow-iterations-divergence}. 
For $\delta \geq 10^{-2}$, the iterated penalty method
terminates within $10^3$ iterations for all of the examples, which is
consistent with \cref{cor:iter-penalty-conv} since $\beta_X$ is not so 
small. 
For $\delta \leq 10^{-3}$, we observe that the method begins to fail to 
terminate within $10^3$ iterations unless $\delta$ is sufficiently small 
depending on the particular example. This behavior is 
in contrast with the convergence rate in 
\cref{cor:iter-penalty-conv}, which would suggest that more iterations are 
needed to terminate as $\delta$ decreases since $\beta_X$ is 
becoming increasingly smaller.
Additionally, the divergence plot \cref{fig:flow-divergence} shows that the 
value of $\|\div \vecbd{u}_{X, \rho}^n - g\|_{L^2(\Omega)}$ upon 
termination seems to scale like $\delta$ for $\delta < 10^{-3}$ for the 
examples that fail to converge within $10^3$ iterations.

\begin{figure}[htb]
    \centering
    \begin{subfigure}[b]{0.48\linewidth}
        \centering
        \includegraphics[width=\linewidth]{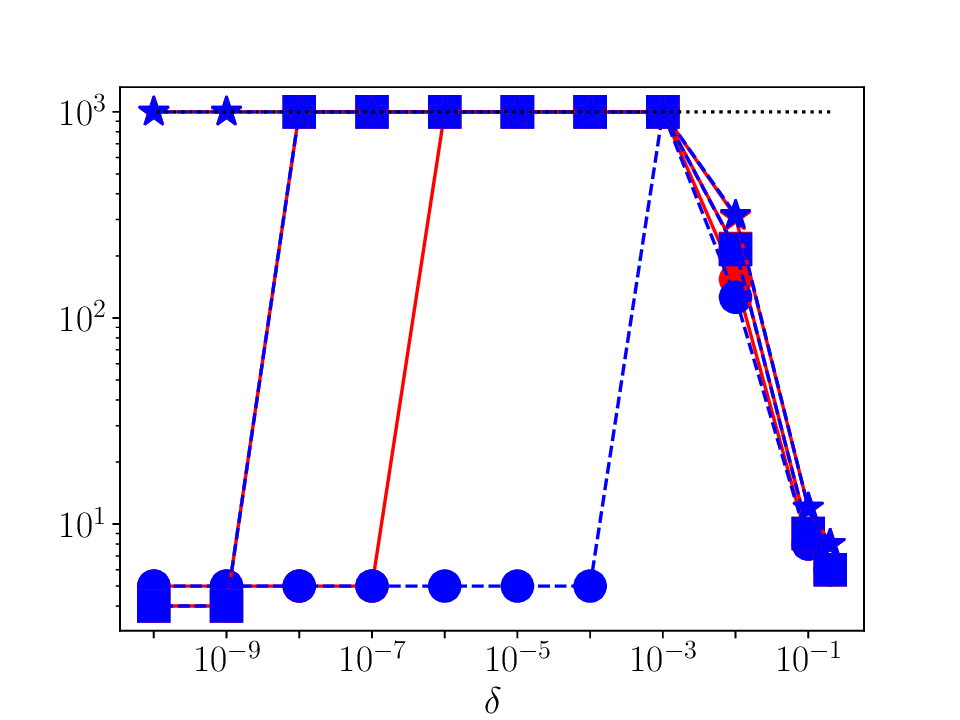}
        \caption{}
        \label{fig:flow-iterations}
    \end{subfigure}
    \hfill
    \begin{subfigure}[b]{0.48\linewidth}
        \centering
        \includegraphics[width=\linewidth]{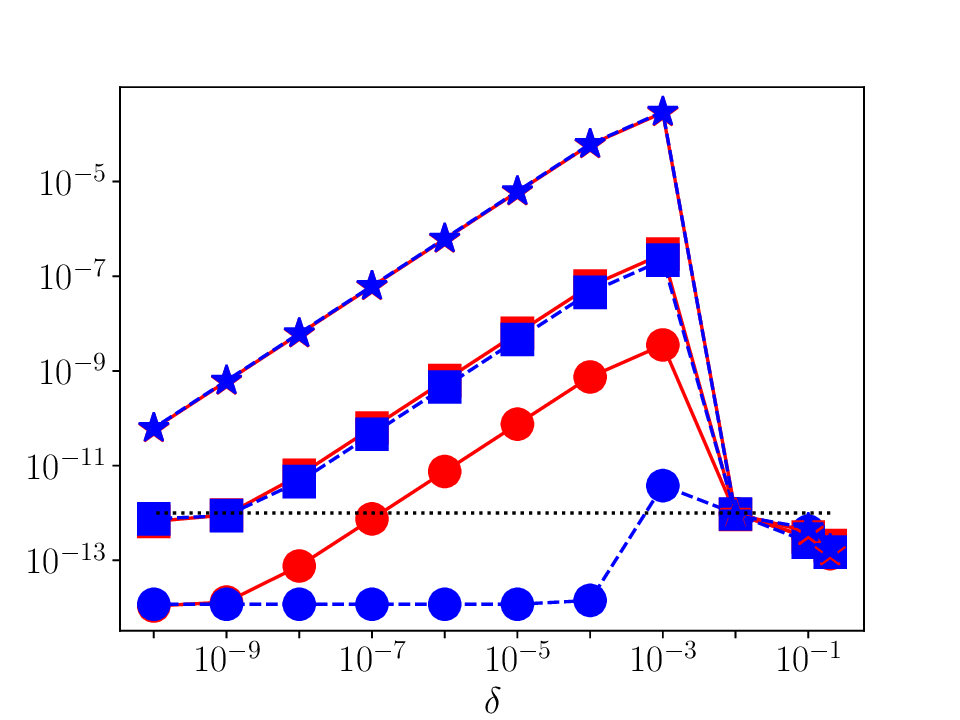}
        \caption{}
        \label{fig:flow-divergence}
    \end{subfigure}
    \caption{(a) Iteration counts and 
        (b) $\|\div \vecbd{u}_{X, \rho}^n - g\|_{L^2(\Omega)}$ upon 
        termination for the 2D flow example in 
        \cref{sec:numerics-incompressible} with 
        $\delta \in \{ 10^{-10}, 10^{-9}, \ldots, 10^{-1}, 2\cdot 10^{-1} 
        \}$. 
        \raisebox{-0.7mm}{\scalebox{1.75}{$\bullet$}}: $\vecbd{f}$ 
        smooth and $g \equiv 0$;  
        \raisebox{-0.4mm}{$\blacksquare$}: $\vecbd{f}$ random and 
        $g \equiv 0$; 
        $\bigstar$: $\vecbd{f}$ random and $g$ random;
        \textcolor{red}{red solid lines}: $\vecbd{w} = \vecbd{0}$;
        \textcolor{blue}{blue dashed lines}: $\vecbd{w}$ smooth. The dotted 
        black line represents the maximum number of iterations in (a) and 
        the stopping tolerance in (b). }
    \label{fig:flow-iterations-divergence}
\end{figure}

\section{Refining the analysis under \ref{asmp:a-sym-coercive-all}}
\label{sec:refined-analysis-aspd}

As we have seen in \cref{sec:numerics-incompressible} above, 
the estimates in 
\cref{thm:iter-penalty-2param-error} do not adequately explain the 
behavior of the 
iterated penalty method in all cases. In the case that $a(\cdot,\cdot)$ 
satisfies \ref{asmp:a-sym-coercive-all}, which we will assume for the rest 
of this section, we can perform a more refined analysis with generalized 
eigenfunctions, as was done in \cite{FortinGlow83} in the case $G=0$.

To this end, let $N_0 := \dim \discrete{Z}$ and
$N_{\perp} := \dim \tilde{\discrete{Z}} = \dim \discrete{Q}$.
Then, there exist an $a(\cdot,\cdot)$-orthonormal basis $\{ z_i 
\}_{i=1}^{N_0} \subset \discrete{Z}$ and an $a(\cdot,\cdot)$-orthonormal 
basis $\{ \tilde{z}_j \}_{j=1}^{N_{\perp}} \subset 
\tilde{\discrete{Z}}$ and strictly positive eigenvalues $\{ \sigma_j 
\}_{j=1}^{N_{\perp}}$ (sorted in ascending order) satisfying
\begin{align}
    \label{eqn:generalized-eigenvalue-problem}
    (\operator{D} \tilde{z}_i, \operator{D} \tilde{z}_j)_{Q} = 
    \delta_{ij} \sigma_i^2 
    a(\tilde{z}_i, \tilde{z}_j), \qquad 1 \leq i, j \leq N_{\perp}.
\end{align}
Note that the basis $\{ z_i \}_{i=1}^{N_0} \cup \{ \tilde{z}_j 
\}_{j=1}^{N_{\perp}}$ is both $a(\cdot,\cdot)$ and $(\operator{D} \cdot, 
\operator{D} 
\cdot)_Q$-orthogonal and $\{ \psi_j := \sigma_j^{-1} \operator{D} 
\tilde{z}_j 
\}_{j=1}^{N_{\perp}}$ is an orthonormal basis for $\discrete{Q}$. 
We also remark that $\sigma_1$ coincides with the following inf-sup 
constant \cite[Section 3.5]{ElmanBook}:
\begin{align*}
    \beta_{X, a} := \newinf_{q \in \discrete{Q}} \sup_{v \in \discrete{V}} 
    \frac{(\operator{D} v, q)_Q}{|v|_a \|q\|_Q},
\end{align*}
where we recall that $|\cdot|_a := \sqrt{a(\cdot,\cdot)}$.
Indeed, taking 
$q = \sum_{j=1}^{N_{\perp}} q_j \psi_j \in 
\discrete{Q}$ arbitrary, we set 
$v = \sum_{j=1}^{N_{\perp}} q_j \tilde{z}_j$ which gives
\begin{align*}
    \frac{(\operator{D} v, q)_Q}{|v|_a \|q\|_Q} =  
    \frac{\sum_{j=1}^{N_{\perp}} \sigma_j q_j^2}{\sum_{j=1}^{N_{\perp}} 
    q_j^2 } \geq \sigma_1  
\end{align*}
with equality if $q_j = 0$ for $j \geq 2$. This shows that 
$\beta_{X, a} = \sigma_{1}$. Similarly, we have
\begin{align*}
    M_{\operator{D}, a} := \sup_{ v \in \discrete{V} } \frac{\| 
    \operator{D} v\|_Q}{|v|_a} = \sigma_{N_{\perp}}.
\end{align*}

\subsection{Convergence estimates}

We first express the solution to \cref{eqn:saddle-discrete} in terms of the 
above eigenfunctions. Its proof follows from standard arguments
and is therefore omitted.
\begin{lemma}
    \label{lem:a-coercive-saddle-soln-eigen}
    The solution $u_X \in \discrete{V}$ and $p_X \in \discrete{Q}$ to 
    \cref{eqn:saddle-discrete} satisfies
    \begin{align*}
        u_X &= \sum_{i=1}^{N_0} F(z_i) z_i 
        + \sum_{j=1}^{N_{\perp}} \sigma_j^{-1} G(\psi_j) \tilde{z}_j
        \quad \text{and} \quad
        p_X = \sum_{j=1}^{N_{\perp}} \sigma_j^{-1} \left( F(\tilde{z}_j) - 
        \sigma_j^{-1} G(\psi_j) \right) \psi_j. 
    \end{align*}
\end{lemma}
Now, we have a more refined expression for the error in the iterated 
penalty method.
\begin{lemma}
    \label{lem:iter-penalty-2param-spd-error-equality}
    Let $\rho, \lambda > 0$, $u_X, u_{X, \lambda, \rho}^n \in 
    \discrete{V}$, 
    and $p_X, p_{X, \lambda, \rho}^n \in \discrete{Q}$, $n \in \mathbb{N}$, 
    be as 
    in \cref{thm:iter-penalty-2param-error}. Then, there holds
    \begin{subequations}
        \label{eqn:iter-penalty-2param-spd-error-equality}
        \begin{align}
            \label{eqn:iter-penalty-2param-spd-error-equality-u}
            | u_X - u_{X, \lambda, \rho}^n |_a^2 &= 
            \sum_{j=1}^{N_{\perp}} 
                \frac{|1 + (\lambda-\rho) \sigma_j^2|^{2(n-1)}}{ |1 + \lambda 
            \sigma_j^2 |^{2n}}  
            \left( F(\tilde{z}_j) - \sigma_j^{-1} G(\psi_j) \right)^2, \\
            \label{eqn:iter-penalty-2param-spd-error-equality-p}
            \|p_X - p_{X, \lambda, \rho}^{n} \|_Q^2 &= 
            \sum_{j=1}^{N_{\perp}} \frac{1}{\sigma_j^{2}} 
            \left| \frac{1 + (\lambda-\rho) \sigma_j^2}{ 1 + \lambda 
            \sigma_j^2} \right|^{2n}
            \left( F(\tilde{z}_j) - \sigma_j^{-1} G(\psi_j) \right)^2, \\
            \label{eqn:iter-penalty-2param-spd-error-equality-d}
            \|\operator{R}_{\discrete{Q}}^{-1} G 
            - \operator{D} u_{X, \lambda, \rho}^n \|_{Q}^2 
            &= \sum_{j=1}^{N_{\perp}} 
                \frac{\sigma_j^2|1 + (\lambda-\rho) \sigma_j^2|^{2(n-1)}}{ 
                |1 + \lambda 
                \sigma_j^2 |^{2n}}  
            \left( F(\tilde{z}_j) - \sigma_j^{-1} G(\psi_j) \right)^2.
        \end{align}
    \end{subequations}
\end{lemma}
\begin{proof}
    Let $e_X^n := u_X - u_{X, \lambda, \rho}^{n}$ and 
    $r_X^n := p_X - p_{X, \lambda, \rho}^n$. On recalling that $e_X^n \in 
    \tilde{\discrete{Z}}$, we expand these 
    functions in the eigenexpansion: $e_X^{n} = \sum_{j=1}^{N_{\perp}} 
    \tilde{e}_j^n \tilde{z}_j$ and $r_X^n = \sum_{j=1}^{N_{\perp}} r_j^n 
    \psi_j$. Choosing $v = \tilde{z}_k$ in 
    \cref{eqn:proof:iter-penalty-2param-key-ids} gives
    \begin{align*}
        (1 + \lambda \sigma_k^2) \tilde{e}_k^n = -\sigma_k r_k^{n-1} \quad 
        \text{and} \quad r_k^n = r_k^{n-1} + \rho \sigma_k \tilde{e}_k^n 
        \implies 
        r_k^n = \frac{1 + (\lambda - \rho) \sigma_k^2}{1 + \lambda 
            \sigma_k^2} r_k^{n-1}.
    \end{align*}
    Since $r_X^0 = p_X$, \cref{lem:a-coercive-saddle-soln-eigen} gives
    \begin{align*}
        \tilde{e}_k^n &= -\frac{(1 + 
            (\lambda - \rho) 
            \sigma_k^2)^{n-1}}{(1 + \lambda \sigma_k^2)^n} 
            (F(\tilde{z}_k) - \sigma_k^{-1} G(\psi_k)) \\
        r_k^n &= \frac{1}{\sigma_k} \left( \frac{1 + (\lambda - \rho) 
            \sigma_k^2}{1 + \lambda \sigma_k^2}  \right)^n 
            (F(\tilde{z}_k) - \sigma_k^{-1} G(\psi_k)). 
    \end{align*}
    The result now follows from \cref{eqn:generalized-eigenvalue-problem} 
    on noting that $\operator{R}_{\discrete{Q}}^{-1} 
    G = \operator{D} u_X$.
\end{proof}

Note that for any $\lambda > 0$ and $\rho \in (0, 2\lambda)$, 
\cref{eqn:iter-penalty-2param-spd-error-equality} shows that the iterated 
penalty method converges geometrically with a rate given by
\begin{align}
    \label{eqn:iter-penalty-2param-spd-error-rate}
    \max_{1 \leq j \leq N_{\perp}} \left| \frac{1 + (\lambda - 
    \rho)\sigma_j^2}{1 + \lambda \sigma_j^2} \right| 
    = \max\left\{ 
    \frac{1 + (\lambda - \rho)\beta_{X, a}^2}{1 + \lambda \beta_{X, a}^2}, 
    \frac{(\rho - \lambda)M_{\operator{D}, a}^2 - 1}{1 + \lambda 
    M_{\operator{D}, a}^2} \right\},
\end{align}
where the equality follows from elementary calculus.

\subsection{Re-examining the incompressible flow example}

We revisit two of the incompressible flow problems in 
\cref{sec:numerics-incompressible}: (i) $\vecbd{f}$ smooth and $g=0$ with 
$\delta \in \{ 10^{-6}, 10^{-7}\}$ and (ii) $\vecbd{f}$ random and $g=0$ 
with $\delta \in \{ 10^{-9}, 10^{-10} \}$. For each example, we numerically 
compute the eigenfunctions $\vecbd{\tilde{z}}_j \in  \tilde{\discrete{Z}}$ 
and eigenvalues $\sigma_j$ defined by 
\cref{eqn:generalized-eigenvalue-problem}
and plot in \cref{fig:flow-eigen} the ${|\cdot|_a}$ velocity error, the $L^2$ 
pressure error, and the 
$L^2$ norm of $\div \vecbd{u}_X^n$ computed via 
\cref{eqn:iter-penalty-2param-spd-error-equality}. 
Since $\sigma_1$ is approximately $(0.867) \delta$ and $\sigma_2$ is 
approximately $5.48 \times 10^{-1}$ for 
$\delta \in \{10^{-6}, 10^{-7}, \ldots, 10^{-10}\}$,
we separate the errors into the contributions from  
the first eigenmode and the remaining eigenmodes. For instance, we plot
\begin{align}
    \label{eqn:eigen-mode-error-split}
    \frac{1}{ |1 + \rho \sigma_1^2 |^{n}} |F(\vecbd{\tilde{z}}_1)|
    \quad \text{and} \quad 
    \left( \sum_{j=2}^{N_{\perp}} 
    \frac{1 }{ |1 + \rho \sigma_j^2 |^{2n}}  
    |F(\vecbd{\tilde{z}}_j)|^2 \right)^{1/2}
\end{align} 
for the $|\cdot|_a$ velocity error, with the other errors split 
analogously.
    
\begin{figure}[htb]
    \centering
    \begin{subfigure}[b]{0.48\linewidth}
        \centering
        \includegraphics[width=\linewidth]%
        {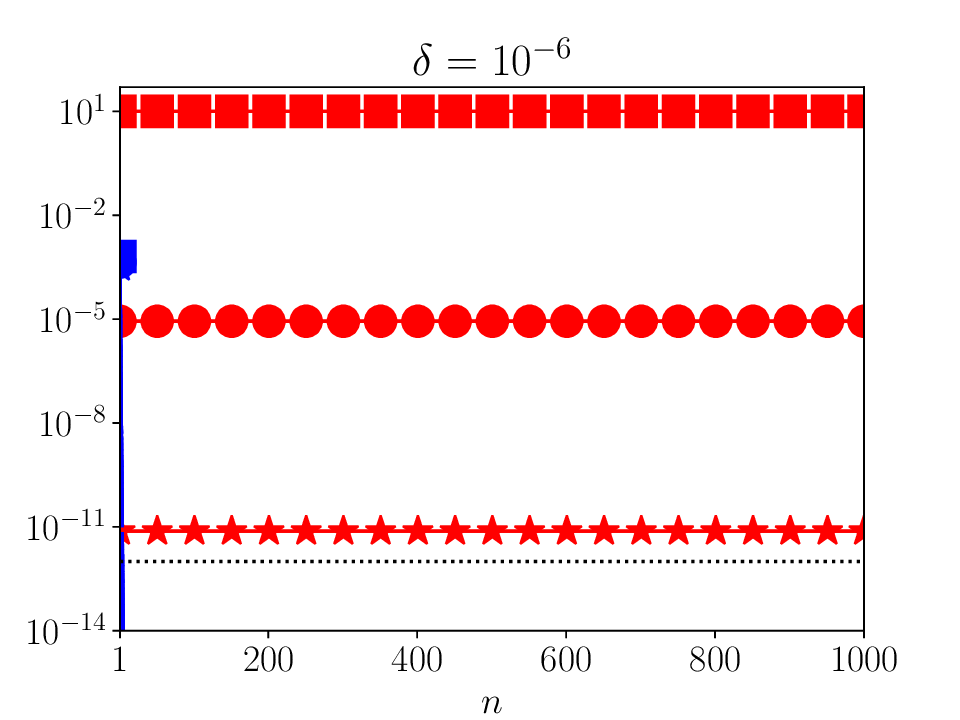}
        \caption{}
    \end{subfigure}
    \hfill
    \begin{subfigure}[b]{0.48\linewidth}
        \centering
        \includegraphics[width=\linewidth]%
        {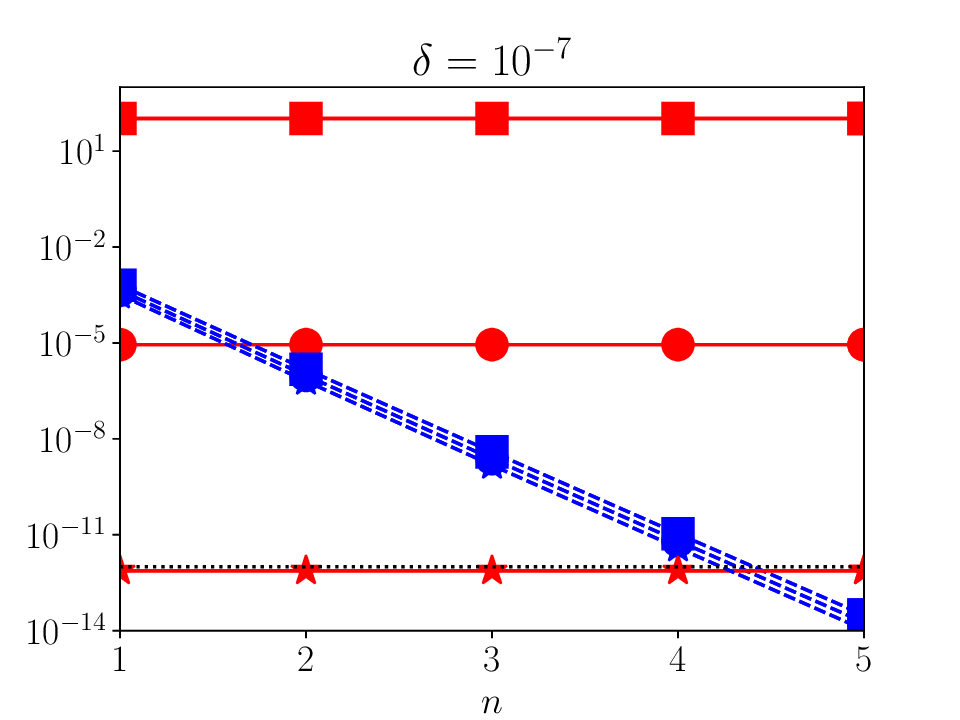}
        \caption{}
    \end{subfigure}
    \\
    \begin{subfigure}[b]{0.48\linewidth}
        \centering
        \includegraphics[width=\linewidth]%
        {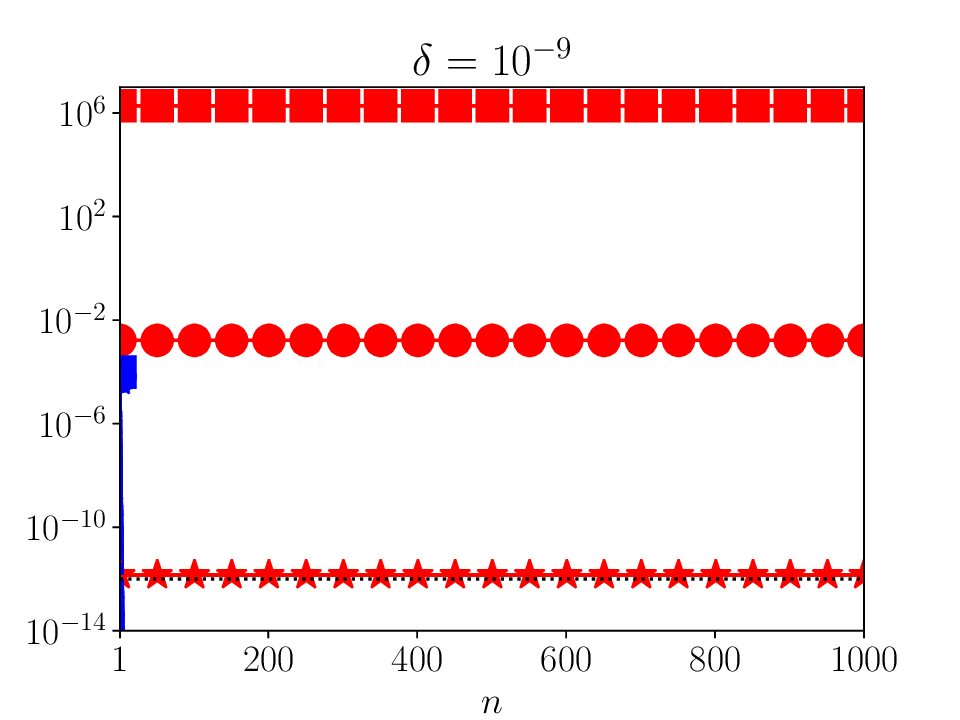}
        \caption{}
    \end{subfigure}
    \hfill
    \begin{subfigure}[b]{0.48\linewidth}
        \centering
        \includegraphics[width=\linewidth]%
        {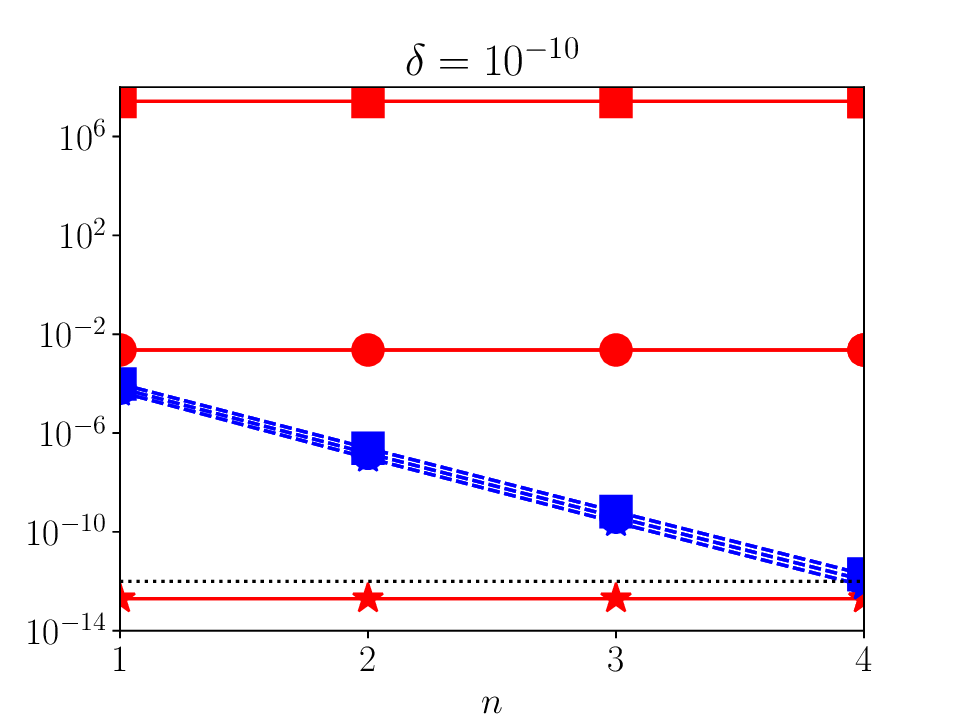}
        \caption{}
    \end{subfigure}
    \caption{Numerical results for the incompressible flow problem in 
    \cref{sec:numerics-incompressible} with 
        (a-b) $\vecbd{f}$ smooth and $g=0$ and (c-d) $\vecbd{f}$ random and 
        $g=0$. The errors are split into the first and remaining eigenmodes 
        as in \cref{eqn:eigen-mode-error-split}.
        \raisebox{-0.7mm}{\scalebox{1.75}{$\bullet$}}: velocity 
        $|\cdot|_a$ error;  
        \raisebox{-0.4mm}{$\blacksquare$}: pressure $\|\cdot\|_{L^2}$ 
        error; 
        $\bigstar$: $\|\div \cdot \|_{L^2(\Omega)}$ error;
        \textcolor{red}{red solid lines}: error in first eigen 
        mode;
        \textcolor{blue}{blue dashed lines}: error in the remaining 
        eigenmodes. 
        The dotted black lines in (b) and (d) indicate the stopping 
        tolerance. }
    \label{fig:flow-eigen}
\end{figure}
    
For each of the examples, we observe little change from the first and final 
iterations in the error in the first eigenmode for each of the three 
error metrics, as $1/(1 + \rho \sigma_1^2) \approx 1$. Meanwhile, 
the error in the remaining eigenmodes converges quickly to 0 since 
$1/(1 + \rho \sigma_2^2) \ll 1$. 
Thus, the iterated penalty method will terminate within $10^3$ iterations  
only if $\|\div \vecbd{u}_{X, \rho}^1\|_{L^2(\Omega)} < \texttt{tol2}$. 
Since $\|\div \vecbd{u}_{X, \rho}^1\|_{L^2(\Omega)}$ scales like $\delta$ 
thanks to \cref{eqn:iter-penalty-2param-spd-error-equality-d}, the method 
terminates only if $\delta$ is sufficiently small, which is precisely what 
we observe 
in \cref{fig:flow-eigen} and what we observed earlier in 
\cref{fig:flow-iterations-divergence}. 
However, even when the iterated penalty method terminates, 
the errors in $\vecbd{u}_X$ and $p_X$ are still well-above the tolerance. 
In fact, the error in $p_X$ scales like $1/\delta$. 
Consequently, one must exercise 
caution if the value of the inf-sup constant $\beta_X$ is unknown or 
$\beta_X \ll 1$, as \cref{alg:iter-penalty-computable} may terminate since 
the stopping criterion \cref{eqn:residual-stopping-criteria} is satisfied, 
but the errors $\|\vecbd{u}_{X,\rho}^n - \vecbd{u}_X\|_V$ and 
$\|p_{X, \rho}^n - p_X\|_{Q}$ are large.

\begin{remark}
    In \cref{sec:numerics-incompressible}, the example with $\vecbd{f}$  
    random and $g=0$ converged for $\delta \leq 10^{-9}$, 
    whereas the computations in \cref{fig:flow-eigen} show that it 
    converges for $\delta \leq 10^{-10}$. 
    This discrepancy is likely due to rounding errors.
\end{remark}

\section{Relation to previous work}
\label{sec:previous-work}

\subsection{Penalty method}

The stability of the perturbed saddle point problem or penalty method 
\cref{eqn:saddle-discrete-penalty} has long been known, with
the first result seemingly being \cite[Theorem 5.1]{Arnold81}, which
assumed \ref{asmp:a-sym-nonneg-coercive-z} and showed existence and 
stability of solutions for all $\epsilon \geq 0$. However, the dependence
of the stability constants on $M_a$, $\beta_X$, etc.~was not explicitly
given, as is the case in many standard references e.g.~\cite[p.~65, Theorem 4.3]{GiraultRaviart86},
\cite[p.~47, Theorem 1.2]{BrezziFortin91}, and
\cite[Theorem 4.11]{Braess07}. When explicit constants
do appear, e.g.
\cite[Theorem 4.3.2]{BoffiBrezziFortin13} assuming 
\ref{asmp:a-sym-nonneg-coercive-z},
\cite[eq.~(3.16)]{AinsworthParker22unlocking} for linear elasticity,
and
\cite[Lemma 7.1.1]{BernardiGiraultHechRaviartRiviere25}
for the Stokes equations,
the dependence of the constants in \cref{tab:saddle-perturb-cont-const}
are sharper, particularly as $\beta_X \to 0$ and/or as $\epsilon \to 0$. 
We also remark that the paragraph after \cite[p.~359, Theorem 
13.1.19]{BrennerScott08} notes that the arguments in 
\cite{BrennerScott08} are not sufficient to prove that
$\|p_{X, \epsilon} - p_X\|_{Q} \to 0$ as $\epsilon \to 0$, but 
\cref{lem:saddle-perturb-error} shows that this is always the case.
Thus the last sentence in the paragraph after \cite[p.~359, Theorem 
13.1.19]{BrennerScott08} is wrong.

\subsection{Iterated penalty/augmented Lagrangian method}

The convergence of the iterated penalty/augmented Lagrangian method
\cref{eqn:iter-penalty-2param}
has also been previously established under certain conditions. The 
first convergence proofs appear to be due to Fortin and Glowinski
\cite{FortinGlow83}, assuming that $a(\cdot,\cdot)$ satisfies 
\ref{asmp:a-sym-coercive-all}, $G \equiv 0$, and $\rho \in (0, 2\lambda]$.
In particular, if we define 
\begin{align*}
    y_{X, \lambda \rho}^{n} := \mathcal{A}^{-1} \operator{D}^{\star}(
    p_{X, \lambda, \rho}^{n-1} - p_X),
\end{align*}
then \cite[p.~10, eq.~(2.33)]{FortinGlow83} gives the relation
\begin{align*}
    y_{X, \lambda, \rho}^{n+1} &= 
        \frac{1 + (\lambda - \rho) \sigma_k^2}{1 + \lambda \sigma_k^2} 
        y_{X, \lambda, \rho}^{n},
\end{align*}
which resembles expressions in the proof of 
\cref{lem:iter-penalty-2param-spd-error-equality}.
In particular $y_{X, \lambda, \rho} \to 0$ if and only if the
augmented Lagrangian method converges to the solution to 
\cref{eqn:saddle-discrete}. Pages 10-15 of \cite{FortinGlow83}
then show convergence rates of $y_{X, \lambda, \rho}$ for various 
choices of $\rho$ for a given choice of $\lambda$. Although not explicitly
stated in \cite{FortinGlow83}, one could arrive at 
\cref{lem:iter-penalty-2param-spd-error-equality} with elementary 
manipulations. It is also stated in \cite[p.~42, Remark 5.1]{FortinGlow83}
that all of the discussions extend to the case $G \neq 0$ without proof.
In \cite[p.~42-3, Remark 5.2]{FortinGlow83}, it is also shown that
\cref{eqn:iter-penalty-2param} converges if $a(\cdot,\cdot)$ satisfies
\ref{asmp:a-coercive-all} provided that
$0 < \rho \leq 2\lambda$ for any $\lambda > 0$, but rates of convergence
are not given.

The geometric convergence of a variable step variant of 
\cref{eqn:iter-penalty-2param} where $\rho$ is replaced by $\rho_n$ is 
proved in
\cite[p.~74 Theorem 4.7]{GiraultRaviart86} under the assumption that
$a_{1/\lambda}(\cdot,\cdot)$ is coercive and 
\begin{align*}
0 < \inf_{n} \rho_n \leq \sup_n \rho_n 
    < 2 \inf_{v \in \discrete{V} : \operator{D} v \neq 0} 
        \frac{a_{1/\lambda}(v, v)}{\|\operator{D} v\|_Q^2}. 
\end{align*}
However, suitable conditions
under which $a_{1/\lambda}(\cdot,\cdot)$ is coercive are not given, 
and the convergence rate in \cite[p.~76, eq.~(4.65)]{GiraultRaviart86}
is not explicit in $\lambda$.

Convergence rates and stopping criterion of the iterated penalty method
are also discussed in \cite[Chapter 13]{BrennerScott08}. 
The precise  conditions on $a(\cdot,\cdot)$ are not 
explicitly stated, but the proof
of the main convergence result, \cite[Theorem 13.1.19]{BrennerScott08},
uses that $a(\cdot,\cdot)$ satisfies \ref{asmp:a-sym-nonneg-coercive-z}.
Theorem 13.1.19 of \cite{BrennerScott08} shows that if $\lambda$ is 
sufficiently large and $0 < \rho < 2 \lambda$, then the iterated 
penalty method converges at a geometric rate. A rate for the case $\rho 
\neq \lambda$ is not explicitly given, while the rate for $\lambda = \rho$
in Theorem 13.1.19 is not as sharp as the rate given by 
\cref{thm:iter-penalty-2param-error}. The stopping criterion 
\cref{eqn:residual-stopping-criteria} is also proposed. We also note
that \cite[Remark 13.2.8]{BrennerScott08} states that taking
$\rho = \lambda$ large adversely affects the approximation of $p_X$.
However, due to the labeling of our iterates in 
\cref{eqn:iter-penalty-2param} and sharper estimates in
\cref{thm:iter-penalty-2param-error}, \cref{cor:iter-penalty-conv},
and \cref{eqn:iter-penalty-1param-error-by-residual},
large values of $\lambda = \rho$ do not adversely affect convergence of 
either variable.

Nochetto and Pyo \cite{NochettoPyo04} consider an inf-sup
stable conforming discretization of the Stokes equations with
$a(\cdot,\cdot) = \nu (\grad \cdot, \grad \cdot)_{L^2(\Omega)} 
= \nu (\cdot,\cdot)_{\discrete{V}}$ with $\nu > 0$, $G = 0$, 
and homogeneous Dirichlet boundary conditions. 
Theorem 1 and Corollary 1 of \cite{NochettoPyo04} show that 
\cref{eqn:iter-penalty-2param} converges at a geometric rate for all 
$\lambda > 0$ and $0 < \rho < 2(\nu + \lambda)$. As shown
in Remark 1.1 of \cite{NochettoPyo04}, the optimal choice of $\rho$
gives a rate of $1-\beta_X^2$. However, the rate in
\cref{eqn:iter-penalty-2param-spd-error-rate} can be made 
arbitrarily small by taking $\lambda = \rho$ large.

Awanou and Lai \cite{AwanouLai05} appear to be the first to show the 
convergence of the augmented Lagrangian method if the symmetric part of 
$a(\cdot,\cdot)$ is positive definite on $\discrete{Z}$ but 
$a(\cdot,\cdot)$ is indefinite and not symmetric in general. Note
that this assumption is strictly stronger than 
\cref{eqn:a-infsup-kernel-discrete}. In particular, 
\cite[Theorem 4]{AwanouLai05} shows that if $\lambda \geq \rho/2 > 0$,
then the augmented Lagrangian method \cref{eqn:iter-penalty-2param} 
converges, while \cite[Theorem 6]{AwanouLai05} show geometric convergence 
if $\lambda = \rho$ is sufficiently large. However,
how large $\lambda$ must be and the exact convergence rate are not 
explicitly stated. 

Huang et al. \cite{HuangDaiOrbanSaunders24} consider the same setting
as \cite{AwanouLai05} in the case $\lambda = \rho$ and provide explicit 
convergence rates for $\lambda$ sufficiently large. In particular, if
we define  
\begin{align*}
    \chi := \inf_{ v \notin \discrete{Z} } 
        \frac{ a(v, v) }{ (\operator{D} v, \operator{D} v)_Q },
\end{align*}
then \cite[Lemma 2.1]{HuangDaiOrbanSaunders24} shows that if $\lambda > 
\max\{-\chi, 0\}$, 
then $\operator{A}_{1/\lambda}$ is positive definite. This 
result is comparable to \cref{lem:tilde-z-eps-coercivity} below,
which shows that if $\lambda > 1/\epsilon_0$, $\operator{A}_{1/\lambda}$
is coercive on $\tilde{\discrete{Z}}$. Note that
\begin{align}
    \label{eqn:tilde-chi-def}
    \tilde{\chi} := \inf_{ v \in 
        \tilde{\discrete{Z}} } 
    \frac{ a(v, v) }{ (\operator{D} v, \operator{D} v)_Q } \geq 
    -M_a \sup_{v \in \tilde{\discrete{Z}}} \frac{\|v\|_V^2}{ 
        \|\operator{D} v\|_Q^2} \geq -\frac{M_a 
        \Upsilon_X^2}{\beta_X^2} = -\frac{1}{\epsilon_0},
\end{align}
and so \cref{lem:tilde-z-eps-coercivity} is less sharp than 
\cite[Lemma 2.1]{HuangDaiOrbanSaunders24}, but is explicit in 
constants related to $a(\cdot,\cdot)$ and $\operator{D}$. 
Theorem 2.1 of \cite{HuangDaiOrbanSaunders24} shows that if $\lambda > 
\max\{-2\chi, 0\}$, then the iterated penalty method converges, 
and the proof shows that the rate is
\begin{align}
    \label{eqn:huang-rate}
    \max_{\mu} \frac{ \lambda^{-1} |\mu|}{|1 + \lambda^{-1} \mu|} = 
    \max_{\mu} \frac{|\mu|}{|\lambda + \mu|}.
\end{align}
Here, $\mu$ is a complex generalized eigenvalue of the system
\begin{align*}
    \tensorbd{A} \vec{x} = \mu \tensorbd{B} \vec{x}
    \qquad
    \vec{x} \in \mathbb{C}^{\dim \discrete{V}}, \ \tensorbd{B} \vec{x} \neq 
    0, 
\end{align*} 
where $\tensorbd{A}$ and $\tensorbd{B}$ are the matrices corresponding
to $a(\cdot,\cdot)$ and $(\operator{D} \cdot, \operator{D} \cdot)_Q$ for 
some fixed basis of $\discrete{V}$. In contrast, 
\cref{cor:iter-penalty-conv} requires $\lambda > \epsilon_0^{-1}(1 + 
\Upsilon_X^{-1})$ (recall that $1 + \Upsilon_X^{-1} \leq 2$) for 
convergence of the iterated penalty method. However, 
it is not straightforward to compare the rate \cref{eqn:huang-rate}
to the rate in \cref{cor:iter-penalty-conv} owing to the presence of 
complex eigenvalues. If $a(\cdot,\cdot)$ is symmetric, then 
all of the eigenvalues are real with $|\mu| \leq 1/\epsilon_0$ arguing as 
in \cref{eqn:tilde-chi-def}, and so 
\begin{align*}
    \max_{\mu} \frac{ \lambda^{-1} |\mu|}{|1 + \lambda^{-1} \mu|} \leq 
    \frac{1}{\lambda \epsilon_0 - 1}.
\end{align*}
The upper bound above is a less optimal rate than in 
\cref{cor:iter-penalty-conv}.

For the method described in \cref{sec:fourth-order} with $A(\cdot,\cdot)$
satisfying \cref{eqn:h2-a-infsup} ($G \equiv 0$), a convergence proof for 
the iterated penalty method with $\lambda = \rho$ sufficiently large was 
given in \cite[Theorem 5.1]{AinsworthParker24c1} with an explicit 
convergence rate. The rates obtained in \cref{cor:iter-penalty-conv} are 
sharper and the condition on the size of $\lambda$ is less restrictive.

\appendix

\section{Well-posedness of the penalty method}
\label{sec:penalty-well-posedness-proofs}

In this section, we prove \cref{thm:saddle-discrete-continuity}.
We begin with some auxiliary results.

\subsection{Auxiliary results}
	
We first note that the inf-sup condition \cref{eqn:b-inf-sup-discrete} 
gives the following lower bound for the operator norm of the adjoint
operator $\operator{D}^{\star}$:
\begin{align}
    \label{eqn:d-adjoint-inf-sup}
    \beta_X \|q\|_Q \leq \sup_{v \in \discrete{V}} 
    \frac{(\operator{D} v, q)_Q}{ \|v\|_V  } = \sup_{v \in \discrete{V}} 
    \frac{(\operator{D}^{\star} q)(v)}{ \|v\|_V  } = 
    \|\operator{D}^{\star} q\|_{\dual{\discrete{V}}}
    \qquad \forall q \in \discrete{Q}.
\end{align}
We also recall \cite[Lemma 4.2.1]{BoffiBrezziFortin13}: if 
$a(\cdot,\cdot)$ is symmetric and nonnegative, then
\begin{alignat}{2}
    \label{eqn:weak-cauchy-schwarz}
    a(v, w)^2 &\leq a(v, v) a(w, w) \qquad & &\forall v, w 
    \in \discrete{V}, \\
    \label{eqn:a-sym-nonneg-av-dual-by-a}
    \|\operator{A} v\|_{\dual{\discrete{V}}}^2 &\leq M_a a(v, v) 
    \qquad & &\forall v \in \discrete{V}.
\end{alignat}
Using the shorthand notation $|\cdot|_a^2 := a(\cdot,\cdot)$,
\cref{eqn:weak-cauchy-schwarz} shows that $|\cdot|_a$ is a semi-norm
if $a(\cdot,\cdot)$ satisfies \ref{asmp:a-sym-nonneg-coercive-z}
and a norm if $a(\cdot,\cdot)$ satisfies \ref{asmp:a-coercive-all}.
    
The next result shows that 
$\discrete{Z}$ and $\tilde{\discrete{Z}}$ (or $\hat{\discrete{Z}}$) are 
a stable decomposition of $\discrete{V}$.	
\begin{lemma}
    \label{lem:v-z-tildez-decomp}
    Let $\check{\discrete{Z}} = \tilde{\discrete{Z}}$ or 
    $\check{\discrete{Z}} = 
    \hat{\discrete{Z}}$. Then, there holds $\discrete{V} = \discrete{Z} 
    \oplus 
    \check{\discrete{Z}}$. In 
    particular, for every $u \in \discrete{V}$, there exists unique $z \in 
    \discrete{Z}$ and 
    $\check{z} \in \check{\discrete{Z}}$ such that $u = z + \check{z}$ and
    \begin{subequations}
        \label{eqn:v-z-tildez-decomp-cont}	
        \begin{alignat}{4}
            \label{eqn:v-z-tildez-decomp-cont-v}
            \|z\|_V &\leq \Phi_X \|u\|_V   
            \quad & &\text{and} \quad & \|\check{z}\|_V &\leq \Upsilon_X 
            \|u\|_V, \qquad & & \\
            \label{eqn:v-z-tildez-decomp-cont-a}
            |z|_a &\leq |u|_a \quad & &\text{and} \quad & |\check{z}|_a 
            &\leq 
            |u|_a \qquad & &\text{if \ref{asmp:a-sym-nonneg-coercive-z} 
                holds},
        \end{alignat}	
    \end{subequations}
    where $\Phi_X$ and $\Upsilon_X$ are defined in 
    \cref{eqn:phix-upsilonx-def}. 
\end{lemma}
\begin{proof}
    \noindent \textbf{Step 1: General case. } Suppose first that $u \in 
    \discrete{Z} \cap 
    \tilde{\discrete{Z}}$. Then, \cref{eqn:a-infsup-kernel-discrete} gives
    \begin{align*}
        \alpha_X \|u\|_{V} \leq \sup_{w \in \discrete{Z}} 
        \frac{a(u, w)}{\|w\|_V} = 0 \implies u \equiv 0. 
    \end{align*}		
    Similarly, if $u \in \discrete{Z} \cap \hat{\discrete{Z}}$, then
    replacing $a(u, w)$ with $a(w, u)$ above gives $u \equiv 0$.
    Consequently, ${\discrete{Z} \cap \tilde{\discrete{Z}} = \discrete{Z} 
    \cap \hat{\discrete{Z}} = \{ 0 \}}$.
    
    Now let $u \in \discrete{V}$. 
    Thanks to \cref{eqn:a-infsup-kernel-discrete} and the boundedness of 
    $a(\cdot,\cdot)$, there exists $z \in \discrete{Z}$ such that ${a(z 
    - u, w) = 0}$ for all $w \in \discrete{Z}$ and 
    \begin{align*}
        \alpha_X \|z\|_V \leq \sup_{w \in \discrete{Z}} 
        \frac{a(z, w)}{\|w\|_V} = \sup_{w \in \discrete{Z}} 
        \frac{a(u, w)}{\|w\|_V} \leq M_a \|u\|_V \implies \|z\|_V \leq 
        \frac{M_a}{\alpha_X} \|u\|_V.
    \end{align*}
    Then, $\tilde{z} := u - z$ satisfies $\tilde{z} \in 
    \tilde{\discrete{Z}}$ 
    by construction and \cref{eqn:v-z-tildez-decomp-cont} follows from the 
    triangle inequality. If we instead take $z\in \discrete{Z}$ to satisfy 
    $a(w, z-u) = 0$
    for all $w\in \discrete{Z}$, 
    then the 
    analogous arguments show that $\hat{z} := u - z$ satisfies $\hat{z} \in 
    \hat{\discrete{Z}}$ and $\|\hat{z}\|_V \leq \Phi_X \|u\|_V$. \\
    
    \noindent \textbf{Step 2: $a(\cdot,\cdot)$ satisfies 
        \ref{asmp:a-sym-nonneg-coercive-z}. } We use 
    \cref{eqn:weak-cauchy-schwarz} to refine our estimate for $z$:
    \begin{align*}
        |z|_a^2 = a(u, z) \leq 
        |z|_a |u|_a \implies |z|_a \leq |u|_a. 
    \end{align*} 
    Similarly, 
    \begin{align*}
        |\tilde{z}|_a^2 &= a(\tilde{z}, u - z) = a(\tilde{z}, u) \leq 
        |\tilde{z}|_a 
        |u|_a \implies |\tilde{z}|_a \leq |u|_a.
    \end{align*}
    Inequality \cref{eqn:a-nonnegative-coercive-kernel} then gives
    \begin{align*}
        \|z\|_V \leq \sqrt{\frac{M_a}{\tilde{\alpha}_X}} \|u\|_V
        \implies \|\check{z}\|_V \leq \|z\|_V + \|u\|_V \leq  
        \left( 1 + \sqrt{\frac{M_a}{\tilde{\alpha}_X}} \right) 
        \| u \|_{V}.
    \end{align*}
    
    \noindent \textbf{Step 3: $a(\cdot,\cdot)$ satisfies 
        \ref{asmp:a-sym-coercive-all}. } Using 
    \cref{eqn:a-coercive-all}, we have
    \begin{align*}
        \| \tilde{z} \|_{V} 
        \leq \frac{1}{\sqrt{\tilde{\alpha}_X}} |\tilde{z}|_a 
        \leq \frac{1}{\sqrt{\tilde{\alpha}_X}} |u|_a 
        \leq  \sqrt{\frac{M_a}{\tilde{\alpha}_X}} \|u\|_V,
    \end{align*}
    with the analogous estimate for $z$. \\
    
    \noindent \textbf{Step 4: $a(\cdot,\cdot)$ satisfies 
    \ref{asmp:a-coercive-all}. } By coercivity \cref{eqn:a-coercive-all}, 
    we have
    \begin{align*}
        \tilde{\alpha}_X \|z\|_{V}^2 \leq a(z, z) = a(u, z) \leq M_a 
        \|u\|_V \|z\|_V \implies \|z\|_V \leq \frac{M_a}{\tilde{\alpha}_X} 
        \|u\|_V.
    \end{align*}
    Similarly, we have
    \begin{align*}
        |\tilde{z}|_a^2 &= a(\tilde{z}, u - z) = a(\tilde{z}, u) \leq 
        M_a \|\tilde{z}\|_V \|u\|_V \implies \|\tilde{z}\|_V \leq 
        \Upsilon_X \|u\|_V, \\
        |\hat{z}|_a^2 &= a(u-z,\hat{z}) = a(u, \hat{z}) \leq 
        M_a \|\hat{z}\|_V \|u\|_V \implies \|\hat{z}\|_V \leq \Upsilon_X 
        \|u\|_V.
    \end{align*}
\end{proof}	

Since $\discrete{Z}$ is the kernel of $\operator{D} : \discrete{V} \to 
\discrete{Q}$, the quantity $\|\operator{D} \cdot\|_{Q}$ is a norm on 
$\tilde{\discrete{Z}}$ and $\hat{\discrete{Z}}$ equivalent to 
$\|\cdot\|_{V}$ with equivalence constants given by the following result.	
\begin{lemma}
    \label{lem:tilde-z-div-continuity}
    Let $\check{\discrete{Z}} = \tilde{\discrete{Z}}$ or 
    $\check{\discrete{Z}} = 
    \hat{\discrete{Z}}$. Then, $\operator{D} : \check{\discrete{Z}} \to 
    \discrete{Q}$ is bijective
    and there holds
    \begin{align}
        \label{eqn:tilde-z-div-continuity}
        \frac{1}{M_{\operator{D}}} \|\operator{D} \check{z}\|_Q \leq 
        \|\check{z}\|_V \leq \frac{\Upsilon_X}{\beta_X}  \|\operator{D} 
        \check{z}\|_{Q} \qquad \forall 
        \check{z} \in \check{\discrete{Z}},
    \end{align}
    where $M_{\operator{D}}$ is defined in \cref{eqn:a-d-bounded},
    $\Upsilon_X$ is defined in \cref{eqn:phix-upsilonx-def},
    and $\beta_X$ is defined in \cref{eqn:b-inf-sup-discrete}.
    Moreover, if 
    \ref{asmp:a-sym-nonneg-coercive-z} holds, then
    \begin{align}
        \label{eqn:tilde-z-div-continuity-a-sspd}
        |\check{z}|_a \leq \frac{\sqrt{M_a}}{ \beta_X} \|\operator{D} 
        \check{z}\|_{Q} \qquad \forall 
        \check{z} \in \check{\discrete{Z}}.
    \end{align}
\end{lemma}	
\begin{proof}
    Since $\operator{D} : \discrete{V} \to \discrete{Q}$ is surjective by 
    assumption, $\operator{D} : \check{\discrete{Z}} \to \discrete{Q}$ is 
    surjective. Injectivity will follow once we establish 
    \cref{eqn:tilde-z-div-continuity}.
    
    Let $\check{z} \in \check{\discrete{Z}}$. Thanks to 
    \cref{eqn:b-inf-sup-discrete},
    there exists $v \in \discrete{V}$ such that $\operator{D} v = 
    \operator{D} \check{z}$ and
    \begin{align}
        \label{eqn:proof:invert-d}
        \|v\|_V \leq \frac{1}{\beta_X} \|\operator{D} \check{z} \|_Q 
        \implies a(v, v) \leq \frac{M_a}{\beta_X^2} \|\operator{D} 
        \check{z} \|_Q^2.
    \end{align}
    Let $v = w + \check{w}$ with $w \in \discrete{Z}$ and $\check{w} \in 
    \check{\discrete{Z}}$ be given by \cref{lem:v-z-tildez-decomp}. Since 
    $\check{z} - \check{w} \in \discrete{Z} \cap \tilde{\discrete{Z}}$, 
    $\check{w} = \check{z}$. The rightmost inequality 
    \cref{eqn:tilde-z-div-continuity} now 
    follows from \cref{eqn:proof:invert-d,eqn:v-z-tildez-decomp-cont-v}. 
    Similarly, if \ref{asmp:a-sym-nonneg-coercive-z} holds, then 
    \cref{eqn:tilde-z-div-continuity-a-sspd} follows from 
    \cref{eqn:proof:invert-d,eqn:v-z-tildez-decomp-cont-a}. 
    The leftmost inequality in \cref{eqn:tilde-z-div-continuity} 
    is a restatement of the continuity of $\operator{D}$.
\end{proof}

The next result	will be useful in a number of proofs below.
\begin{lemma}
    Suppose that $\tilde{z} \in \tilde{\discrete{Z}}$ and $\hat{z} \in 
    \hat{\discrete{Z}}$ satisfy $\operator{D} \tilde{z} = \operator{D} 
    \hat{z}$. Then, there holds
    \begin{align}
        \label{eqn:interchange-tilde-hat-same-d}
        a(\tilde{z}, \tilde{z}) = a(\tilde{z}, \hat{z}) = a(\hat{z}, 
        \hat{z}).
    \end{align}
\end{lemma}
\begin{proof}
    Since $\tilde{z} - \hat{z} \in \discrete{Z}$, there holds
    \begin{align*}
        a(\tilde{z}, \tilde{z}) = a(\tilde{z}, \hat{z}) + a(\tilde{z}, 
        \tilde{z} - \hat{z}) = a(\tilde{z}, \hat{z}) = a(\tilde{z} - 
        \hat{z}, 
        \hat{z}) + a(\hat{z}, \hat{z}) = a(\hat{z}, \hat{z}).
    \end{align*}
\end{proof}

The following result shows that $a_{\epsilon}(\cdot,\cdot)$ is coercive on 
$\tilde{\discrete{Z}}$ and $\hat{\discrete{Z}}$ for $\epsilon$ sufficiently 
small.	
\begin{lemma}
    \label{lem:tilde-z-eps-coercivity}
    Let $\epsilon_0$ be defined as in \cref{eqn:eps0-def}.
    For all $\epsilon \in (0, \epsilon_0)$, there holds
    \begin{align}
        \label{eqn:tilde-z-aeps-coercivity}
        a_{\epsilon}(\check{z}, \check{z}) \geq \left( \frac{1}{\epsilon} - 
        \frac{1}{\epsilon_0} \right) \|\operator{D} \check{z} \|_{Q}^2 
        \geq M_a \left( \frac{\epsilon_0 - \epsilon}{\epsilon} \right) 
        \|\check{z}\|_{V}^2 
        \qquad \forall \check{z} \in \tilde{\discrete{Z}} \cup 
        \hat{\discrete{Z}}.
    \end{align}
    If $a(\cdot,\cdot)$ satisfies \ref{asmp:a-sym-nonneg-coercive-z}, then 
    for all $\epsilon >0$, 
    there holds
    \begin{align}
        \label{eqn:tilde-z-aeps-coercivity-sym-nonneg-coercive-z}
        a_{\epsilon}(\tilde{z}, \tilde{z}) \geq \frac{1}{\epsilon} 
        \|\operator{D} 
        \tilde{z}\|_{Q}^2 \geq \frac{1}{\epsilon} \left( 
        \frac{\beta_X}{\Upsilon_X} \right)^2 \|\tilde{z}\|_{V}^2 \qquad 
        \forall \tilde{z} \in 
        \tilde{\discrete{Z}}.
    \end{align}
    If $a(\cdot,\cdot)$ satisfies \ref{asmp:a-coercive-all}, then for all
    $\epsilon > 0$, there holds
    \begin{subequations}
        \begin{alignat}{2}
            \label{eqn:tilde-z-aeps-coercivity-coercive-all-d}
            a_{\epsilon}(\check{z}, \check{z}) &\geq \left( 
            \frac{\tilde{\alpha}_X}{ 
                M_{\operator{D}}^2} + \frac{1}{\epsilon} \right) 
            \|\operator{D} \check{z} \|_{Q}^2 \qquad & &\forall \check{z} 
            \in \tilde{\discrete{Z}} \cup 
            \hat{\discrete{Z}} \\
            \label{eqn:tilde-z-aeps-coercivity-coercive-all-v}
            a_{\epsilon}(\check{z}, \check{z}) &\geq 
            \left( \tilde{\alpha}_X + \frac{1}{\epsilon} 
            \left( 
            \frac{\beta_X}{\Upsilon_X} \right)^2 \right) 
            \|\check{z}\|_{V}^2 
            \qquad & &\forall \check{z} \in \tilde{\discrete{Z}} \cup 
            \hat{\discrete{Z}},
        \end{alignat}
    \end{subequations}
    while if $a(\cdot,\cdot)$ satisfies \ref{asmp:a-sym-coercive-all}, then 
    we also have
    \begin{align}
        \label{eqn:tilde-z-aeps-coercivity-sym-coercive-all-v-a}
        a_{\epsilon}(\tilde{z}, \tilde{z}) &\geq \left( 1 + 
        \frac{1}{\epsilon} \left( \frac{\beta_X}{\sqrt{M_a}} \right)^2 
        \right) |\tilde{z}|_{a}^2
        \qquad \forall \tilde{z} \in \tilde{\discrete{Z}}.
    \end{align}
\end{lemma}
\begin{proof}
    Thanks to 
    \cref{eqn:tilde-z-div-continuity}, there holds
    \begin{align*}
        a_{\epsilon}(\check{z}, \check{z}) = a(\check{z}, \check{z}) + 
        \frac{1}{\epsilon} \| 
        \operator{D} \check{z} \|_{Q}^2 
        &\geq \left( \frac{1}{\epsilon} - \frac{M_a 
        \Upsilon_X^2}{\beta_X^2} \right) 
        \|\operator{D} \check{z} \|_{Q}^2 
        \geq  \left( \frac{1}{\epsilon}  
        \left( 
        \frac{\beta_X}{\Upsilon_X} \right)^2 - M_a \right) 
        \|\check{z}\|_{V}^2 
        \qquad \forall \check{z} \in \tilde{\discrete{Z}} \cup 
        \hat{\discrete{Z}},
    \end{align*}
    which is exactly \cref{eqn:tilde-z-aeps-coercivity}. Inequalities 
    \cref{eqn:tilde-z-aeps-coercivity-sym-nonneg-coercive-z},
    \cref{eqn:tilde-z-aeps-coercivity-coercive-all-v},
    and \cref{eqn:tilde-z-aeps-coercivity-sym-coercive-all-v-a}
    similarly follow from 
    \cref{eqn:tilde-z-div-continuity,eqn:tilde-z-div-continuity-a-sspd,%
        eqn:a-coercive-all}, 
    while inequality 
    \cref{eqn:tilde-z-aeps-coercivity-coercive-all-d} follows from 
    \cref{eqn:a-coercive-all,eqn:a-d-bounded}.
\end{proof}

In each remaining subsections, we will prove 
\cref{thm:saddle-discrete-continuity} under each assumption
on $a(\cdot,\cdot)$.

\subsection{Proof of \cref{thm:saddle-discrete-continuity} for general 
$a(\cdot,\cdot)$}
\label{sec:saddle-discrete-continuity-proof-a-gen}

Uniqueness of solutions to \cref{eqn:saddle-discrete-penalty} will follow
from \cref{eqn:saddle-perturb-cont}. 
For the sake of \cref{rem:results-valid-v-q},
we present an existence argument that is valid independent on the dimension 
of $\discrete{V}$. The case 
$\epsilon = 0$ follows from standard Babu\v{s}ka-Brezzi theory 
\cite{Babuska71,Brezzi74,Necas62}, so suppose $\epsilon > 0$.
Thanks to \cref{eqn:a-infsup-kernel-discrete} 
(or \cref{eqn:a-inf-sup-kernel} if $\discrete{V} = V$), there 
exists $z \in \discrete{Z}$ such that 
\begin{align*}
    a(z, w) = F(w) \qquad \forall w \in \discrete{Z}.
\end{align*}
On appealing to \cref{lem:tilde-z-eps-coercivity,lem:tilde-z-div-continuity} 
and \cref{eqn:interchange-tilde-hat-same-d}, there exists 
$\varpi > 0$ such that 
\begin{align*}
    \varpi < \newinf_{\tilde{v} \in \tilde{\discrete{Z}}} 
                \sup_{\hat{w} \in \hat{\discrete{Z}}}
                \frac{a_{\epsilon}(\tilde{v}, \hat{w})}{
                    \|\tilde{v}\|_V \|\hat{w}\|_V}
    \quad \text{and} \quad 
    \varpi < \newinf_{\hat{w} \in \hat{\discrete{Z}}} 
                \sup_{\tilde{v} \in \tilde{\discrete{Z}}}
                \frac{a_{\epsilon}(\tilde{v}, \hat{w})}{
                \|\tilde{v}\|_V \|\hat{w}\|_V},
\end{align*}
and so Babu\v{s}ka theory \cite{Babuska71,Necas62} shows that
there exists $\tilde{z} \in \tilde{\discrete{Z}}$ such that 
\begin{align*}
    a_{\epsilon}(\tilde{z}, \hat{w}) = F(\hat{w}) 
        + \epsilon^{-1} G(\operator{D} \hat{w}) 
    \qquad \forall \hat{w} \in \hat{\discrete{Z}}.
\end{align*}
One may directly verify that $u_X = z + \tilde{z}$ satisfies 
\cref{eqn:saddle-discrete-penalty-decoupled-1} thanks to
\cref{eqn:interchange-tilde-hat-same-d},
and $p_X$ can be defined via \cref{eqn:saddle-discrete-penalty-decoupled-2}.

We now turn to \cref{eqn:saddle-perturb-cont}.
We treat the cases $\epsilon = 0$ and $\epsilon > 0$ and the cases
$G = 0$ and $F = 0$ separately, and the result will follow by linearity
and the triangle inequality.  Let 
$u_{X, \epsilon} = z_{X, \epsilon}' + \hat{z}_{X, \epsilon}$, 
where $z_{X, \epsilon}' \in \discrete{Z}$ and 
$\hat{z}_{X, \epsilon} \in \hat{\discrete{Z}}$ are given by 
\cref{lem:v-z-tildez-decomp}, and $p_{X, \epsilon} \in \discrete{Q}$
be any solution to \cref{eqn:saddle-discrete-penalty}. \\

\noindent \textbf{Step 1: $\epsilon = 0$. } First consider the case 
$\epsilon = 0$. \\

\noindent \textbf{Part 1: $G = 0$. } Thanks to 
\cref{eqn:saddle-discrete-2}, $u_X 
\in \discrete{Z}$, and so $\tilde{z}_{X} = 0$ and
\begin{align*}
    a(u_X, z) = F(z) \quad \forall z \in \discrete{Z} \implies 
    \|u_X\|_{V} \leq \frac{1}{\alpha_X} \|F\|_{\dual{\discrete{Z}}},
\end{align*}
which corresponds to the value of $C_{X}^{[1]}$ in 
\cref{tab:saddle-perturb-cont-const}.
Thanks to \cref{lem:tilde-z-div-continuity}, there exists $v \in 
\hat{\discrete{Z}}$ satisfying $\operator{D} v = p_X$ and 
\cref{eqn:tilde-z-div-continuity}, and so
\begin{align*}
    \|p_X\|_Q = \frac{(\operator{D} v, p_X)_Q}{\|p_X\|_Q} = 
    \frac{F(v)}{\|p_X\|_Q} \leq 
    \frac{\|v\|_{V}}{ \|p_X\|_Q } \cdot 
    \|F\|_{\dual{\hat{\discrete{Z}}}}  
    \leq \frac{\Upsilon_X}{\beta_X}
    \|F\|_{\dual{\hat{\discrete{Z}}}},
\end{align*} 
which corresponds to the value of $C_{X, 0}^{[3]}$ in 
\cref{tab:saddle-perturb-cont-const}. \\

\noindent \textbf{Part 2: $F = 0$. } Thanks to 
\cref{eqn:saddle-discrete-1},  
$u_X \in \tilde{\discrete{Z}}$ and $\|\operator{D} u_X\|_Q = 
\|G\|_{\dual{\discrete{Q}}}$ and so applying 
\cref{eqn:tilde-z-div-continuity} gives
\begin{align*}
    \|u_X\|_V \leq \frac{\Upsilon_X}{\beta_X} 
    \|\operator{D} 	u_X\|_{Q} = \frac{\Upsilon_X}{\beta_X} \|G
    \|_{\dual{\discrete{Q}}},
\end{align*}
which corresponds to the value of $C_{X, 0}^{[3]}$ in 
\cref{tab:saddle-perturb-cont-const}.
The inf-sup condition \cref{eqn:b-inf-sup-discrete} and 
\cref{eqn:saddle-discrete-penalty-1} then give
\begin{align}
    \label{eqn:proof:px-cont-F-zero}
    \beta_X \|p_X\|_Q \leq \sup_{v \in \discrete{V}} 
    \frac{(\operator{D} v, 
        p_X)_Q}{\|v\|_{V}} = \sup_{v \in \discrete{V}} 
    \frac{a(u_X, 		
        v)}{\|v\|_{V}} \leq M_a \|u_X\|_V \leq \frac{M_a 
        \Upsilon_X}{\beta_X} \|G\|_{\dual{\discrete{Q}}},
\end{align}
which corresponds to the value of $C_{X, 0}^{[4]}$ in 
\cref{tab:saddle-perturb-cont-const}. \\

\noindent \textbf{Step 2: $\epsilon > 0$. } Now consider the case 
$\epsilon > 0$. \\

\noindent \textbf{Part 1: $G = 0$. } Inequality 
\cref{eqn:saddle-perturb-z-cont} 
follows from the same arguments as for $\epsilon = 0$.
Thanks to 
\cref{eqn:tilde-z-aeps-coercivity,eqn:interchange-tilde-hat-same-d}, we 
also have
\begin{align*}
    a_{\epsilon}(u_{X, \epsilon}, \hat{z}_{X, \epsilon}) = 
    a(\tilde{z}_{X, \epsilon}, \hat{z}_{X, \epsilon}) + 
    \frac{1}{\epsilon} (\operator{D} \tilde{z}_{X, \epsilon}, 
    \operator{D} \hat{z}_{X, 
        \epsilon} )
    &= a_{\epsilon}(\tilde{z}_{X, \epsilon}, \tilde{z}_{X, \epsilon}) 
    \geq \left( 
    \frac{1}{\epsilon} - 
    \frac{1}{\epsilon_0} \right) \| \operator{D} \tilde{z}_{X, 
        \epsilon} \|_{Q}^2,
\end{align*}
while \cref{eqn:tilde-z-div-continuity} gives
\begin{align*}
    |F(\hat{z}_{X, \epsilon})| \leq \|F\|_{\dual{\hat{\discrete{Z}}}} 
    \| \hat{z}_{X, 
        \epsilon}\|_{V} 
    \leq \frac{\Upsilon_X}{\beta_X} \|F\|_{\dual{\hat{\discrete{Z}}}} 
    \|\operator{D} 
    \hat{z}_{X, \epsilon}\|_{Q} = \frac{\Upsilon_X}{\beta_X} 
    \|F\|_{\dual{\hat{\discrete{Z}}}} 
    \|\operator{D} \tilde{z}_{X, \epsilon}\|_{Q}.
\end{align*}
Consequently, there holds
\begin{align}
    \label{eqn:proof:d-tilde-z-bound}
    \|\operator{D} \hat{z}_{X, \epsilon}\|_{Q} = \|\operator{D} 
    \tilde{z}_{X, \epsilon}\|_{Q} 
    \leq 
    \frac{\epsilon \Upsilon_X}{\beta_X} \left( 
    \frac{\epsilon_0}{\epsilon_0 - \epsilon} \right) 
    \|F\|_{\dual{\hat{\discrete{Z}}}}
\end{align}
Applying \cref{eqn:tilde-z-div-continuity} once again gives
\begin{align*}
    \max\left\{\| \tilde{z}_{X, \epsilon} \|_V, \| \hat{z}_{X, 
        \epsilon} \|_V \right\} \leq  
    \frac{\epsilon \Upsilon_X^2}{\beta_X^2}  
    \left( 
    \frac{\epsilon_0}{\epsilon_0 - \epsilon} \right) 
    \|F\|_{\dual{\hat{\discrete{Z}}}},
\end{align*}
which corresponds to the value of $C_{X, \epsilon}^{[2]}$ in 
\cref{tab:saddle-perturb-cont-const}.
Thanks to \cref{eqn:saddle-discrete-penalty-decoupled-2}, 
$p_{X, \epsilon} = \epsilon^{-1} \operator{D} \tilde{z}_{X, 
    \epsilon}$, and so the value of $C_{X, \epsilon}^{[3]}$ in 
\cref{tab:saddle-perturb-cont-const} follows from 
\cref{eqn:proof:d-tilde-z-bound}. \\

\noindent \textbf{Part 2: $F = 0$. } The same arguments as in the case 
$\epsilon = 0$ show that $z_{X, \epsilon} \equiv 0$.
Choosing $v \in \tilde{\discrete{Z}}$ in 
\cref{eqn:saddle-discrete-penalty-decoupled-1} and applying 
\cref{eqn:tilde-z-aeps-coercivity} gives
\begin{align*}
    a_{\epsilon}(\tilde{z}_{X, \epsilon}, v) = \epsilon^{-1} 
    G(\operator{D} 
    v) \quad \forall v \in 
    \tilde{\discrete{Z}} \implies \left( \frac{1}{\epsilon} - 
    \frac{1}{\epsilon_0} \right) \|\operator{D} 
    \tilde{z}_{X, \epsilon}\|_Q^2 
    &\leq  
    \frac{1}{\epsilon} \|G\|_{\dual{\discrete{Q}}} \|\operator{D} 
    \tilde{z}_{X, \epsilon}\|_Q, 
\end{align*}
and so \cref{eqn:tilde-z-div-continuity} gives
\begin{align*}
    \|\operator{D} \tilde{z}_{X, \epsilon}\|_Q \leq 
    \frac{\epsilon_0}{\epsilon_0 - 
        \epsilon} \|G\|_{\dual{\discrete{Q}}} 
    \implies
    \|\tilde{z}_{X, \epsilon}\|_V &\leq \frac{\Upsilon_X}{\beta_X} 
    \left(
    \frac{\epsilon_0}{\epsilon_0 - \epsilon} \right) 
    \|G\|_{\dual{\discrete{Q}}},
\end{align*}
which corresponds to the value of $C_{X, \epsilon}^{[3]}$ in 
\cref{tab:saddle-perturb-cont-const}.
The inf-sup condition \cref{eqn:b-inf-sup-discrete} and 
\cref{eqn:saddle-discrete-penalty-1} then give
\begin{align*}
    \beta_X \|p_{X,\epsilon}\|_Q \leq \sup_{v \in \discrete{V}} 
    \frac{(\operator{D} v, p_{X, \epsilon})_Q}{\|v\|_{V}} 
    = \sup_{v \in \discrete{V}} \frac{a(\tilde{z}_{X, \epsilon}, 
        v)}{\|v\|_{V}} 
    &\leq M_a \|\tilde{z}_{X, \epsilon}\|_V 
    \leq \frac{M_a \Upsilon_X}{\beta_X} 
    \left( \frac{\epsilon_0}{\epsilon_0 - \epsilon}  \right) 
    \|G\|_{\dual{\discrete{Q}}},
\end{align*}
which corresponds to the value of $C_{X, \epsilon}^{[4]}$ in 
\cref{tab:saddle-perturb-cont-const}. \hfill \qedsymbol

\subsection{Proof of \cref{thm:saddle-discrete-continuity} when 
    $a(\cdot,\cdot)$ satisfies \ref{asmp:a-sym-nonneg-coercive-z}}
\label{sec:saddle-discrete-continuity-proof-a-sym-nonneg}	

We proceed as in \cref{sec:saddle-discrete-continuity-proof-a-gen}. \\
	
\noindent \textbf{Step 1: $\epsilon = 0$. } First consider the case 
$\epsilon = 0$. \\

\noindent \textbf{Part 1: $G = 0$. } Using 
\cref{eqn:a-nonnegative-coercive-kernel}, we may replace $\alpha_X$
with $\tilde{\alpha}_X$ in Step 1: Part 1 of 
\cref{sec:saddle-discrete-continuity-proof-a-gen}. The
values of $C_{X, 0}^{[i]}$, $i = 2,3$, in the 
\ref{asmp:a-sym-nonneg-coercive-z} column of
\cref{tab:saddle-perturb-cont-const} are the same as in the
general case. \\

\noindent \textbf{Part 2: $F = 0$. } Note that $C_{X, 0}^{[3]}$ 
in the \ref{asmp:a-sym-nonneg-coercive-z} column of 
\cref{tab:saddle-perturb-cont-const} is the same 
value in the general case, so we need only consider the value of 
$C_{X, 0}^{[4]}$. Inequality \cref{eqn:d-adjoint-inf-sup}, equation
\cref{eqn:saddle-discrete-1}, and inequality  
\cref{eqn:a-sym-nonneg-av-dual-by-a} give
\begin{align*}
    \beta_X \|p_X \|_Q = \|\operator{D}^{\star} p_X\|_{\dual{\discrete{V}}}
        = \| \operator{A} u_X \|_{\dual{\discrete{V}}} \leq \sqrt{M_a} 
        |u_X|_a.
\end{align*}
Since $u_X \in \tilde{\discrete{Z}}$ and $\operator{D} u_X = 
\operator{R}_{\discrete{Q}}^{-1} G$, inequality 
\cref{eqn:tilde-z-div-continuity} gives
$\|p_X\|_{Q} \leq M_a \beta_X^{-2} \|G\|_{\dual{\discrete{Q}}}$,
which corresponds to the value of $C_{X, 0}^{[4]}$ in the 
\ref{asmp:a-sym-nonneg-coercive-z} column of 
\cref{tab:saddle-perturb-cont-const}. \\

\noindent \textbf{Step 2: $\epsilon > 0$. } Now consider the case 
$\epsilon > 0$. \\

\noindent \textbf{Part 1: $G = 0$. } The value of $C_{X}^{[1]}$
follows from the same arguments as for $\epsilon = 0$. Since $\hat{z}_{X, 
\epsilon} = \tilde{z}_{X, \epsilon}$, we now have
\begin{align}
    \label{eqn:proof:sym-tildez-eqn}
    a_{\epsilon}(\tilde{z}_{X, \epsilon}, \tilde{z}_{X, \epsilon}) = 
    a_{\epsilon}(u_{X, 
        \epsilon}, 
    \tilde{z}_{X, \epsilon}) = F(\tilde{z}_{X, \epsilon}),
\end{align} 
and so \cref{eqn:tilde-z-aeps-coercivity-sym-nonneg-coercive-z} gives
\begin{align*}
    \| \tilde{z}_{X, \epsilon}\|_{V} \leq \frac{\epsilon 
        \Upsilon_X^2}{\beta_X^2} 
    \|F\|_{\dual{\tilde{\discrete{Z}}}} \quad \text{and} \quad 
    \|\operator{D} 
    \tilde{z}_{X, \epsilon}\|_{Q}^2 \leq \epsilon 
    \|F\|_{\dual{\tilde{\discrete{Z}}}} \| 
    \tilde{z}_{X, \epsilon}\|_{V} \leq \frac{\epsilon^2 
        \Upsilon_X^2}{\beta_X^2} 
    \|F\|_{\dual{\tilde{\discrete{Z}}}}^2.
\end{align*}
The values of $C_{X, \epsilon}^{[2]}$ and  $C_{X, \epsilon}^{[3]}$ in 
the \ref{asmp:a-sym-nonneg-coercive-z} column of
\cref{tab:saddle-perturb-cont-const} now follow since 
$p_{X, \epsilon} = \epsilon^{-1} \operator{D} \tilde{z}_{X, 
    \epsilon}$. \\

\noindent \textbf{Part 2: $F = 0$. } The same arguments as in the case 
$\epsilon = 0$ show that $z_{X, \epsilon} = 0$.	Choosing
$v = \tilde{z}_{X, \epsilon}$ in 
\cref{eqn:saddle-discrete-penalty-decoupled-1} gives
\begin{align*}
    |\tilde{z}_{X, \epsilon}|_{a_{\epsilon}}^2 = \epsilon^{-1} 
    G(\operator{D} 
    \tilde{z}_{X, \epsilon}) \leq \frac{1}{\sqrt{\epsilon}} 
    \|G\|_{\dual{\discrete{Q}}} 
    |\tilde{z}_{X, \epsilon}|_{a_{\epsilon}} \implies 
    |\tilde{z}_{X, \epsilon}|_{a_{\epsilon}}^2 \leq \frac{1}{\epsilon} 
    \|G\|_{\dual{\discrete{Q}}}^2,
\end{align*}
and so the value of $C_{X, \epsilon}^{[3]}$ in 
the \ref{asmp:a-sym-nonneg-coercive-z} column of
\cref{tab:saddle-perturb-cont-const} follows from 
\cref{eqn:tilde-z-aeps-coercivity-sym-nonneg-coercive-z}. 	
Now, choosing 
$v = u_{X, \epsilon} = \tilde{z}_{X, \epsilon}$ in 
\cref{eqn:saddle-discrete-penalty-1} and $q = p_{X, \epsilon}$ in 
\cref{eqn:saddle-discrete-penalty-2} and summing the two equations, 
we obtain
\begin{align}
    \label{eqn:proof:add-saddle-perturb}
    |\tilde{z}_{X, \epsilon}|_a^2 + \epsilon \|p_{X, \epsilon}\|_Q^2 = 
    -G(p_{X, \epsilon}).
\end{align}
Applying \cref{eqn:a-sym-nonneg-av-dual-by-a,eqn:d-adjoint-inf-sup} 
and using \cref{eqn:saddle-discrete-penalty-1} then gives
\begin{align*}
    |\tilde{z}_{X, \epsilon}|_a^2 \geq 
    \frac{1}{M_a} \| \operator{A} \tilde{z}_{X, 
        \epsilon}\|_{\dual{\discrete{V}}}^2 
    = \frac{1}{M_a} \| \operator{D}^{\star} p_{X, 
        \epsilon}\|_{\dual{\discrete{V}}}^2 
    \geq \frac{\beta_X^2}{M_a} \|p_{X, \epsilon}\|_{Q}^2,
\end{align*}
and so the value of $C_{X, \epsilon}^{[4]}$ in 
the \ref{asmp:a-sym-nonneg-coercive-z} column of
\cref{tab:saddle-perturb-cont-const} follows as
\begin{align*}
    \left( \frac{\beta_X^2}{M_a} + \epsilon \right) \|p_{X, 
        \epsilon}\|_{Q}^2 \leq -G(p_{X, \epsilon}) \implies \|p_{X, 
        \epsilon}\|_{Q} \leq \frac{M_a}{\epsilon M_a + \beta_X^2} 
    \|G\|_{\dual{\discrete{Q}}},
\end{align*}
which completes the proof assuming $a(\cdot,\cdot)$ satisfies 
\ref{asmp:a-sym-nonneg-coercive-z}. \hfill \qedsymbol

\subsection{Proof of \cref{thm:saddle-discrete-continuity} when 
    $a(\cdot,\cdot)$ satisfies \ref{asmp:a-coercive-all}}
\label{sec:saddle-discrete-continuity-proof-a-coercive-all}	

For $\epsilon = 0$, the values of $C_{X}^{[1]}$ and 
$C_{X, \epsilon}^{[i]}$, $i=2,\ldots,4$, are the
same as in the general case, replacing $\alpha_X$ with $\tilde{\alpha}_X$, 
and so we only consider the case $\epsilon > 0$. \\

\noindent \textbf{Step 1: $G = 0$. } The value of $C_X^{[1]}$
follows from the same arguments as for $\epsilon = 0$.
Let $\hat{z}_{X, \epsilon} \in \hat{\discrete{Z}}$ be as in 
\cref{sec:saddle-discrete-continuity-proof-a-gen}. Since $\tilde{z}_{X, 
    \epsilon} - \hat{z}_{X, \epsilon} \in \discrete{Z}$, 
\cref{eqn:interchange-tilde-hat-same-d} gives 
\begin{align*}
    |\hat{z}_{X, \epsilon}|_{a_{\epsilon}}^2
    = a_{\epsilon}(\tilde{z}_{X, \epsilon}, \hat{z}_{X, \epsilon})
    = a_{\epsilon}(u_{X, 
        \epsilon}, \hat{z}_{X, \epsilon}) = |F(\hat{z}_{X, \epsilon})|.
\end{align*} 
Applying \cref{eqn:tilde-z-aeps-coercivity-coercive-all-v} gives
\begin{align*}
    \left( \tilde{\alpha}_X + \frac{1}{\epsilon} \left( 
    \frac{\beta_X}{\Upsilon_X} \right)^2 \right) \|\hat{z}_{X, 
        \epsilon}\|_{V}^2 \leq \|F\|_{\dual{\hat{\discrete{Z}}}} 
    \|\hat{z}_{X, \epsilon}\|_{V} 
    \implies 
    \|\hat{z}_{X, \epsilon}\|_V \leq \frac{\epsilon 
        \Upsilon_X^2}{\epsilon \tilde{\alpha}_X \Upsilon_X^2 + \beta_X^2   
    } 
    \|F\|_{\dual{\hat{\discrete{Z}}}}
\end{align*}
Now, \cref{lem:v-z-tildez-decomp} gives $\|\hat{z}_{X, \epsilon}\|_V \leq 
\Upsilon_X \|\tilde{z}_{X, \epsilon}\|_V$, which gives
the value of $C_{X, \epsilon}^{[2]}$ in the \ref{asmp:a-coercive-all} column
in \cref{tab:saddle-perturb-cont-const}.
Applying \cref{eqn:tilde-z-aeps-coercivity-coercive-all-d} then gives
\begin{align*}
    \left( \frac{\tilde{\alpha}_X}{M_{\operator{D}}^2} + \frac{1}{\epsilon} 
    \right) \|\operator{D} \hat{z}_{X, \epsilon}\|_{Q}^2 &\leq 
    \|F\|_{\dual{\hat{\discrete{Z}}}} 
    \|\hat{z}_{X, \epsilon}\|_{V} 
    \implies 
    \|\operator{D} \hat{z}_{X, \epsilon}\|_{Q}^2 
    &\leq 
    \frac{\epsilon^2 \Upsilon_X^2 M_{\operator{D}}^2  }{ 
        (\epsilon \tilde{\alpha}_X \Upsilon_X^2 + \beta_X^2)
        (M_{\operator{D}}^2 + \epsilon \tilde{\alpha}_X) 
        }
    \|F\|_{\dual{\hat{\discrete{Z}}}}^2.
\end{align*}
Since $p_{X, \epsilon} = \epsilon^{-1} \operator{D} \tilde{z}_{X, \epsilon} 
= \epsilon^{-1} \operator{D} \hat{z}_{X, \epsilon}$, we obtain the value of
$C_{X, \epsilon}^{[3]}$ in the \ref{asmp:a-coercive-all} column
of \cref{tab:saddle-perturb-cont-const}. \\

\noindent \textbf{Step 2: $F = 0$. } From the general case, we have 
$z_{X, \epsilon} \equiv 0$.  As in 
\cref{sec:saddle-discrete-continuity-proof-a-sym-nonneg}, we have
\begin{align*}
    |\tilde{z}_{X, \epsilon}|_a^2 + \epsilon \|p_{X, \epsilon}\|_Q^2 = 
    -G(p_{X, \epsilon}).
\end{align*}
The inf-sup condition \cref{eqn:b-inf-sup-discrete}, 
\cref{eqn:saddle-discrete-penalty-1}, \cref{eqn:a-d-bounded}, and 
\cref{eqn:a-coercive-all} give
\begin{align*}
    \beta_X \|p_{X, \epsilon}\|_Q = \sup_{ v \in \discrete{V} } 
    \frac{(\operator{D} v, p_{X, \epsilon})_Q}{ \|v\|_V } 
    = \sup_{ v \in \discrete{V} } \frac{a(\tilde{z}_{X, \epsilon}, v)}{ 
        \|v\|_V } \leq M_a \|\tilde{z}_{X, \epsilon}\|_V \leq 
    \frac{M_a}{\sqrt{\tilde{\alpha}_X}} |\tilde{z}_{X, \epsilon}|_a. 
\end{align*}
Consequently, we obtain
\begin{align*}
    \left( \frac{\tilde{\alpha}_X \beta_X^2}{M_a^2} + \epsilon \right) 
    \|p_{X, \epsilon}\|_Q^2 \leq \|G\|_{\dual{\discrete{Q}}} \|p_{X, 
        \epsilon}\|_Q \implies \|p_{X, \epsilon}\|_{Q} \leq 
    \frac{M_a^2}{\epsilon M_a^2 + \tilde{\alpha}_X \beta_X^2} 
    \|G\|_{\dual{\discrete{Q}}},
\end{align*}
which corresponds to the value of $C_{X, \epsilon}^{[3]}$ in the 
\ref{asmp:a-coercive-all} column of \cref{tab:saddle-perturb-cont-const}
on noting that $\Upsilon_X = M_a / \tilde{\alpha}_X$.

Returning to \cref{eqn:saddle-discrete-penalty-1}, \cref{eqn:a-d-bounded} 
and \cref{eqn:a-coercive-all} give
\begin{align*}
    |\tilde{z}_{X, \epsilon}|_a^2 \leq M_{\operator{D}} \|p_{X, 
        \epsilon}\|_Q \|\tilde{z}_{X, \epsilon}\|_V \leq 
    \frac{M_{\operator{D}}}{\sqrt{\tilde{\alpha}_X}} \|p_{X, \epsilon}\|_Q 
    |\tilde{z}_{X, \epsilon}|_a \implies |\tilde{z}_{X, \epsilon}|_a^2 
    \leq \frac{M_{\operator{D}}^2 }{ \tilde{\alpha}_X } \|p_{X, 
        \epsilon}\|_Q^2,
\end{align*}
which gives the value of $C_{X, \epsilon}^{[4]}$ in the 
\ref{asmp:a-coercive-all} column of \cref{tab:saddle-perturb-cont-const}.
Consequently, we obtain
\begin{align*}
    \left( 1 + \frac{\epsilon \tilde{\alpha}_X}{M_{\operator{D}}^2} \right) 
    |\tilde{z}_{X, \epsilon}|_a^2 &\leq \|G\|_{\dual{\discrete{Q}}} 
    \|p_{X, \epsilon}\|_{Q} 
    \implies
    \|\tilde{z}_{X, \epsilon}\|_V &\leq 
    \frac{ \Upsilon_X M_{\operator{D}}}{\sqrt{
        (\epsilon M_a \Upsilon_X + \beta_X^2)
        (\epsilon \tilde{\alpha}_X + M_{\operator{D}}^2)} } 
    \|G\|_{\dual{\discrete{Q}}},  
\end{align*}
which corresponds to the value of $C_{X, \epsilon}^{[3]}$ in the
\ref{asmp:a-coercive-all} column of \cref{tab:saddle-perturb-cont-const}.
\hfill \qedsymbol

\subsection{Proof of \cref{thm:saddle-discrete-continuity} when 
		$a(\cdot,\cdot)$ satisfies \ref{asmp:a-sym-coercive-all}}
\label{sec:saddle-discrete-continuity-proof-a-sym-coercive-all}	

For $\epsilon = 0$, the values of $C_{X}^{[1]}$ and 
$C_{X, \epsilon}^{[i]}$, $i=2,\ldots,4$, are the
same as in the case $a(\cdot,\cdot)$ satisfies 
\ref{asmp:a-sym-nonneg-coercive-z}, and so we only consider the case 
$\epsilon > 0$. In this case, the only value not covered by 
\ref{asmp:a-sym-nonneg-coercive-z} or \ref{asmp:a-coercive-all} 
is $C_{X, \epsilon}^{[2]}$. Thus, we assume that $G = 0$.

Applying 
\cref{eqn:tilde-z-aeps-coercivity-sym-coercive-all-v-a} to 
\cref{eqn:proof:sym-tildez-eqn} gives
\begin{align*}
    | \tilde{z}_{X, \epsilon}|_{a} &\leq \frac{\epsilon M_a}{\epsilon 
        M_a + \beta_X^2} \sup_{\tilde{v} \in \tilde{\discrete{Z}}} 
    \frac{|F(\tilde{v})|}{|\tilde{v}|_a}
    \implies 
    \| \tilde{z}_{X, \epsilon}\|_{V} \leq \frac{\epsilon 
        \Upsilon_X^2}{\epsilon M_a + \beta_X^2} 
    \|F\|_{\dual{\tilde{\discrete{Z}}}},
\end{align*}
which corresponds to the
value of $C_{X,\epsilon}^{[2]}$ in the \ref{asmp:a-sym-coercive-all} column 
of \cref{tab:saddle-perturb-cont-const}. 	 \hfill \qedsymbol


\section*{Funding}
PEF was funded by 
the Engineering and Physical Sciences Research Council [grant number EP/W026163/1],
the Science and Technology Facilities Council [grant number UKRI/ST/B000495/1],
the Donatio Universitatis Carolinae Chair ``Mathematical modelling of 
multicomponent systems'', 
the UKRI Digital Research Infrastructure 
Programme through the Science and Technology Facilities Council's Computational 
Science Centre for Research Communities (CoSeC), and
the Swedish Research Council under grant no.~Z2021-06594 while in residence at 
Institut Mittag-Leffler in Djursholm, Sweden.
For the purpose of open access, the authors have applied a CC BY public 
copyright licence to any author accepted manuscript arising from this submission.
No new data were generated or analysed during this study.
MN was supported in part by the NSF, grant no.~DMS-2309425.
CP was supported in part by the NSF, grant no.~DMS-2201487,
the Mary Wheeler Fellowship from the Ridgway Scott Foundation,
and an appointment to the NRC Research
Associateship Program at the U.S. Naval Research Laboratory, 
administered by the Fellowships Office of the National Academies of Sciences, 
Engineering, and Medicine.
\ifarxiv
\else	
Distribution Statement A.  
Approved for public release: distribution is unlimited.
\fi
\bibliographystyle{abbrvnat}
\bibliography{bib.bib}

\end{document}